\documentclass[12pt,leqno,a4paper]{article}
\usepackage{lineno,hyperref}

\usepackage{amsmath}
\numberwithin{equation}{section}
\usepackage{bigints}
\usepackage{amsthm}
\usepackage{amssymb}
\usepackage{amsfonts}
\usepackage{mathtools}
\usepackage{thmtools}
\usepackage{enumerate}
\usepackage{enumitem}
\usepackage{scalerel}
\usepackage{graphicx}
\usepackage{float}
\usepackage{setspace}
\newtheorem{theorem}{Theorem}[section]
\newtheorem{corollary}[theorem]{Corollary}
\newtheorem{lemma}[theorem]{Lemma}
\newtheorem{proposition}[theorem]{Proposition}
\declaretheoremstyle[bodyfont=\normalfont]{normalbody}
\declaretheorem[style=normalbody,name=Example,numberwithin=section]{example}

\def\mTa{\mathcal{T}}
\def\bmTa{\overline{\mathcal{T}}}
\def\mU{\mathcal{U}}
\def\del0{\delta_{0}}
\def\tzau{\tau_{0}}
\def\toau{\tau_{1}}
\def\ttwo{\tau_{2}}

\def\taujk{\tau_{j_{k}}}
\def\alphai{\alpha_{i}}
\def\epszerosquare{\varepsilon_{0}^{2}}
\def\epszero{\varepsilon_{0}}
\def\Czero{C_{0}}
\def\Czerosquare{C_{0}^{2}}
\def\Gammaj{\Gamma^{j}}
\def\Gammajk{\Gamma^{j_{k}}}

\def\Gammai{\Gamma^{i}}

\def\Gi{\mathcal{G}^{i}}
\def\Gj{\mathcal{G}^{j}}
\def\parGi{\partial\overline{\mathcal{G}^{i}}}

\def\xnot{x_{0}}
\def\tnot{t_{0}}
\def\loc{\rm{loc}}
\def\Vm{V_{-}}
\def\Vp{V_{+}}
\def\qm{q_{-}}

\def\teps{\tau_{\epsilon}}
\def\RRn{\mathbb{R}^{n}}

\def\Tb{\bar{T}}
\def\Fcal{\mathcal{F}}
\def\Qcal{\mathcal{Q}}
\def\dist{\text{dist}}
\def\Unp{\text{Unp}}
\def\reach{\text{reach}}
\def\meas{\text{meas}}
\def\overGi{\overline{\mathcal{G}^{i}}}
\def\overGj{\overline{\mathcal{G}^{j}}}
\def\Lip{\text{Lip}}
\def\Exp{\text{Exp}}
\def\vol{\text{vol}}

\newcommand{\subjclass}[1]
{
  \small	
  \textbf{\textit{MSC[2020]:}} #1
}

\newcommand{\keywords}[1]
{
  \small	
  \textbf{\textit{Keywords---}} #1
}

\newcommand{\quotes}[1]{``#1''}

\begin{document}

\title{Solutions of the Second Order Schr{\"o}dinger Wave Equations Near Static Black Holes and Strong Singularities of the Potentials}

\author{
{\bf Igor M. Oliynyk}
\\
E-mail: igor.m.oliynyk@gmail.com
}

\date{}

\maketitle

\begin{abstract}
  We consider a linear Schr{\"o}dinger operator $H = -\Delta + V$ with a strongly singular potential $V$ not bounded
  from below on a non-compact incomplete Riemannian manifold $M$.
  We assume that the negative part of potential $\Vm$ is measurable, and that it does not necessarily belong to
  either local Kato or Stummel classes, and we define new geometric conditions on the growth of $\Vm$
  in a special {\it range control neighborhood (RCN)} such that $H$
  is semibounded from below on functions compactly supported in these neighborhoods.
  We define RCN by means of {\it an inner time metric} that estimates the minimal time for a classical
  particle to travel between any two points on $M$,
  and we assume that $M$ is complete w.r.t.~this metric, i.e. the potential $V$ is classically complete on $M$.
  For the corresponding Cauchy problem of the wave equation $u_{tt} + Hu = 0$, we define locally a Lorentzian
  metric such that its light cone is formed along the minimizing curves with respect
  to the inner time metric, where both an energy inequality and uniqueness of solutions hold.
  Conversely, for well-known Lorentzian metrics of static black holes
  - Schwarzschild, Reissner\text{--}Nordstr{\"o}m, and de Sitter metrics - we study the wave equations
  for the corresponding Schr{\"o}dinger operators,
  and we show that the event horizons of these black holes belong to the RCNs of infinity with respect to
  the inner time metrics, and that all solutions of the mixed problems remain in these neighborhoods indefinitely long.
\end{abstract}



\section{Introduction}
Let $M$ be a $C^{\infty}$ noncompact connected, possibly incomplete,
oriented Riemannian manifold without a boundary with $\dim(M)=n, n\geq 1$.

We denote by $dl$ and $d\mu$ its standard Riemannian metric and corresponding volume measure
and by $L^{2}(M)$ the Hilbert spaces of real-valued square integrable functions with the norm
\begin{equation*}
  \left( f, f\right) = \|f\|^{2} = \int_{M} |f|^{2}d\mu.
\end{equation*}
Throughout this paper, if not noted otherwise, all functions are assumed to be real valued.

Define the Schr{\"o}dinger operator with the potential $V$
\begin{equation}\label{eqn:schrodinger}
  Hu = -\Delta u + Vu, u \in L^{2}(M),
\end{equation}
where
\begin{equation*}
  \Delta u = \rm{div}\nabla u
\end{equation*}
is the Laplace\text{--}Beltrami operator, and in local coordinates $x_{i}, i=1, \dots, n,$
the gradient vector field is defined by
\begin{equation*}
  \nabla^{i} = g^{ij}(x)\partial_{j}
\end{equation*}
with $\partial_{j}\vcentcolon= \frac{\partial}{\partial x_{j}}$, 
and 
\begin{equation*}
  \rm{div}X = \frac{1}{\sqrt{\det(g)}} \partial_{j}\left(\sqrt{\det(g)}X^{j}\right)
\end{equation*}
is the divergence of the vector field $X$. Here $g^{ij} = (g_{ij})^{-1}$,
and $\det(g)$ is the determinant of the metric matrix $g$.
We have used a conventional notation for the summation over repeated indices.

Let us turn to the real-valued potential function $V$. We assume that $V$ is measurable on $M$, and that
\begin{equation*}
  V = \Vp - \Vm,\ \Vp,\ \Vm \geq 0, \Vp\in L^1_{\loc}(M).
\end{equation*}
Therefore, we assume that the negative part $\Vm$ is only measurable, and
we make further assumptions on its growth at infinity and on the behavior of its singularities. 

We define an inner time metric with a minorant function $\qm>0$, $\Vm \leq \qm$, by 
\begin{equation}\label{eqn:metric}
  \tau(p_1, p_2) = \inf_{\gamma}\int_{\gamma}\qm^{-1/2}(x)dl, 
\end{equation}
where the infimum is taken over all piecewise smooth curves $\gamma$ connecting $p_1$ and $p_2\in M$.

A motivation for this study comes from an earlier author's work -
see~\cite{Oleinik1993},~\cite{Oleinik1994}, and~\cite{OleinikProceedings} -  
on the essential self-adjointness of
the Schr{\"o}dinger operator~\eqref{eqn:schrodinger} in $L^{2}(M)$
with a regular potential $V\in  L^{\infty}_{\loc}(M)$.
In~\cite{Oleinik1993} we noticed, in particular,
that defined in~\cite{Rof70} sufficient conditions on the essential
self-adjointness of~\eqref{eqn:schrodinger} in $L^{2}(\RRn)$ imply
{\it the classical completeness} of the potential, i.e. impossibility for a classical particle moving
in a potential field with the potential $V$ to reach infinity in a finite time.
In other words, we showed that $M$ is complete w.r.t.~the metric~\eqref{eqn:metric}
for some $\qm$.

Conversely, ~\cite{Oleinik1994} showed
that the classical completeness of $V$ is one of the sufficient conditions
for the essential self-adjointness of~\eqref{eqn:schrodinger} on any noncompact Riemannian manifold.

The paper~\cite{OleinikProceedings} extends and generalizes the results
of~\cite{DEVINATZ197758, HELLWIG1969279, Oleinik1994, Rofe-Beketov1985} to second-order elliptic operators of the divergent type
$$Hu = -\sum_{i, j = 1}^{n}\partial_{i}a^{ij}(x)\partial_{j}u + Vu, u \in L^{2}(\RRn)$$
with a positive definite matrix $a^{ij}(x), i, j = 1,\dots, n$ for all $x\in \RRn$.

We require that the potential $\Vm$ be classically complete on
the Riemannian manifold $M = \RRn$ with
the metric $g = (a^{ij})^{-1}$, i.e. the matrix inverse to the matrix $a^{ij}$.

The survey in~\cite{Braverman_2002} extends the results of~\cite{Oleinik1993} and~\cite{Oleinik1994}
to Schr{\"o}dinger-type operators (with a singular electric 
potential) that act on sections of Hermitian vector bundles over manifolds.
The paper~\cite{Braverman_2002} contains an extensive bibliography and a good overview (see appendix D)
of the subject of the essential self-adjointness of the Schr\"odinger operators on $\RRn$ and manifolds.
For more recent works on this topic,
see~\cite{Bruse2004, GRUMMT2012480, Gneysu2015HeatKI, Mila2023, NENCIU2008, NENCIU20172619, Prandi2016QuantumCO}.

The classical completeness of $V$ implies that, in particular, the light cone of the metric
$$
d\ell^{2} = -q dt^{2} + dl^{2}
$$
at any point $(0, p)$ of the space-time $[0, \infty)\times M$ has the form
$t = \tau(p, x), x\in M$, and this cone can be extended to infinity
in both time and spatial directions, where the minimizing curves of metric~\eqref{eqn:metric}
starting at $p$ can be uniquely extended to infinity.

In the present paper, our focus is on the global finite propagation speed
of the solutions of the Cauchy problem for the wave equation
\begin{equation}\label{eqn:wave-eq-1}
  u_{tt} + Hu(t, x) = \rho(t, x), t \in [0, T],
\end{equation}
where $u_{tt}:=\frac{\partial^{2}u}{\partial t^{2}}$,
and the source function $\rho\in L_{\loc}^{2}([0, T] \times M)$.
We study dependency cones for~\eqref{eqn:wave-eq-1}, their relation with the light cones above,
and explain how these cones are defined for strongly singular potentials $V$
with $\Vp\in L^{1}_{\loc}(M)$ and its negative part $\Vm$ just being measurable.

Very often, when we study the propagation speed of solutions of~\eqref{eqn:wave-eq-1} or the relative boundedness of $\Vm$ with respect to the operator $-\Delta$ in the sense of
operator norms or operator forms, we assume that $\Vm\in S_{n,\loc}(\RRn)$ or
$\Vm\in K_{n,\loc}(\RRn)$, local Stummel or Kato classes;
please refer to the Appendix at the end of this paper for definitions and
important properties of these classes.

The Kato conditions guarantee that $\Vm$ is form-bounded
w.r.t.~the Laplacian operator $\Delta$
for any relative bound $\delta \in (0, 1)$, namely, this inequality holds for any $\phi\in C_{0}^{\infty}(\RRn)$
and some constant $C = C(\delta)\geq 0$
\begin{equation}\label{eq:relative-bound}
  \int_{\RRn}\Vm \phi^{2}dx \leq \delta \int_{\RRn}|\nabla \phi|^{2}dx + C \int_{\RRn}\phi^{2}dx. 
\end{equation}
The condition~\eqref{eq:relative-bound}, together with a sufficient regularity of $\Vp$, for instance, $\Vp\in L_{\loc}^{2}(\RRn)$,
leads to nonnegativity of $H$ in the sense of forms and its essential self-adjointness in $L^{2}(\RRn)$.

The nonnegativity of $H$ with $\Vm \in K_{n,\loc}(M)$ implies an energy estimate for the solutions of~\eqref{eqn:wave-eq-1}
in specially defined dependency cones where the Cauchy problem is correctly posed.

In Section~\ref{chap:range-control}, we impose growth conditions on $\Vm$ 
along the minimizing curves (w.r.t.~\eqref{eqn:metric}) starting
at a connected compact submanifold of $M$; if these conditions are satisfied
in some neighborhood of this submanifold,
then we call it a {\it range control neighborhood (RCN)},
where the operator~\eqref{eqn:schrodinger} is nonnegative for any smooth
function compactly supported in that RCN.
The introduction of RCNs is one of the main features
of this paper, and the essential self-adjointness of $H$
and the global finite propagation speed of solutions of~\eqref{eqn:wave-eq-1}
are just corollaries of a special open cover of $M$ with RCNs.

The RCN conditions are easily interpretable,
and they are weaker (at least in $\RRn$), as it has been noted in Example~\ref{ex:not_kato_class}, than the conditions formulated for the Kato and Stummel classes of the potential.
These conditions also lead to a direct estimate of the relative bound $\delta$ in~\eqref{eq:relative-bound},
but with $C = 0$, and $\delta$ can be of any value in the interval $(0, 1)$.
Two other important properties of RCNs are as follows:
\begin{enumerate}
\item The singularity points of $\Vm$ cannot belong
  to RCNs of other centers, so the singularity points of $\Vm$ must be centers of their own RCNs;
  in particular, singularity points may belong to connected submanifolds, which are centers of corresponding RCNs;
\item We define RCNs of infinity w.r.t.~the metric~\eqref{eqn:metric}; in Section~\ref{chap:examples}
  we present examples of RCNs of infinity for the static black holes - Schwarzschild, Reissner-Nordstr{\"o}m, and de Sitter spaces.
\end{enumerate}

The local nonnegativity property of the Schr{\"o}dinger operator
allows us to consider a mixed problem for the related wave equation~\eqref{eqn:wave-eq-1},
and we derive an energy inequality for its solutions
at a domain of the dependency cone contained in a past light cone of the above metric $d\ell^{2}$ defined on the RCN;
thus, we prove the uniqueness and existence of solutions of a corresponding Cauchy\text{--}Dirichlet problem (see Sections~\ref{chap:uniqueness} and~\ref{chap:existence}).

Another interesting fact is that the relative bound value $\delta \in (0, 1)$ is related to the slope
of the corresponding dependency cone; the smaller the value is the steeper the slope is,
and, vice versa, shallow cones correspond to $\delta$ values close to one.

In Section~\ref{chap:gfps}, we study the global propagation speed of solutions of
the wave equation on the entire $M$ and the essential self-adjointness
of the operator~\eqref{eqn:schrodinger}. 
We assume here that there exists an open cover of $M$ consisting of RCNs
and additional conditions ensuring semiboundedness of the Schr{\"o}dinger operator.
These assumptions lead to a proof
of the global finite propagation speed of solutions of the Cauchy problem on $M$.
The essential self-adjointness of the Schr{\"o}dinger operator follows from
the Berezansky theorem - see Theorem 6.2 in~\cite{berezansky1978self}, which defines
the method of hyperbolic equations to answer the essential self-adjointness question for any symmetric operator in a Hilbert space. This section provides a more detailed bibliography of this method.

In Section~\ref{chap:examples}, we study the domain of dependency cones
defined in Section~\ref{chap:uniqueness} and their relationship with the metric $d\ell^{2}$.

This metric implies \quotes{unbounded speed of light} near singularities of $\Vm$;
for the case of the hydrogen atom, we define a Lorentzian metric for the corresponding wave equation,
and we derive a time\text{--}energy uncertainty condition from a special form of the light cones of this metric.

Conversely, for known static black hole Lorentzian metrics - Schwarzschild, Nordstr{\"o}m\text{--}Reissner, and de Sitter -
we define wave equations for the corresponding Schr{\"o}dinger operators. We determine that the event horizons
of these metrics are infinities in metric~\eqref{eqn:metric}, we can choose their neighborhoods to be
RCNs of infinities, and the results of Section~\ref{chap:uniqueness} imply that the solutions of the wave equations
in the neighborhoods of corresponding event horizons never reach their boundaries. 

Our focus is on the qualitative behavior of the solutions of the wave equation~\eqref{eqn:wave-eq-1},
so we avoid other generalizations by letting the manifold $M$ be infinitely smooth,
by considering a simple Laplacian instead of a second-order symmetric operator,
and by defining the Schr{\"o}dinger operator
on real-valued functions instead of sections of Hermitian vector bundles, etc. 

\section{Basic Notations, Main Assumptions and Conditions}
We define the space of Lipschitz functions $f \in \Lip^{0,\alpha}_{\loc}(M)$, if for any compact and measurable $\Fcal\subset M$
\begin{equation*}
  |f(x) - f(y)| \leq C(\Fcal) \dist(x, y)^{\alpha}\ \text{for some }\alpha > 0\ \text{and any }x, y \in \Fcal.
\end{equation*}
Sometimes, we use local coordinates $x_{i}, i = 1, 2,\ ...,\ n$ on $M$ with the Riemannian
metric $g_{ij}(x) = \left\langle\frac{\partial}{\partial x_{i}}, \frac{\partial}{\partial x_{j}}\right\rangle$ for
the coordinate vectors $\frac{\partial}{\partial x_{i}}, \frac{\partial}{\partial x_{j}} \in T_{x}M$,
the tangent space at $x \in M$.

$W_{\loc}^{1,2}(M)$ denotes the Sobolev space
of locally square integrable functions and their first derivatives
with the norm 
\begin{equation*}
  ||f||_{1,2}(U) = \left(\int_{U}\left(|\nabla f|^{2} + |f|^{2}\right)d\mu\right)^{1/2}
\end{equation*}
for each open set $U\Subset M$,
where $|\nabla f(x)| \vcentcolon= \sqrt{\left\langle\nabla f(x), \nabla f(x)\right\rangle}$.
We also denote by $W^{1,2}_{0}(U)$ the Sobolev space of functions with compact support on $U$.
If not noted otherwise, we use a zero subscript to denote compactly supported functions
of the corresponding space.

We use the notation $\dist_g$ for the usual distance on $M$ and $\dist_\tau$
for the distance due to the metric~\eqref{eqn:metric}. Note that $M$ may not be complete
w.r.t.~the original metric $\dist_g$.

Denote by $B_g(x, r) = \{y \in M: \dist_g(y, x) < r\}$ an open ball in metric $g$
with the center at $x\in M$ and radius $r > 0$.

\noindent Denote a closed ball about $p\in M$ with the radius $\tzau > 0$
w.r.t.~the metric~\eqref{eqn:metric}
by
\begin{equation}\label{eqn:tau-point}
  \begin{aligned}
    &\mTa_{p,\tzau} = \{x \in M: \tau(p, x) < \tzau\},\\
    &\text{its closure}\ \bmTa_{p,\tzau} = \{x \in M: \tau(p, x) \leq \tzau\},\\
    & \text{and its boundary}\ \partial \bmTa_{p,\tzau} = \{x \in M: \tau(p, x) = \tzau\}.
  \end{aligned}
\end{equation}

We impose the following condition on the classical completeness of the potential.

\noindent \textbf{Condition A. Classical Completeness of the Potential.}

\noindent The manifold $M$ is geodesically complete w.r.t.~\eqref{eqn:metric}, i.e.,
according to the Hopf\text{--}Rinow theorem, it means that either of two equivalent conditions
\begin{enumerate}[label=\textbf{A.\arabic*}, ref=A.\arabic*]\label{eqn:classical-completeness}
\item\label{eqn:classical-completeness:a} The metric space $(M, \dist_\tau)$ is complete, or
\item\label{eqn:classical-completeness:b} Any closed and bounded set w.r.t.~$\dist_\tau$  on $M$ is compact
\end{enumerate}
hold.

This condition of the classical completeness of the potential at infinity
for the self-adjointness of the Schr{\"o}dinger operator
was formulated in~\cite{Oleinik1994}, and it returns to the original
E.~C.~Titchmarsch~\cite{titchmarsh_1949} and D.~B.~Sears~\cite{sears_1950}
conditions for spherically symmetric potentials. Going forward
we will always use the term \textit{infinity} w.r.t.~the metric~\eqref{eqn:metric}.

The next condition defines the regularity of $\qm$ on $M$.

\noindent \textbf{Condition B. Regularity Conditions for Potential.}

\noindent Define the set $\Qcal_{-} \vcentcolon= \{\xnot\in M: \limsup_{x\to\xnot}\qm(x)=\infty\}$;
we require that $\meas(\Qcal_{-}) = 0$.

\begin{enumerate}[label=\textbf{B.\arabic*}, ref=B.\arabic*]
\item\label{prop:q-minus:a} We require that $\qm(x) \to \infty$ when $x\to\Qcal_{-}$,
  so that we can set $\qm^{-1/2}(x) = 0$ for $x\in\Qcal_{-}$; hence, $\qm^{-1/2}$ is continuous on $\Qcal_{-}$,
  and $\Qcal_{-}$ is a closed set.
\item\label{prop:q-minus:b} Function $\qm^{-1/2}\in \Lip_{\loc}^{0, 1}(M\setminus \Qcal_{-}) \cap W_{\loc}^{1, 2}(M)$.
\end{enumerate}

Note that the relation~\ref{prop:q-minus:b} is a standard
regularity condition in the Titchmarsh\text{--}Sears theorem
at infinity, where $\Vm$ is assumed to be locally bounded, i.e., when $\Qcal_{-} = \emptyset$.

We want minimizing curves w.r.t.~the metric~\eqref{eqn:metric} to be regular and unique in 
any small neighborhood of any point $x\in M\setminus\Qcal_{-}$. Under condition~\ref{prop:q-minus:b},
the metric~\eqref{eqn:metric} is Lipschitz in a small open neighborhood $U(x)\subset M\setminus\Qcal_{-}$, and,
as noted in~\cite{S_mann_2018} (see also an overview in~\cite{lytchak_yaman}),
all the minimizing curves starting at $x$ are of regularity $\Lip^{1,1}(U)$,
i.e., the continuously differentiable curves whose first derivatives belong to 
$\Lip^{0,1}(U)$.

\textbf{Submanifolds of Positive Reach}

For a connected compact submanifold $\Fcal\subset M$, define a set
$\Unp_{M}(\Fcal) =\{x\in M: \text{ exists and is unique to } a\in\Fcal \text{ such that }\dist_{\tau}(x, \Fcal) = \tau(x, a)\}$.
Thus, we can define a unique map
\begin{equation}\label{eqn:xi-map}
\xi_{\Fcal, M}\vcentcolon\Unp_{M}(\Fcal) \to \Fcal
\end{equation}
satisfying $\xi_{\Fcal, M}(x) = a$ with $\tau(\Fcal, x) = \tau(a, x)$.

Furthermore, for each $a\in\Fcal$ we define
$\reach_{M}(a, \Fcal)\vcentcolon= \sup\{\tzau:\{x\in M: \tau(x, a) < \tzau\}\subset \Unp_{M}(\Fcal)\}$,
and, finally, 
\begin{equation}\label{eqn:reach-set}
  \reach_{M}(\Fcal) = \inf_{a}\reach_{M}(a, \Fcal).
\end{equation}
Therefore, $\Fcal$ is of a \textit{positive reach} if $\reach_{M}(\Fcal) > 0$.

The sets of positive reach were defined by H.~Federer in~\cite{Federer_1959}, \S 4
for $\RRn$. V.~Bangert in~\cite{Bangert1982} proved that on a smooth complete Riemannian manifold,
the condition for a set to be of positive reach largely depends on the atlas; it is local in nature,
and it does not depend on the metric. A.~Lytchak~\cite{lytchak2023note} established necessary
regularity conditions for a connected closed submanifold $\Fcal\subset M,\ \text{dim}(\Fcal) = m < n$
to be a set of positive reach; very informally, these are $\Lip^{1,1}$-submanifolds
whose boundary points $x\in\partial \Fcal$ have the tangent space $T_{x}\Fcal$ isomorphic
to a half-space in $\mathbb{R}^{m}$.

Throughout this paper, we always assume and consider submanifolds $\Fcal$ that have a positive reach.

The definition~\eqref{eqn:reach-set} implies, in particular, that if $\reach_{M}(\Fcal) = \tzau$,
then for any $x\in M: \tau(x, \Fcal) < \tzau$, there exists a unique minimizing curve
connecting $x$ and $\xi_{\Fcal, M}(x)$ with $\tau(x, \xi_{\Fcal, M}(x)) < \tzau$.

\noindent \textbf{Neighborhood of Infinity}

\noindent An open domain $\Omega\subset M$ with a closed (in the topological sense) boundary
$\partial\Omega\subset M, \dim \partial\Omega = n - 1$, is
the \textit{neighborhood of infinity} if
\begin{equation}\label{dfn:neign-infinity}
  \reach_{\Omega}(\partial\Omega) = \infty.
\end{equation}
In definition~\eqref{dfn:neign-infinity}, we stress
that definitions~\eqref{eqn:xi-map} and~\eqref{eqn:reach-set} are restricted to the domain $\Omega$, so
that for each $p\in\partial \overline{\Omega}$,
there is a unique minimizing curve $\gamma_{p}\subset \overline{\Omega}$,
the closure of $\Omega$ in~\eqref{eqn:metric},
which can be extended to infinity in $\Omega$.

\noindent \textbf{Infinity Covers}

We assume that there exists a countable set of neighborhoods of infinity $\Gi$ such that 
\begin{equation*}
  \begin{aligned}
    & \Gi \cap \Gj = \emptyset \text{ for } i \ne j,\\
    & \Gi \subset B_g(x_{i}, d_{i}) \text{ for some } x_{i}\text{ and } d_{i} > 0.
  \end{aligned}
\end{equation*}

\noindent For each $\Gi$, we define these sets
\begin{equation*}
  \begin{aligned}
    & \Gi_{\tzau} = \{x \in M: 0 < \tau(\partial \overGi, x) < \tzau\},\\
    & \parGi_{\tzau} = \partial \overGi \cup \{x \in \Gi: \tau(\partial \overGi, x) = \tzau\}, \text{ and}\\
    & \overGi_{[\tau_1, \tau_2]} = \{x \in \Gi: \tau_1 \leq \tau(\partial \overGi, x) \leq \tau_2\},
  \end{aligned}
\end{equation*}
where $\overGi$ denotes the closure of $\Gi$ w.r.t.~\eqref{eqn:metric}.

The function $\qm^{-1/2}$ is continuous on $\Qcal_{-}$,
and we assume that $\Qcal_{-}$ can be represented as a countable union of connected submanifolds
\begin{equation*}
  \Qcal_{-} = \cup_{j}\Gammaj,\ \dim\left(\Gammaj\right) < \dim(M), j = 1,\ 2,\ ... 
\end{equation*}
where we assume that $\Gammaj$ possesses a positive reach w.r.t.~\eqref{eqn:metric}.

\noindent Assume further that there exists $\bar{\tau} > 0$ with
\begin{equation*}
  \begin{aligned}
    & \inf_{i \ne j} \dist_\tau(\Gammai, \Gammaj) \geq \bar{\tau},\\
    & \inf_{i \ne j} \dist_\tau(\partial \overGi, \partial \overGj) \geq \bar{\tau},\text{ and}\\
    & \inf_{i,j} \dist_\tau(\partial \overGi, \Gammaj) \geq \bar{\tau}.
  \end{aligned}
\end{equation*}

If we assume that the neighborhoods $\mTa_{p,\tzau}\subset \Unp_{M}(p)$
or $\mTa_{\Gammaj,\tzau}\subset\Unp_{M}(\Gammaj)$ for some $\tzau > 0$, then
we can define the Riemannian metrics in these neighborhoods and in $\Gi$ by 
\begin{equation}\label{eqn:riemann-tau}
  dl^{2} = \qm(x)d\tau^{2} + d\omega^{2}, \\
  x \in \mTa_{p,\tzau}\setminus\{p\}\left(\mTa_{\Gammaj,\tzau}\setminus \Gammaj\ \text{or } \Gi\right),
\end{equation}
where $d\omega^{2}$ is the metric induced on submanifold $\tau(p,x) = \widetilde{\tau},\ \text{for a.e., }\widetilde{\tau} \leq \tzau\\
(\tau(\parGi,x) = \widetilde{\tau}\ \text{or }
\tau(\Gammaj,x) = \widetilde{\tau}\ \text{for a.e }\widetilde{\tau} \leq \tzau)$.
Note that the regularity of minimizing curves and the fact that the
function $\tau(p, \cdot) \in \Lip^{1,1}(\mTa_{p,\tzau})$
ensures the local regularity of such submanifolds,
and the form \eqref{eqn:riemann-tau} of the Riemannian metric is well defined.

Let us consider the volume density of the induced metric $d\omega^{2}$.
For any fixed $0 < \toau \leq \tzau$ and any compact connected submanifold $C \subset M$
of positive reach with $\bmTa_{C,\tzau}\vcentcolon=\{x\in M: \tau(x, C)\leq \tzau\}\subset \Unp_{M}(C)$ we define a map
\begin{equation}\label{eqn:contraction-map}
  \begin{aligned}
    & \zeta\vcentcolon\partial \bmTa_{C,\tzau} \to \partial \bmTa_{C,\toau}, \\
    & \zeta(x) = y \text{ iff } \xi_{C}(x) = \xi_{C}(y).
  \end{aligned}
\end{equation}
Note that we may assume that both submanifolds $\bmTa_{C,\tzau}$ and $\partial \bmTa_{C,\toau}$ are of class $\Lip^{1, 1}$.

We prove the following
\begin{lemma}\label{lemma:regularity-zeta}
  The map $\zeta\in \Lip^{0, 1}$ and uniform on $\partial\bmTa_{C,\tzau}$.
\end{lemma}
\begin{proof}
  A similar statement can be found in Theorem~4.5(8) of~\cite{Federer_1959}; it was shown
  that in the Euclidean case with $M=\RRn$ for any $a$ and $b\in \Unp_{\RRn}(C)$ we have this inequality
  \begin{equation}\label{eqn:xi-lipschitz}
    |\xi_{C}(a) - \xi_{C}(b)| < K |a - b|
  \end{equation}
  with some constant $K = K(\reach_{M}(C), \max(\dist(a, C), \dist(b, C))$, and this constant is uniform as long as
  $\reach_{M}(C) > \max(\dist(a, C), \dist(b, C))$. We have to adopt the proof of this theorem for any submanifold $C\subset M$
  and metric~\eqref{eqn:metric}.

  Select any $a\in \partial\bmTa_{C,\tzau}$; note that $\xi_{C}(\xi_{\partial\bmTa_{C,\toau}}(a)) = \xi_{C}(a)$, so let us denote
  $a_{1} = \xi_{\partial\bmTa_{C,\toau}}(a)$, and we define normal coordinates at $a_{1}$ with an orthonormal basis
  at $T_{a_1}(M) = T_{a_1}(\partial\bmTa_{C,\toau}) \oplus \{\lambda v_{a}\}, |v_{a}|= 1, v_{a}\perp T_{a_1}(\partial\bmTa_{C,\toau}),
  \lambda \in \mathbb{R}$ such that
  $\Exp(\tau(a_{1}, a)v_{a}) = a$ for the corresponding exponential map $\Exp\ \vcentcolon \mathcal{E} \cap T_{a_1}(M) \to M$
  defined on some open neighborhood $\mathcal{E}$ of the origin in $T_{a_1}(M)$. For convenience, we identify
  the origin $\{0\}\in \mathcal{E}$ with the point $a_{1}$ in the neighborhood of the affine space $\mathcal{E}$,
  so that the entire interval $[a_1, a_1 + \tau(a_{1}, a)v_{a}]\subset \mathcal{E}$.

  Similarly, we select $b\in \partial\bmTa_{C,\tzau}$ such that $b_{1} = \xi_{\partial\bmTa_{C,\toau}}(b)$ and the corresponding
  normal coordinates at $b_{1}$ with an orthonormal basis at
  $T_{b_1}(M) = T_{b_1}(\partial\bmTa_{C,\toau}) \oplus \{\lambda v_{b}\}, |v_{b}|= 1, v_{b}\perp T_{b_1}(\partial\bmTa_{C,\toau}),
  \lambda \in \mathbb{R}$ and with the
  corresponding exponential map $\Exp(\tau(b_{1}, b)v_{b}) = b$.

  Note here that $\tau(b_{1}, b) = \tau(a_{1}, a)$, and that the vector $\tau(b_{1}, b)v_{b}$ can be obtained as
  the parallel transport of $\tau(a_{1}, a)v_{a}$ in the Riemannian connection on $M$
  with respect to the conformal metric~\eqref{eqn:metric} along a geodesic curve on $\partial\bmTa_{C,\toau}$
  connecting $a_1$ and $b_1$. 

  We can select $b$ so close to $a$ that the interval $[b_1, b_1 + \tau(b_{1}, b)v_{b}] \subset \mathcal{E}$,
  and, given the fact that vectors $\tau(a_{1}, a)v_{a}$ and $\tau(b_{1}, b)v_{b}$ are normal to
  $T_{a_1}(\partial\bmTa_{C,\toau})$ and $T_{b_1}(\partial\bmTa_{C,\toau})$ respectively, then
  $\tau(a_{1}, a) = \tau(b_{1}, b)$ are the shortest Euclidean distances on $\mathcal{E}$ between $a$ and $a_1$ and
  $b$ and $b_1$.

  Therefore, the inequality 
  \begin{equation*}
    |\xi_{\partial\bmTa_{C,\toau}}(a) - \xi_{\partial\bmTa_{C,\toau}}(b)| < K |a - b|, 
  \end{equation*}
  similar to~\eqref{eqn:xi-lipschitz}, holds on $\mathcal{E}$, and, given a uniform equivalence of both Euclidean metric and metric $g$ on the compact domain $\bmTa_{C,\tzau}$, the statement of the lemma follows.
\end{proof}

Given Lemma~\ref{lemma:regularity-zeta}, we can define the corresponding
pullback map $\zeta^{*}\vcentcolon \Lip^{0,1}(\partial \bmTa_{C,\toau}) \to \Lip^{0,1}(\partial \bmTa_{C,\tzau})$, and
the density measure at $y\in \partial \bmTa_{C,\toau}$ can be correctly defined by
\begin{equation}\label{eqn:sigma-squared}
  \sigma^{2}(y) \vcentcolon= \frac{\left.\text{dVol}\right|_{\zeta^{*}g_{\toau}}}{\left.\text{dVol}\right|_{g_{\tzau}}},
\end{equation}
where $g_{\tzau}$ and $g_{\toau}$ are the metrics on $\partial \bmTa_{C,\tzau}$ and $\partial \bmTa_{C,\toau}$
induced from $g$ on $M$, and the fraction on the right-hand side of~\eqref{eqn:sigma-squared} is
the Radon\text{--}Nikodym derivative of two Riemannian measures.
For a technical convenience, we used the square power in the above expression.
For example, for a spherically symmetric potential $\qm(r)$
in the neighborhood of the origin, we have $\sigma = c_{n}r^{\frac{n-1}{2}}$ in polar coordinates on $\RRn$.

Note also that the definitions~\eqref{eqn:contraction-map} and~\eqref{eqn:sigma-squared} are valid for
the case when $\dim(C) = n - 1$ and $\toau = 0$. If $\dim(C) < n - 1$, then we may set $\sigma = 0$ on $C$.

The volume element in $\mTa_{C,\tzau}$ can be defined as
\begin{equation}\label{eqn:volume-element}
  d\mu  = \qm^{1/2}(y)\sigma^{2}(y)d\tau \left.\text{dVol}\right|_{g_{\tzau}}.
\end{equation}
Note that in this formula the volume measure $\left.\text{dVol}\right|_{g_{\tzau}}$ is independent of $y$.

\noindent \textbf{Condition C. Admissible Open Covers on $M$.}
There exists $\tzau > 0$ such that an open cover on $M$ is defined by these components
and properties below
\begin{enumerate}[label=\textbf{C.\arabic*}, ref=C.\arabic*]
\item\label{cond_c_1} All open neighborhoods $\mTa_{\Gammaj, \tau_{j}}\subset\Unp_{M}(\Gammaj), j > 0$
  are with $\reach_{M}(\Gammaj) = \tau_{j}\geq\tzau$.
\item\label{cond_c_2} For neighborhoods of infinity $\overGi, i > 0$ all boundaries $\parGi$
  have a positive $\reach_{M\setminus \Gi}(\parGi) = \tau_{i}\geq\tzau$, and,
  due to the definition~\eqref{dfn:neign-infinity}, $\reach_{\Gi}(\parGi) = \infty$.
\item\label{cond_c_3} The cover is locally finite, i.e., any compact subset $K\subset M$
  can be covered by a finite number of neighborhoods of infinity $\Gi$ and of the neighborhoods
  $\mTa_{k_{\alphai},\tau_{\alphai}}, i = 1, 2,\dots, l$ of closed and connected submanifolds,
  $\dim(k_{\alphai}) < n$, and with $\reach_{M}(k_{\alphai}) = \tau_{\alphai}\geq \tzau$.
  Every point $x\in M$ can be covered with no more than $m$ open covers, i.e., $M$ has a finite {\it covering
  dimension}.
\item\label{cond_c_4} Minimal Cover Intersection. There exists $0 < \teps < \tzau$ such that for any $x\in M$ there exists an element of cover $\mTa_{k_{\alphai}, \tau_{\alphai}}$ such that the distance
  $\tau\left(x, \partial \mTa_{k_{\alphai}, \tau_{\alphai}}\right) \geq \teps$.
  This distance is taken along the unique segment connecting $k_{\alphai},\ x$ and $\partial \mTa_{k_{\alphai}, \tau_{\alphai}}$.
\end{enumerate}

Condition C uses sets of positive reach w.r.t.~the metric~\eqref{eqn:metric} for all open covers.

The conditions~\ref{cond_c_1} and~\ref{cond_c_2} imply that both~$\Gammaj$
and~$\parGi$, by definition, are at the centers of their corresponding positive
reach neighborhoods.

The condition~\ref{cond_c_3} leaves an option to have a variety of centers $k_{\alphai}$, not necessary points, for these covers. 

The property in~\ref{cond_c_4} will be used to build a special partition of unity subordinate to this cover.

We consider examples of sufficiently regular spherically symmetric potentials, so that conditions B and C can be easily verified.

From definition~\eqref{eqn:riemann-tau}, we derive the expression for the square gradient norm
\begin{equation}\label{eqn:gradfunc}
  \left|\nabla f\right|^{2} = \qm^{-1}\left(\frac{\partial f}{\partial \tau}\right)^{2} + \left|\nabla_{\omega} f\right|^{2}.
\end{equation}

We may assume that $\qm > 1$ in $\mTa_{p,\tzau}$, $\mTa_{\Gammaj,\tzau}$, so that
from~\eqref{eqn:gradfunc}, the following inequality holds:
\begin{equation}\label{eqn:limited-grad-tau}
|\nabla \tau(x) | = \qm^{-1/2}(x) \leq 1, x \in \mTa_{p,\tzau}\setminus\{p\}\left(\mTa_{\Gammaj,\tzau}\setminus \Gammaj\right).
\end{equation}

Throughout this paper, for brevity of notation, we use $\epsilon > 0$ and $\delta > 0$ to denote small positive constants; these values could differ in a different contexts.

\section{RCN of the Potential of the Schr{\"o}dinger Operator}\label{chap:range-control}
\noindent \textbf{Definition. RCN of the potential.}

\noindent A neighborhood of positive reach $\bmTa_{p,\tzau}(\overGi, \bmTa_{\Gammaj,\tzau})$
is in \textit{range control of the potential }of operator~\eqref{eqn:schrodinger}, if for some $\tzau > 0$, the function
\begin{equation}\label{eqn:w}
  w(x) \vcentcolon=\qm^{3/4}(x) \sigma(x) \tau(x)
\end{equation}
is locally absolutely continuous on $(0, \tzau)$ (or on $(0, \infty)$ for $\overGi$) w.r.t.~parameter $\tau$
along all the minimizing curves connecting $p$ and $\partial \bmTa_{p,\tzau}$
($\Gammaj$ and $\partial \bmTa_{\Gammaj,\tzau}$, or any two disjoint boundaries $\parGi_{[0, \tzau]}$ for any $\tzau$), and 
the following
conditions for some large $\Czero(\tzau) > 0$ and small $\epszero(\tzau) > 0$ hold:
\begin{subequations}\label{dfn:positive-neighborhood-upd}
  \begin{align}
    & \tau \left|\frac{\partial \log\left(w(x)\right)}{\partial \tau}\right| <
      \frac{\Czero}{2}\sqrt{\frac{1-A - \del0}{1 + \Czero^{2}\epszerosquare}}
      \label{dfn:positive-neighborhood-upd:a}\\
    & \qm\tau <  \frac{\epszero}{2}\sqrt{\frac{A - \del0}{1 + \Czero^{2}\epszerosquare}}\label{dfn:positive-neighborhood-upd:b}
  \end{align}
\end{subequations}
\noindent with $0 < A < 1$ and $\del0 > 0$ such that $A - \del0>0$ and $A + \del0 < 1$,
and for all $x \in \bmTa_{p,\tzau}(\overGi, \bmTa_{\Gammaj,\tzau})$
with minorant function $\qm$ in~\eqref{eqn:metric} and
the Radon\text{--}Nikodym derivative $\sigma^{2}$ defined in~\eqref{eqn:sigma-squared}.

The conditions~\eqref{dfn:positive-neighborhood-upd} are symmetric in nature, so sometimes we use the name
\textit{dual potential} for the expression $\left|\frac{\partial \log\left(w(x)\right)}{\partial \tau}\right|$
in~\eqref{dfn:positive-neighborhood-upd:a}.

For the RCN of infinity neighborhoods $\overGi_{\tzau}$ this definition implies that $\Czero = \Czero(\tzau)$, $\epszero = \epszero(\tzau)$, and $A$ and $\del0$ are the same for any $\tzau$. 

We have a simple necessary condition of the RCN
\begin{corollary}\label{cor:positive-positive-neighborhood-upd}
  From definition~\eqref{dfn:positive-neighborhood-upd} of the RCN, the following condition holds:
  \begin{equation}\label{eqn:positive-neighborhood-necessary}
    \qm\tau^{2}\left|\frac{\partial \log\left(w(x)\right)}{\partial \tau}\right| < 1/16.
  \end{equation}
\end{corollary}
\begin{proof}
  Indeed, the left-hand side of \eqref{eqn:positive-neighborhood-necessary} is just the product of the left-hand side
  expressions in \eqref{dfn:positive-neighborhood-upd:a} and \eqref{dfn:positive-neighborhood-upd:b}, so we only estimate the product of their right-hand sides by
  \begin{equation*}
    \begin{aligned}
      & 1/4\frac{\Czero\epszero}{1 + \Czero^{2}\epszerosquare}\sqrt{(A -\del0)(1 - A - \del0)} \\
      & \leq 1/4 * 1/2 * \sqrt{(A -\del0)(1 - (A -\del0) - 2\del0)} \\
      & \leq 1/8 * \sqrt{1/4 - 2\del0(A -\del0)} \leq 1/16 * \sqrt{1 - 8\del0(A -\del0)} <  1/16.
    \end{aligned}
  \end{equation*}
\end{proof}

\noindent \textbf{Remarks}
\begin{enumerate}[label=\arabic*., ref=\arabic*]
\item We will later explain the choice of words we used in this definition when
  we consider the corresponding wave equation and its associated energy inequality in
  RCNs.
\item Since our examples below are for the case of spherically symmetric potentials, then we check the validity of condition~\eqref{eqn:w}
  only for the case of singular points $\Gammaj$ in a small neighborhood of zero;
  in all other cases, $w$ is locally Lipschitz,
  as it is the product of three Lipschitz functions.
\item The definitons~\eqref{eqn:tau-point} of $\bmTa_{p,\tzau}$ and Condition A
  imply that the Eq.~\eqref{dfn:positive-neighborhood-upd:b} can hold only at some possibly small neighborhood $\bmTa_{p,\tzau}$ of a regular or singular $p \in M$
  or at the entire infinity neighborhood.
  Additionally, condition~\eqref{dfn:positive-neighborhood-upd:b} implies that a singular point
  cannot be in the RCN of any other point, however close it may be
  near that singular point. Therefore, in this sense, singular points must be at the centers of their
  own RCNs.
\item\label{rem:tau_der} The first condition~\eqref{dfn:positive-neighborhood-upd:a} contains
  the square root of the volume element~$\sigma$. For regular and singular points of $\qm$, the expression under the logarithm sign of the dual potential tends
  to zero because of the condition in the second inequality.
  Note also that the expression
  $\qm^{-1/2} \frac{\partial \log\left(w(x)\right)}{\partial \tau}$ is
  the derivative in the direction of the unit vector field $\qm^{-1/2}\frac{\partial}{\partial \tau}$,
  and in the case of spherically symmetric functions in $\RRn$, the
  left-hand side of~\eqref{dfn:positive-neighborhood-upd:a} can be written as
  $\qm^{1/2}\tau\left|\frac{\partial \log\left(w(r)\right)}{\partial r}\right|.$
\end{enumerate}
Let us consider few examples of spherically symmetric potentials and
investigate whether they satisfy the conditions in~\eqref{dfn:positive-neighborhood-upd}.
\begin{example}\label{ex:minkowski}
  The regular potential $\qm = 1, M = \RRn, n \geq 1$.
  We see that $\tau = |x| = r$, so $w=C_1 r^{\frac{n-1}{2}}r = C_1 r^{\frac{n+1}{2}}$, and in~\eqref{dfn:positive-neighborhood-upd:b}, we have $\qm\tau = r$.
  For the condition \eqref{dfn:positive-neighborhood-upd:a}
  \begin{equation*}
    \begin{split}
      & \log (w(r)) = \log\left(C_1 r^{\frac{n+1}{2}}\right) = C_2 \log r + C_3,\text{ and}\\
      & \tau \left|\frac{\partial \log (w(r))}{\partial \tau}\right| 
       = \qm^{1/2}\tau \left|\qm^{-1/2}\frac{\partial \log (w(r))}{\partial \tau}\right| \\
      & = C_4\left(r \frac{\partial \log (w(r))}{\partial r}\right) = O(1).
    \end{split}
  \end{equation*}
  Therefore, the expression on the left-hand side of~\eqref{dfn:positive-neighborhood-upd:b} tends to zero, and we can always
  find constants $A, \del0, \Czero$, and $\epszero$ satisfying both conditions \eqref{dfn:positive-neighborhood-upd} in
  the neighborhood of the origin.
\end{example}
\begin{example}\label{ex:hydrogen}
  $\qm = \beta^{2}|x|^{-2\alpha}, M = \RRn, n \geq 1, \alpha, \beta > 0$.
  
  \noindent We have
  $\tau = \int_{0}^{r}\qm^{-1/2}dr = \frac{1}{\beta(\alpha + 1)}r^{\alpha + 1},$
  so
  $\qm\tau  = \frac{\beta}{\alpha + 1}r^{1-\alpha}$.

  \noindent Moreover, $w = \frac{C_1\beta^{1/2}}{\alpha+1}r^{-3\alpha/2+(n-1)/2+\alpha+1}  =
  \frac{C_1\beta^{1/2}}{\alpha+1}r^{(n-\alpha+1)/2},$
  so $w$ satisfies~\eqref{eqn:w} for $\alpha \leq 1$ and any $n$, in particular, and
  $\log(w) = \frac{n-\alpha+1}{2}\log r +  C_2$, so
  \begin{equation*}
    \begin{aligned}
      & \tau \left|\frac{\partial \log(w)}{\partial \tau} \right|  =
      \qm^{1/2} \tau \left|\qm^{-1/2}\frac{\partial \log(w)}{\partial \tau} \right| \\
      & = \beta r^{-\alpha}\frac{1}{\beta(\alpha + 1)}r^{\alpha + 1} \left|\frac{\partial \log(w)}{\partial r} \right| 
      = \frac{r}{\alpha + 1}  \frac{n-\alpha+1}{2r}\\
      & = \frac{n-\alpha+1}{2(\alpha + 1)}.
    \end{aligned}
  \end{equation*}
  Therefore, condition~\eqref{dfn:positive-neighborhood-upd} is clearly satisfied for $\alpha < 1$ and not satisfied for $\alpha > 1$ because of Corollary~\ref{cor:positive-positive-neighborhood-upd}.

  Let us consider the remaining boundary case for $\alpha = 1$. Note that the left-hand side of \eqref{dfn:positive-neighborhood-upd:a} is $n/4$, and the left-hand side of \eqref{dfn:positive-neighborhood-upd:b} is $\beta/2$; thus, to satisfy the necessary condition \eqref{eqn:positive-neighborhood-necessary}, we must have $\beta^{2} < 1/(4n^{2})$.

  The case for $\alpha = 1$ was studied in the D.1 example in~\cite{Braverman_2002}
  and in~\cite{Kalf_Walter_1972}, where the authors
  explored the essential self-adjointness of the operator~\eqref{eqn:schrodinger}
  in $D(H) = C^{\infty}_{0}\left(\RRn\setminus\{0\}\right)$. For $n \geq 5$, then $\Vm\in L^{2}_{\loc}(\RRn)$, and the operator~\eqref{eqn:schrodinger} is essentially self-adjoint if and only if $\beta^{2} \leq \left(\frac{n-2}{2}\right)^{2} - 1$, and it is semibounded from below when $\beta^{2} \leq \left(\frac{n-2}{2}\right)^{2}$.
  We will also prove that RCN conditions together with Conditions~C
  imply semiboundedness of the operator~\eqref{eqn:schrodinger} and its essential self-adjointness,
  but our condition $\beta^{2} < 1/(4n^{2})$ is more restrictive.

  Unlike in our case, the origin does not belong to $M$, so $M$ is not
  complete in both the Euclidean and inner time metrics. It is well known to require $n \geq 4$
  even for the Laplacian to be essentially self-adjoint in this open domain; see Remark~3 in~\cite{Bruse1998}, so, in this sense, we cannot compare our conditions to those stated above.
  
  In the addendum of this paper, we provide another proof of the
  sufficient conditions of the semiboundedness and essential self-adjointness
  for both the Laplacian operator and the Schr{\"o}dinger operator
  given in~\cite{Bruse1998} by using several important definitions introduced in this paper, i.e., the metric~\eqref{eqn:metric}, the vector field $\partial/\partial \tau$, etc.
\end{example}
\begin{example}\label{ex:not_kato_class}
  $\qm = |x|^{-2}(-\log |x|)^{-\hat{\delta}}, M = \RRn, n \geq 1$, and $\hat{\delta} > 0$ -
  see the examples in~\cite{cycon1987schrodinger} after Theorem 1.12 for $n \geq 3$.
  The authors note, in particular, that for $\hat{\delta} \leq 1$, this potential does not belong to the Kato class,
  and for $n\geq 5$ and $\hat{\delta} \leq 1/2$, it does not belong to the Stummel class. Moreover, the inequality~\eqref{eq:relative-bound} holds for any relative bound $\hat{\delta} > 0$.
  Let us denote $r = |x|$ and investigate the singularity at $r=0$, and we estimate $\tau$ for small $r>0$ by
  \begin{equation*}
    \tau = \left.\int_{0}^{r}\qm^{-1/2} dr\right. = \left.\int_{0}^{r}r(-\log r)^{\hat{\delta}/2} dr\right..
  \end{equation*}
  Now, using L'Hopital's rule, we obtain
  \begin{equation*}
    \qm\tau = r^{-2} \left(-\log r\right)^{-\hat{\delta}} \left.\int_{0}^{r}r(-\log r)^{\hat{\delta}/2} dr\right. \to 0 \text{ when } r \to 0,
  \end{equation*}
  and the left-hand side of \eqref{dfn:positive-neighborhood-upd:b} can be made arbitrarily small.

  \noindent For condition \eqref{dfn:positive-neighborhood-upd:a}, note that $\sigma = C_2r^{(n-1)/2}$, and let us estimate
  \begin{equation*}
    \begin{aligned}
      & w = C_2r^{-3/2}\left(-\log r\right)^{-3\hat{\delta}/4}r^{(n-1)/2} \left.\int_{0}^{r}r(-\log r)^{\hat{\delta}/2} dr\right.\\
      & = C_2r^{n/2 - 2}\left(-\log r\right)^{-3\hat{\delta}/4}\left.\int_{0}^{r}r(-\log r)^{\hat{\delta}/2} dr\right., 
    \end{aligned}
  \end{equation*}
  the condition~\eqref{eqn:w} is satisfied, and
  $$
  \frac{d\log(w)}{dr} \sim C_3/r.
  $$
  Therefore, using the last Remark~\ref{rem:tau_der}, we obtain
  $$\qm^{-1/2}\frac{\partial \log\left(w\right)}{\partial \tau}
  = \frac{d\log(w)}{dr} = O(1/r).$$
  From simple integration by parts, we have this simple asymptotic behavior for $\tau$
  \begin{equation*}
    \tau \sim r^{2}\left(-\log r\right)^{\hat{\delta}/2}\left(1 -\hat{\delta}/4(-\log r)^{-1}\right),
  \end{equation*}
  and we obtain
  \begin{equation*}
    \begin{aligned}
      & \qm^{1/2} \tau \qm^{-1/2} \frac{\partial \log(w)}{\partial \tau} \\
      & = O\left(r^{-1}(-\log r)^{-\hat{\delta}/2}r^{2}\left(-\log r\right)^{\hat{\delta}/2}\left(1 -\hat{\delta}/4(-\log r)^{-1}\right) r^{-1}\right)\\
      & = O(1),
    \end{aligned}
  \end{equation*}
  and the left-hand side of \eqref{dfn:positive-neighborhood-upd:a} is bounded.
\end{example}
   
The next example investigates RCNs for $M = \RRn$ at infinity.
\begin{example}
  $\qm = \beta^{2}|x|^{2\alpha}, r = |x| \gg 1, M = \RRn, n \geq 1, \alpha, \beta > 0$.
  Considering the interval length $\Delta r = \epsilon r^{-\alpha}$ with some small and fixed values $\epsilon$ and for $\alpha \leq 1$, we have
  $$ \tau = \int_{r}^{r+\Delta r}\qm^{-1/2}dr = \frac{r^{1 - \alpha}}{\beta(1-\alpha)}\left[\left(1 + \frac{\Delta r}{r}\right)^{1 - \alpha}  - 1\right]
  \sim 1/\beta r^{-\alpha}\Delta r,
  $$
  so
  \begin{equation}\label{eqn:cond-b-thm}
    \qm\tau  \sim \beta^{2} r^{2\alpha} (1/\beta) r^{-\alpha}\Delta r = \beta r^{\alpha} \Delta r = \epsilon \beta.
  \end{equation}
  
  If $\alpha > 1$, then $M$ is not complete w.r.t.~the metric~\eqref{eqn:metric}.

  \noindent For the expression with the dual potential in condition~\eqref{dfn:positive-neighborhood-upd:a}
  and selected $\Delta r$ above, we calculate
  \begin{equation*}
    w = \qm^{3/4}\sigma\tau
    \sim C_1(\alpha, \beta, \epsilon) r^{3\alpha/2} r^{-2\alpha} r^{(n-1)/2} =  C_1(\alpha, \beta, \epsilon) r^{(-\alpha + n - 1)/2} 
  \end{equation*}
  so $\log(w) \sim C_2(\alpha, n)\log r$, and
  \begin{equation}\label{eqn:cond-a-thm}
    \begin{aligned}
      & \tau \left|\frac{\partial \log(w)}{\partial \tau} \right|  =
      \qm^{1/2} \tau \left|\qm^{-1/2}\frac{\partial \log(w)}{\partial \tau} \right| 
      \sim \beta r^{\alpha} (1/\beta) r^{-\alpha}\Delta r \left|\frac{\partial \log(w)}{\partial r} \right| \\
      & \sim C_3(\alpha, n) \epsilon r^{-1-\alpha}.
    \end{aligned}
  \end{equation}
  Combining estimates~\eqref{eqn:cond-b-thm} and~\eqref{eqn:cond-a-thm} and noticing the monotonicity of their upper bounds,
  we see that the product from Corollary~\ref{cor:positive-positive-neighborhood-upd} can be estimated by
  $\qm\tau^{2} \left|\frac{\partial \log(w)}{\partial \tau} \right|
  < C_4(\alpha, \beta, n) \epsilon^{2} r^{-\alpha-1}$, and it can be made arbitrarily small for large $r$ (and smaller than 1/16),
  so we can always find constants $\Czero = \Czero(r), \epszero = \epszero(r)$
  and fixed constants $\del0 > 0$ and $0 < A < 1$ such that
  conditions~\eqref{dfn:positive-neighborhood-upd:a} and~\eqref{dfn:positive-neighborhood-upd:b} are satisfied.
  Note that with this way defined as $\Delta r$, we can extend RCNs to infinity.
\end{example}

\noindent We are ready to formulate the following:
\begin{lemma}\label{lemma:positivity}
  Let the neighborhood $\bmTa_{p, \tzau}$ be in the range control of the potential at $p \in M.$
  Then, for any real-valued $\phi \in W^{1,2}_{0}(\bmTa_{p, \tzau})$, such that $\qm^{1/2}\phi \in L^{2}(\bmTa_{p, \tzau})$, 
  we have
  \begin{equation}\label{eqn:bilinear-form}
    \left.\int_{\bmTa_{p, \tzau}}\qm\phi^{2}d\mu\right. \leq \delta \left.\int_{\bmTa_{p, \tzau}}\left|\nabla\phi\right|^{2}d\mu\right.
  \end{equation}
  for some $\delta < 1$. 
\end{lemma}
For the proof of Lemma~\ref{lemma:positivity}, we need this
\begin{proposition}[Domain Definition in Estimate~\eqref{eqn:bilinear-form}]\label{prop:approx}
  Under the conditions of Lemma~\ref{lemma:positivity} functions $\phi \in W^{1,2}_{0}(\bmTa_{\Gammaj, \tzau})$
  with $\qm^{1/2}\phi \in L^{2}(\bmTa_{\Gammaj, \tzau})$
  can be approximated by $C^{\infty}_{0}(\bmTa_{\Gammaj, \tzau}\setminus\Gammaj)$ in the norm
  of the Sobolev space $W^{1,2}_{0}(\bmTa_{\Gammaj, \tzau})$.
\end{proposition}
\begin{proof}
  We select the case of an RCN whose center $\Gammaj$ contains a singularity of $\Vm$; the other cases of RCNs
  of regular points or infinity are easier to prove here because of the regularity of $\qm$ in these neighborhoods; see condition~\ref{prop:q-minus:b}.

  Owing to condition~\ref{prop:q-minus:b}, the operator of multiplication
  $$
  \mathcal{A}: W^{1,2}_{0}(\bmTa_{\Gammaj, \tzau}) \rightarrow W^{1,1}_{0}(\bmTa_{\Gammaj, \tzau})
  $$
  defined by
  $$
  \phi \mapsto \qm^{-1/2}\phi
  $$
  is continuous, and, since we have a continuous embedding $W^{1,2}_{0}(\bmTa_{\Gammaj, \tzau}) \hookrightarrow W^{1,1}_{0}(\bmTa_{\Gammaj, \tzau})$,
  it is sufficient to prove that $\text{Ran}(\mathcal{A}) \cap W^{1,2}_{0}(\bmTa_{\Gammaj, \tzau})$ is dense in $W^{1,2}_{0}(\bmTa_{\Gammaj, \tzau})$.
  It is clear that the space $W^{1,2}_{0}(\bmTa_{\Gammaj, \tzau}\setminus \Gammaj)$ is dense in both $W^{1,2}_{0}(\bmTa_{\Gammaj, \tzau})$ and
  $W^{1,2}_{0}(\bmTa_{\Gammaj, \tzau}\setminus \Gammaj) \subset \text{Ran}(\mathcal{A})$, and it is, by definition, the completion
  of $C^{\infty}_{0}(\bmTa_{\Gammaj, \tzau}\setminus\Gammaj)$ in the norm of $W^{1,2}_{0}(\bmTa_{\Gammaj, \tzau})$. 
\end{proof}
\begin{proof}[Proof of Lemma~\ref{lemma:positivity}]
  In the lemma statement, we use a point $p\in M$ and its neighboring $\bmTa_{p, \tzau}$, but the lemma is also
  valid both for RCNs $\bmTa_{\Gamma_{j}, \tzau}$ and $\overGi_{\tzau}$.
  In the proof below, we stress the difference for these cases when necessary.
  
  From Proposition~\ref{prop:approx} above, it is sufficient to prove this lemma for $\phi \in C^{\infty}_{0}(\bmTa_{p, \tzau})$,
  or $\phi \in C^{\infty}_{0}(\bmTa_{\Gammaj, \tzau}\setminus \Gammaj)$ for the RCN of $\Gammaj$. 

  For the case of $\overGi_{\tzau}$
  the boundary $\parGi_{[0, \tzau]}$ consists of two regular disjoint components,
  and the trace operator can be applied to each one of them separately.
  
  Let us take any minimal curve from $p$ to $\partial \bmTa_{p, \tzau}$
  (from $p\in \Gammaj$ or from $p\in \parGi$)
  w.r.t.~metric \eqref{eqn:metric}, and we then evaluate
  \begingroup
  \allowdisplaybreaks
  \begin{align*}
    & \left.\int_{0}^{\tzau}\qm \phi^{2} \qm^{1/2} \sigma^{2} d\tau\right. 
    = \left.\int_{0}^{\tzau}\frac{w^{2} \phi^{2}} {\tau^{2}} d\tau\right.
    \leq 4 \left.\int_{0}^{\tzau} \left(\frac{\partial (w\phi)}{\partial \tau}\right)^{2} d\tau\right. \\
    & = 4 \left.\int_{0}^{\tzau} \left(\Czero w \frac{\partial \phi}{\partial \tau}\right.  +
      1/\Czero\frac{\partial w}{\partial \tau}\phi\right)^{2} d\tau\\
    & \leq  \left.4\left(\Czerosquare + 1/\epszerosquare\right) \int_{0}^{\tzau} w^{2}\left(\frac{\partial \phi}{\partial \tau}\right)^{2} d\tau\right. 
     + 4\left(1/\Czerosquare + \epszerosquare\right) \left.\int_{0}^{\tzau} \left(\frac{\partial w}{\partial \tau}\right)^{2}\phi^{2} d\tau\right.\\
    & = 4\left(\Czerosquare + 1/\epszerosquare\right) \left.\int_{0}^{\tzau} \left(\qm^{-1}\left(\frac{\partial \phi}{\partial \tau}\right)^{2}(\qm\tau)^{2}\right)
      \qm^{1/2}\sigma^{2}d\tau\right.\\
    & + 4\left(1/\Czerosquare + \epszerosquare\right)\left.\int_{0}^{\tzau}\left(\left(\frac{\partial \log(w)}{\partial \tau}\right)^{2} \qm\tau^{2}\phi^{2}\right) \qm^{1/2}\sigma^{2}d\tau\right.\\
    & = 4\left(\Czerosquare + 1/\epszerosquare\right) \left.\int_{0}^{\tzau} \left(\qm^{-1}\left(\frac{\partial \phi}{\partial \tau}\right)^{2}(\qm\tau)^{2}\right)  \qm^{1/2}\sigma^{2}d\tau\right.\\
    & + 4\left(1/\Czerosquare + \epszerosquare\right)\left.\int_{0}^{\tzau}\qm\left(\left(\tau\frac{\partial \log(w)}{\partial \tau}\right)^{2} \phi^{2}\right) \qm^{1/2}\sigma^{2}d\tau\right..
  \end{align*}%
  \endgroup
The first inequality is the one-dimensional second-degree Hardy inequality; see~\cite{Davies1999},
\begin{equation*}
  \left.\int_{0}^{\tzau}\frac{f^{2}(\tau)}{\tau^{2}}d\tau\right. \leq \left.4 \int_{0}^{\tzau} \left(f^{\prime}(\tau)\right)^{2}d\tau\right.
\end{equation*}
for all locally absolutely continuous $f$ in $(0,\tzau)$ with $f(0) = 0$. We can apply it here,
since the minimizing curve is from $\Lip_{\loc}^{1,1}$, and
the function $f\vcentcolon=w\phi$ is locally absolutely continuous owing to the condition~\eqref{eqn:w} for $w$
and $\phi$ being smooth, and it tends to zero at $p$ even when $\qm(x)\to \infty$ when $x\to p$;
this is due to the inequality~\eqref{dfn:positive-neighborhood-upd:b} and the boundedness of $w$ and $\phi$ in
$\bmTa_{p, \tzau}$. We use the Cauchy inequality in the second inequality.

Integrating both the first and last parts of the above inequality for all the minimal curves yields
with respect to the volume measure $\left.\text{dVol}\right|_{g_{\tzau}}$ on $\partial \bmTa_{p,\tzau}$,
capping $\qm^{-1}\left(\frac{\partial \phi}{\partial \tau}\right)^{2}$
by $\left|\nabla \phi\right|^{2}$ via expression~\eqref{eqn:gradfunc},
and using the form~\eqref{eqn:volume-element} of the measure $d\mu$, we obtain
\begin{equation}\label{eqn:mainineq}
\begin{split}
  & \left.\int_{\bmTa_{p, \tzau}}\qm\phi^{2}d\mu\right. \leq
  \left.4\left(\Czerosquare + 1/\epszerosquare\right) \int_{\bmTa_{p, \tzau}} (\qm\tau)^{2}\left|\nabla \phi\right|^{2}d\mu\right.\\
  & + 4\left(1/\Czerosquare + \epszerosquare\right)\left.\int_{\bmTa_{p, \tzau}}
  \qm\left(\tau\frac{\partial \log(w)}{\partial \tau}\right)^{2} \phi^{2} d\mu\right..
\end{split}
\end{equation}

Under condition~\eqref{dfn:positive-neighborhood-upd}, we selected the constants $A, \del0, \Czero,$ and $\epszero$ such that
\begin{equation}\label{eqn:greater-A}
  1 - 4\left(1/\Czerosquare + \epszerosquare\right)\left(\tau\frac{\partial \log(w)}{\partial \tau}\right)^{2} > A + \del0
\end{equation}
and
\begin{equation}\label{eqn:less-A}
  4\left(\Czerosquare + 1/\epszerosquare\right)(\qm\tau)^{2} < A - \del0
\end{equation}
in $\bmTa_{p, \tzau}$, so in the lemma formulation, we can define $\delta\vcentcolon=(A - \del0)/(A + \del0)$, and
combining the left-hand side of \eqref{eqn:mainineq} with its last term on the right-hand side and taking into account
preceding inequalities leads to the completion of the proof.
\end{proof}

Note that for the infinity neighborhood $\overGi$ conditions~\eqref{dfn:positive-neighborhood-upd}  
can hold throughout $\overGi$; they could be satisfied for $\tau\to\infty$ and $\qm\to 0$ ,
so that the expressions on the left-hand side of~\eqref{dfn:positive-neighborhood-upd:b} could tend to zero,
and the expression in~\eqref{dfn:positive-neighborhood-upd:a} may not grow too rapidly in $\overGi_{\tau}$
when $\tau\to\infty$.

We study the global finite propagation speed of the solutions of the Cauchy
problem for the wave equation with operator~\eqref{eqn:schrodinger}.
As the essential self-adjointness of~\eqref{eqn:schrodinger}
largely depends on the structure of RCNs of its potential, 
we make the following assumption:

\noindent \textbf{Condition D. All admissible covers are RCNs.}
For the admissible open covers defined in Condition C, we assume that
there exists $\delta > 0$, so that all covers in Condition C are RCNs of
their respective centers with the same $\delta$ defined in Lemma~\ref{lemma:positivity}.
For the infinity neighborhoods $\Gi$ we assume that $\overGi_{\tzau}$ are RCNs for any $\tzau > 0$.

This condition D and the definitions~\eqref{dfn:positive-neighborhood-upd} lead to the fact that any layer
$\overGi_{[\toau, \ttwo]}$ is an RCN for any $0 \leq \toau < \ttwo$.

\section{Schr{\"o}dinger Wave Equation in the Range Control Neighborhood.
  Domain of Dependence. Uniqueness of the Solutions.}\label{chap:uniqueness}
Domains of dependence of the wave equation for the Schr{\"o}odinger operator are defined
for regular (and singular) points and neighborhoods of infinity separately. 

\subsection{Domain of Dependence for Regular or Singular Points. Uniqueness of the Solutions.}
For a regular point $p\in M$ or a submanifold $\Gammaj$ of singularities, we define this domain
of dependence cone $G^{\hat{\delta}}$ as
\begin{equation}\label{eqn:domain-dependence-regular}
  t + (1 - \hat{\delta})^{1/2}\tau(x) \leq (1 - \hat{\delta})^{1/2}\tzau,\ 0\leq t\leq \hat{T},\ x\in \bmTa_{p, \tzau}(\bmTa_{\Gammaj, \tzau}),
\end{equation}
where $\bmTa_{p,\tzau}(\bmTa_{\Gammaj, \tzau})$ are the neighborhoods of the positive reach,
$\tau(x) \vcentcolon = \tau(p, x) (\tau(\Gammaj, x))$, and 
$\hat{T}\vcentcolon=(1 - \hat{\delta})^{1/2}\tzau$ with the parameter $\hat{\delta}$ satisfying
$\delta \leq \hat{\delta} < 1$, where $\delta \geq 0$ is defined in Lemma~\ref{lemma:positivity}.
Let us define further $G^{\hat{\delta}}_{\tilde{t}}\vcentcolon=\{t=\tilde{t}\}\cap G^{\hat{\delta}}$.

In this section we investigate a Cauchy\text{--}Dirichlet problem for the wave equation in the domain
$Q_{T} \vcentcolon= [0, T]\times \bmTa_{p,\tzau}$
with the range control neighborhood $\bmTa_{p,\tzau}$
defined in \eqref{dfn:positive-neighborhood-upd} and
\begin{equation}\label{dfn:T-domain-dependence}
  T \vcentcolon=(1 - \delta)^{1/2}\tzau,\ x\in \bmTa_{p,\tzau}(\bmTa_{\Gammaj, \tzau}).
\end{equation}
Therefore, $T = \max(\hat{T})$ over all $\hat{T}$ defined in~\eqref{eqn:domain-dependence-regular}.

We consider this problem with $S_{T} \vcentcolon=[0, T] \times \partial \bmTa_{p,\tzau}$
\begin{subequations}\label{eqn:hyperbolic-equation}
  \begin{align}
    & \frac{\partial^{2} u}{\partial t^{2}} + Hu  = \rho(t, x), \rho \in L^{2}(Q_{T})\label{eqn:hyperbolic-equation:a}\\
    & u(0, x) = f(x), f \in W_{0}^{1, 2}(\bmTa_{p,\tzau}), \Vp^{1/2}f,\qm^{1/2}f \in L^{2}(\bmTa_{p,\tzau}),  \label{eqn:hyperbolic-equation:b}\\
    & \frac{\partial u(0, x)}{\partial t} = g(x), g \in L^{2}(\bmTa_{p,\tzau}),\label{eqn:hyperbolic-equation:c} \\
    & \left.u(t, x)\right|_{S_{T}}, \left.\frac{\partial u(t, x)}{\partial t}\right|_{S_{T}} = 0.  \label{eqn:hyperbolic-equation:d}
  \end{align}
\end{subequations}

\noindent We need to define a class of weak solutions for the Cauchy\text{--}Dirichlet problem.

\noindent \textbf{Definition. Solutions } $\mU$ \textbf{ of} \eqref{eqn:hyperbolic-equation}.
\noindent A function $u \in \mU$  iff   
\begin{subequations}\label{dfn:solutions-class}
  \begin{align}
    & u \in W^{1, 2}(Q_{T}),\label{dfn:solutions-class:a}\\
    & \Vp^{1/2}u, \qm^{1/2}u \in L^{2}(Q_{T}).\label{dfn:solutions-class:b}
  \end{align}
\end{subequations}
Note that this definition is similar to the definition of the Schr{\"o}dinger operator domain in
Section 1 of~\cite{Oro88}, and this class of solutions was defined in~\cite{Wilcox1962} in
Corollary~5.2 and in~\cite{VisLad56} in $\S 3$ for regular potentials.

We are now ready to formulate the following:
\begin{theorem}\label{thm:domain-dependence}
  {\bf The Domain of Dependence for Regular or Singular Points.} Let $u \in \mU$ be the solution of the Cauchy\text{--}Dirichlet problem~\eqref{eqn:hyperbolic-equation}.
  Then, for a.e. $\hat{T} \leq T$ the domain of dependence equality
  \begin{equation}\label{eqn:domain-dependence}
    E(\hat{T}) -  E(0) =  \int_{0}^{\hat{T}}\left(\int_{G^{\hat{\delta}}_{\tilde{t}}}\rho(t, x)\frac{\partial u}{\partial t}d\mu\right) d\tilde{t},
  \end{equation}
  holds with
  \begin{equation}\label{eqn:energy-integral}
    E(\hat{T}) \vcentcolon= \frac{1}{2}\left.\int_{\hat{S}}\left[\left(\left(\frac{\partial u}{\partial t}\right)^{2} + \left|\nabla u\right|^{2} +
          Vu^2 \right)\left\langle \frac{\partial}{\partial t}, n_{t}\right\rangle -
        2\frac{\partial u}{\partial t}\left\langle \nabla u, n_{x}\right\rangle\right] dS\right.,
  \end{equation}
  where the integral~\eqref{eqn:energy-integral} is taken over the lateral surface of the cone
  $\hat{S} = \{(t, x) | t + (1 - \hat{\delta})^{1/2}\tau(x) = \hat{T}\}$ with $\hat{T}\vcentcolon=(1 - \hat{\delta})^{1/2}\tzau$
  and with a.e. $\hat{\delta}$ satisfying $\delta \leq \hat{\delta} < 1$.
\end{theorem}
In the theorem formulation, we denote $n_{t}$ and $n_{x}$ as the time and spatial components of
the normal vector to the lateral surface $\hat{S}$, and
$dS$ denotes the volume measure induced by the standard Riemannian metric
$d\ell^{2} = dt^{2} + dl^{2}$ defined on the manifold $\mathbb{R}\times M$.

Notably, in the literature, the energy integrals $E(\hat{T})$
are typically defined on the bases $G^{\hat{\delta}}_{0}$ and $G^{\hat{\delta}}_{t}$ of a truncated cone
for any $0 < t < \hat{T}$, and
here we define them on their lateral surfaces $\hat{S}$ to overcome
challenges arising from possible singularities of $\qm$ at the vertices of these cones.

Note that according to the initial conditions \eqref{eqn:hyperbolic-equation:b} and \eqref{eqn:hyperbolic-equation:c},
we have
$$E(0) = \left.\frac{1}{2} \int_{\bmTa_{p,\tzau}}\left(g^{2} + \left|\nabla f\right|^{2} + Vf^{2} \right) d\mu,\right.$$
so that $E(0)\geq 0$ due to Lemma~\ref{lemma:positivity}, and $E(0) = 0$ iff $f=g=0.$
It could be useful to view the base of the dependency cone $\bmTa_{p,\tzau}$ as the limit of lateral surfaces
$\hat{S}$ with $\hat{T}\to 0$ when $\hat{\delta} \to 1$.

Before we turn to the proof of Theorem~\ref{thm:domain-dependence}, we need to clarify the properties of
the integral in~\eqref{eqn:energy-integral}.
\begin{lemma}\label{lemma:energy-integral}
  For a.e. $\hat{T} \leq T$ the integral~\eqref{eqn:energy-integral} exists and is nonnegative for any solution $u\in\mU$
  of the Cauchy\text{--}Dirichlet problem~\eqref{eqn:hyperbolic-equation}. 
\end{lemma}
\begin{proof}
  Given that the solution $u\in\mU$ implies that $u\in W^{1,2}(G^{\delta})$,
  where $G^{\delta} = \{(t, x):\ x\in \bmTa_{p,\tzau},\ t\geq 0,\ t + (1 - \delta)^{1/2}\tau(x) \leq T\}$
  with $T$ defined in~\eqref{dfn:T-domain-dependence}.

  The cone $G^{\delta}$ has a Lipschitz boundary $\partial G^{\delta}$,
  so we can always extend a solution $u \in W^{1,2}(G^{\delta})$
  to a function $\tilde{u}\in W^{1,2}_{0}([0, T+\epsilon]\times M)$
  with $\tilde{u}=u$ on $G^{\delta}$; see, for instance, Theorem~4.7 in~\cite{evansgariepy2015}.

  Denote the integrand in~\eqref{eqn:energy-integral} by $h$ with $u$ replaced by $\tilde{u}$; note that $h$
  is summable and of compact support in the neighborhood of the cone; then,
  for a.e. $\hat{T} > 0$ and fixed $\hat{\delta}$ we have the formula - see Theorem~3.13 in~\cite{evansgariepy2015}
  \begin{equation*}
    \begin{aligned}
      & \left.\frac{d}{d\xi}\left(\int_{\{t+(1-\hat{\delta})^{1/2}\tau >\hat{T} + \xi\}}hdtd\mu\right)\right|_{\xi=0+} =
      -\int_{\{t+(1-\hat{\delta})^{1/2}\tau =\hat{T}\}}\frac{h}{|D(t+(1-\hat{\delta})^{1/2}\tau)|}dS  \\
      & = -\int_{\hat{S}}\frac{h}{\sqrt{1 + (1-\hat{\delta})|\nabla\tau|^{2}}}dS.
    \end{aligned}
  \end{equation*}
  The formula in~\cite{evansgariepy2015} uses the $n-$dimensional Hausdorff measure instead of the induced measure $dS$,
  but the regularity of the lateral surface implies equivalence of these measures.
  We have used the fact that the norm of the differential
  $|D(t+(1-\hat{\delta})^{1/2}\tau)| = \sqrt{1 + (1-\hat{\delta})|\nabla\tau|^{2}}$ is  
  bounded on $\hat{S}$, so the expression on the right hand side of~\eqref{eqn:energy-integral} exists
  for a.e. $\hat{T} \leq T$.

  Let us prove that the integral~\eqref{eqn:energy-integral} is nonnegative.
  Note first that
  \begin{equation*}
    d\mu = \frac{(1-\hat{\delta})^{1/2}|\nabla \tau|}{\sqrt{1 + (1-\hat{\delta})|\nabla\tau|^{2}}} dS,
  \end{equation*}
  so that we could estimate
  \begingroup
  \allowdisplaybreaks
  \begin{align*}
    & E(\hat{T}) =\frac{1}{2}\left.\int_{\hat{S}}\left[\left(\left(\frac{\partial u}{\partial t}\right)^{2} + \left|\nabla u\right|^{2} +
      Vu^2 \right)\left\langle \frac{\partial}{\partial t}, n_{t}\right\rangle -
      2 \frac{\partial u}{\partial t}\left\langle \nabla u, n_{x}\right\rangle\right] dS\right.\\
    & =\left.\int_{\hat{S}}\frac{\left(\frac{\partial u}{\partial t}\right)^{2} + \left|\nabla u\right|^{2} +
      Vu^2 - 2(1-\hat{\delta)}^{1/2}\frac{\partial u}{\partial t}\langle \nabla u, \nabla \tau\rangle}{\sqrt{1 + (1-\hat{\delta})|\nabla\tau|^{2}}} dS\right.\\
    & \geq \left.\int_{\hat{S}}\frac{\left(\frac{\partial u}{\partial t}\right)^{2} + \left|\nabla u\right|^{2} +
        Vu^2 - 2(1-\hat{\delta)}^{1/2}|\nabla\tau|\left|\frac{\partial u}{\partial t}\right||\nabla u|} {\sqrt{1 + (1-\hat{\delta})|\nabla\tau|^{2}}}dS\right.\\
    & \geq \left.\int_{\hat{S}}\frac{ \frac{ (1-\hat{\delta} )^{1/2}}{(1-\hat{\delta})^{1/2}}|\nabla\tau|\left(\left(\frac{\partial u}{\partial t}\right)^{2} + \left|\nabla u\right|^{2} + Vu^2\right) - 2(1-\hat{\delta)}^{1/2}|\nabla\tau|\left|\frac{\partial u}{\partial t}\right||\nabla u|} {\sqrt{1 + (1-\hat{\delta})|\nabla\tau|^{2}}}dS\right.\\
    & = \left.\int_{\bmTa_{p,\tzau}}\left[\frac{1}{(1-\hat{\delta})^{1/2}}\left(\left(\frac{\partial u}{\partial t}\right)^{2} + \left|\nabla u\right|^{2} + Vu^2\right) - 2\left|\frac{\partial u}{\partial t}\right||\nabla u|\right] d\mu\right.\\    
    & = \frac{1}{(1-\hat{\delta})^{1/2}}\left.\int_{\bmTa_{p,\tzau}}\left(\left(\frac{\partial u}{\partial t}\right)^{2} + \left|\nabla u\right|^{2} + Vu^2 - 2(1-\hat{\delta})^{1/2}\left|\frac{\partial u}{\partial t}\right||\nabla u| \right)d\mu\right.\\
    & \geq \frac{1}{(1-\hat{\delta})^{1/2}}\int_{\bmTa_{p,\tzau}}\left(\hat{\delta}\left|\nabla u\right|^{2} + Vu^2\right)d\mu \geq 0.
  \end{align*}%
  \endgroup
  Note here that the solution $u\in\mU$, so $u$ vanishes
  at the boundary $\{t = 0\}\times \partial \bmTa_{p,\tzau}$ because
  the initial condition~\eqref{eqn:hyperbolic-equation:b} with $f\in W^{1, 2}_{0}(\bmTa_{p,\tzau})$ - all $\hat{S}$
  share this common boundary; thus, in the last inequality, we can
  apply Lemma~\ref{lemma:positivity}.
  We also use an explicit form for the components $n_{t}$ and $n_{x}$ of the normal vectors to $\hat{S}$,
  and the fact that $|\nabla \tau|=\qm^{-1/2} \leq 1$ in $\bmTa_{p,\tzau}$; see~\eqref{eqn:limited-grad-tau}.
\end{proof}

\begin{proof}[Proof of Theorem~\ref{thm:domain-dependence}]

  To obtain energy estimates for the domain of dependence $G^{\hat{\delta}}$,
  we multiply both sides of~\eqref{eqn:hyperbolic-equation:a} by $\frac{\partial u}{\partial t}$
  and take the integral over $G^{\hat{\delta}}$.

  Since $u\in\mU$ implies that $u\in W^{1,2}(G^{\hat{\delta}})$, we can always approximate $u$ by a mollified sequence $u_{m}\in C_{0}^{\infty}(G^{\delta})$ such that $u_{m}\to u$ in $W^{1,2}(G^{\hat{\delta}})$.
  In case of a RCN of singularity $\Gammaj$, as we note in Proposition~\ref{prop:approx},
  we select a sequence $u_{m}\in C_{0}^{\infty}(G^{\delta}\setminus \{[0, T] \times \Gammaj\})$.
  
  \noindent We have 
  \begin{equation}\label{eqn:ener-ineq-2}
    \frac{\partial u_{m}}{\partial t} \frac{\partial^{2} u_{m}}{\partial t^{2}} 
    = \frac{1}{2}\frac{\partial}{\partial t} \left(\frac{\partial u_{m}}{\partial t}\right)^{2},
  \end{equation}
  \begin{equation}\label{eqn:divergence-identity}
    \begin{split}
      & -\frac{\partial u_{m}}{\partial t}\Delta u_{m} = -\rm{div} \left(\frac{\partial u_{m}}{\partial t} \nabla u_{m}\right) +
      \left\langle \nabla u_{m}, \nabla \left(\frac{\partial u_{m}}{\partial t}\right)\right\rangle \\
      & = -\rm{div} \left(\frac{\partial u_{m}}{\partial t} \nabla u_{m}\right) +
       \frac{1}{2}\frac{\partial}{\partial t}| \nabla u_{m}|^{2},
    \end{split}
  \end{equation}
  and
  \begin{equation}\label{eqn:potentials-identity}
    \begin{split}
      & \frac{\partial u_{m}}{\partial t}V u_{m} = \frac{1}{2}\frac{\partial}{\partial t} \left(Vu_{m}^{2}\right)
    \end{split}
  \end{equation}
  
  Integrating the last terms of~\eqref{eqn:ener-ineq-2},~\eqref{eqn:divergence-identity}, and~\eqref{eqn:potentials-identity}
  over $G^{\hat{\delta}}$ for a.e. $\hat{\delta} > \delta$ with $\hat{\delta} < 1$, and passing the limit, we obtain
  \begin{equation}\label{eqn:ener-ineq-1}
    \begin{split}
      & 1/2\lim_{m\to\infty}\int_{\hat{S}}\left(\left(\frac{\partial u_{m}}{\partial t}\right)^{2} + \left|\nabla u_{m}\right|^{2} +
        Vu_{m}^2 \right)\left\langle \frac{\partial}{\partial t}, n_{t}\right\rangle dS- \\
      & - \lim_{m\to\infty}\int_{\hat{S}}
          \frac{\partial u_{m}}{\partial t}\left\langle \nabla u_{m}, n_{x}\right\rangle dS\\
      & - 1/2\lim_{m\to\infty}\left.\int_{\bmTa_{p,\tzau}}\left(\left(\frac{\partial u_{m}}{\partial t}\right)^{2} + \left|\nabla u_{m}\right|^{2} +
        Vu_{m}^2 \right)d\mu\right. \\
      & = \lim_{m\to\infty}\int_{0}^{\hat{T}}\left(\int_{G^{\hat{\delta}}_{\tilde{t}}}\rho(t, x)\frac{\partial u_{m}}{\partial t}d\mu\right) d\tilde{t}.\\
    \end{split}
  \end{equation}

  The integrals on the left-hand side of~\eqref{eqn:ener-ineq-1} converge to $E(\hat{T}) - E(0)$ when $u_{m}\to u$ in $W^{1,2}(G^{\hat{\delta}})$,
  $\Vp^{1/2}u_{m}\to \Vp^{1/2}u$, $\Vm^{1/2}u_{m}\to \Vm^{1/2}u$ in $L^{2}(G^{\hat{\delta}})$ due to Lemma~\ref{lemma:positivity}
  and the dominated convergence theorem, and the right-hand side of~\eqref{eqn:ener-ineq-1} converges to the right-hand side
  of~\eqref{eqn:domain-dependence} because both $\frac{\partial u}{\partial t}$ and $\rho(t, x)\in L^{2}(G^{\hat{\delta}})$,
  so that the solution $u\in\mU$ of~\eqref{eqn:hyperbolic-equation}
  satisfies the expression~\eqref{eqn:domain-dependence} for a.e. $\hat{T}\leq T$.
\end{proof}

\noindent A very simple corollary of Theorem~\ref{thm:domain-dependence} is
\begin{corollary}\label{corol:unique-solutions}
  \textbf{Uniqueness of the Solutions.}
  A solution $u\in\mU$ of the Cauchy\text{--}Dirichlet problem \eqref{eqn:hyperbolic-equation} is unique in the domain of
  dependence $G^{\hat{\delta}}$ for a.e. $\hat{T} < T$, as defined in~\eqref{eqn:domain-dependence-regular},
  with $T$ defined in~\eqref{dfn:T-domain-dependence}.
\end{corollary}
\begin{proof}
  If $u_1$ and $u_2$ are the solutions of \eqref{eqn:hyperbolic-equation}, then the difference $\tilde{w}\vcentcolon=u_1 - u_2$
  belongs to $\mU$, and it is the solution of the Cauchy\text{--}Dirichlet problem \eqref{eqn:hyperbolic-equation}
  with the initial conditions \eqref{eqn:hyperbolic-equation:b} and \eqref{eqn:hyperbolic-equation:c} for $f=g=0$ and
  with the source function $\rho = 0$.
  Applying the domain of dependence inequality \eqref{eqn:domain-dependence} to $\tilde{w}$, we obtain $E(\hat{T}) = 0$
  for almost all $\hat{T} \leq T$, and, at the same time,
  $$E(\hat{T}) \geq C(\hat{\delta})\int_{\hat{S}}|\nabla \tilde{w}|^{2}d\mu, C(\hat{\delta}) > 0$$
  due to the above estimate for $E(\hat{T})$ in Lemma~\ref{lemma:energy-integral}, so $\tilde{w}$ is constant a.e.~in~$G^{\hat{\delta}}$,
  and it is zero due to $\tilde{w}$ vanishing at $\partial \bmTa_{p, \tzau}$.
\end{proof}

The corollary below establishes the energy inequality in the domain of dependence $G^{\hat{\delta}}$
\begin{corollary}\label{corol:energy-inequality}
  \textbf{Energy Inequality.}
  For a.e. $\hat{T} < T$ with corresponding $\hat{\delta}$ and $\delta$
  defined in~\eqref{eqn:domain-dependence-regular} such that $\hat{\delta} - \delta \geq \delta_{0} > 0$, we have
  \begin{equation*}
    E(\hat{T}) \leq C_1\left(E(0) + \int_{0}^{\hat{T}}\left(\int_{G_{t}^{\hat{\delta}}}\rho^{2}(t, x) d\mu\right)dt\right)
  \end{equation*}
  with $C_1 = C_1(\delta_{0}, \delta, \hat{T})$.
\end{corollary}
\begin{proof}
  Our proof essentially follows the proof of Theorem~8 in~\cite{Wilcox1962}.
  
  We use notation from Lemma~\ref{lemma:energy-integral};
  there we have an estimate for $E(\hat{T})$, so let us rewrite this estimate in a form more suitable for our proof, namely
  \begingroup
  \allowdisplaybreaks
  \begin{align*}
    & E(\hat{T}) \geq \frac{1}{(1-\hat{\delta})^{1/2}}\left.\int_{\bmTa_{p,\tzau}}\left(\left(\frac{\partial u}{\partial t}\right)^{2} + \left|\nabla u\right|^{2} + Vu^2 - 2(1-\hat{\delta})^{1/2}\left|\frac{\partial u}{\partial t}\right||\nabla u| \right)d\mu\right.\\
    & \geq \frac{1}{(1-\hat{\delta})^{1/2}}\left.\int_{\bmTa_{p,\tzau}}\left(\left(\frac{\partial u}{\partial t}\right)^{2} + \left|\nabla u\right|^{2} + Vu^2 - 2\frac{(1-\hat{\delta})^{1/2}}{(1-\hat{\delta_1})^{1/2}}(1-\hat{\delta_1})^{1/2}\left|\frac{\partial u}{\partial t}\right||\nabla u| \right)d\mu\right.\\
      & \geq\frac{1}{(1-\hat{\delta)}^{1/2}} \left.\int_{\bmTa_{p,\tzau}} \left[\frac{\hat{\delta}-\delta_{1}}{1-\delta_1}\left(\frac{\partial u}{\partial t}\right)^{2} +
       \delta_{1}\left|\nabla u\right|^{2} + Vu^2\right]d\mu\right.\\
     & \geq \frac{\hat{\delta}-\delta_{1}}{(1-\delta_{1})(1-\hat{\delta})^{1/2}}\left.\int_{\bmTa_{p,\tzau}}\left(\frac{\partial u}{\partial t}\right)^{2}d\mu\right.\\
     & = \frac{\hat{\delta}-\delta_{1}}{(1-\delta_{1})(1-\hat{\delta})^{1/2}}\left.\int_{\hat{S}}\frac{(1-\hat{\delta})^{1/2}|\nabla \tau|}{\sqrt{1 + (1-\hat{\delta})|\nabla\tau|^{2}}}\left(\frac{\partial u}{\partial t}\right)^{2}dS\right.\\
     & = \frac{\hat{\delta}-\delta_{1}}{1-\delta_{1}}\left.\int_{\hat{S}}\frac{|\nabla \tau|}{\sqrt{1 + (1-\hat{\delta})|\nabla\tau|^{2}}}\left(\frac{\partial u}{\partial t}\right)^{2}dS\right..
  \end{align*}%
  \endgroup
  for some $\hat{\delta} > \delta_{1} > \delta$, and we fix $\delta_{1}$ so that
  $\hat{\delta} - \delta_{1} \geq \delta_{0}/2$.

 \noindent Using the above inequality, we can estimate 
  \begin{equation*}
    \begin{aligned}
      & \int_{G^{\hat{\delta}}}\rho\frac{\partial u}{\partial t}d\mu dt \leq 1/2\int_{G^{\hat{\delta}}}\rho^{2}d\mu dt 
      + 1/2\int_{G^{\hat{\delta}}}\left(\frac{\partial u}{\partial t}\right)^{2}d\mu dt \\
      & = 1/2\int_{0}^{\hat{T}}\left[\int_{\tilde{S}}\frac{(1-\tilde{\delta})^{1/2}|\nabla \tau|}{\sqrt{1 + (1-\tilde{\delta})|\nabla\tau|^{2}}}\rho^{2}dS\right] d\tilde{T}\\
      & + 1/2 \int_{0}^{\hat{T}}\left[\int_{\tilde{S}}\frac{(1-\tilde{\delta})^{1/2}|\nabla \tau|}{\sqrt{1 + (1-\tilde{\delta})|\nabla\tau|^{2}}}\left(\frac{\partial u}{\partial t}\right)^{2}dS\right]d\tilde{T}\\
      & \leq C\left(\int_{0}^{\hat{T}}R(\tilde{T})d\tilde{T} + \int_{0}^{\hat{T}}E(\tilde{T})d\tilde{T}\right).
    \end{aligned}
  \end{equation*}
  Here $C=C(\delta, \delta_{0})$ is sufficiently large,
  $$R(\tilde{T}) =
  \int_{\tilde{S}}\frac{(1-\tilde{\delta})^{1/2}|\nabla \tau|}{\sqrt{1 + (1-\tilde{\delta})|\nabla\tau|^{2}}}\rho^{2}dS,$$
  and we use the same notation for the lateral cone surface $\tilde{S}\vcentcolon=\{(t, x): t+(1-\tilde{\delta})^{1/2}\tau(x, p) =\tilde{T}\}$
  with $\tilde{\delta}\geq \hat{\delta}$ and, consequently, $\tilde{\delta} - \delta_{1} \geq \delta_{0}/2$.

  Combined with~\eqref{eqn:domain-dependence}, we obtain this inequality
  \begin{equation}\label{eqn:temp-eq}
    E(\hat{T}) - C\int_{0}^{\hat{T}}E(\tilde{T})d\tilde{T} \leq C_{1}
  \end{equation}
  with $C_{1}\vcentcolon=\max(1, C)\left(E(0) + \int_{0}^{\hat{T}}R(\tilde{T})d\tilde{T}\right)$.
  Furthermore, keeping in mind~\eqref{eqn:temp-eq}, we have
  \begin{equation*}
    \begin{aligned}
      & \frac{d}{d\tilde{T}}\left(e^{-C\tilde{T}}\int_{0}^{\tilde{T}}E(T_1)dT_1\right) = e^{-C\tilde{T}}\left(E(\tilde{T}) - C \int_{0}^{\tilde{T}}E(T_1)dT_1\right)\\
      & \leq C_{1}e^{-C\tilde{T}}.
    \end{aligned}
  \end{equation*}
  Integrating both sides from $0$ to $\hat{T}$, we obtain
  \begin{equation*}
    e^{-C\hat{T}}\int_{0}^{\hat{T}}E(T_1)dT_1 \leq C_{1}\frac{1 - e^{-C\hat{T}}}{C},
  \end{equation*}
  or
  \begin{equation*}
    C\int_{0}^{\hat{T}}E(T_1)dT_1 \leq C_{1}(e^{C\hat{T}} - 1).
  \end{equation*}
  Using this inequality in~\eqref{eqn:temp-eq}, we obtain
  \begin{equation*}
    E(\hat{T}) \leq C_{1}e^{C\hat{T}}.
  \end{equation*}
\end{proof}

\subsection{Domain of Dependence in the RCN of Infinity. Uniqueness of the Solutions.}
In this subsection we consider a Cauchy\text{--}Dirichlet problem in the cylinder
$Q_{T}\vcentcolon=[0, T] \times \overGi_{\tzau}$  with the boundary $S_{T}\vcentcolon=[0, T] \times \parGi_{\tzau}$
for any $\tzau > 0, T > 0$ and any neighborhood $\Gi$, i.e.,
\begin{subequations}\label{eqn:mixed-problem-type2}
  \begin{align}
    & \frac{\partial^{2} u}{\partial t^{2}} + Hu  = \rho(t, x), \rho \in L^{2}(Q_{T})\label{eqn:mixed-problem-type2:a}\\
    & u(0, x) = f(x), f \in W^{1, 2}_{0}(\Gi_{\tzau}), V^{1/2}_{+}f  \in L^{2}(\Gi_{\tzau}),  \label{eqn:mixed-problem-type2:b}\\
    & \frac{\partial u(0, x)}{\partial t} = g(x), g \in L^{2}(\Gi_{\tzau})\label{eqn:mixed-problem-type2:c}\\
    & \left.u(t, x)\right|_{S_{T}}, \left.\frac{\partial u(t, x)}{\partial t}\right|_{S_{T}} = 0.\label{eqn:mixed-problem-type2:d}      
  \end{align}
\end{subequations}
Similar to the previous subsection, we look for the solutions in this class

\noindent \textbf{Definition. Solutions } $\mU$ \textbf{ of}~\eqref{eqn:mixed-problem-type2}.

\noindent A function $u \in \mU$ iff it satisfies these conditions
\begin{subequations}\label{dfn:solutions-class-type2}
  \begin{align}
    & u \in W^{1,2}(Q_{T}),\label{dfn:solutions-class-type2:a}\\
    & \Vp^{1/2}u \in L^{2}(Q_{T}).\label{dfn:solutions-class-type2:b}
  \end{align}
\end{subequations}
The condition~\eqref{eqn:mixed-problem-type2:d} implies that we search for the solutions that vanish at the boundary
$\parGi_{\tzau}$ for a.e. $t\in [0, T]$. Note that the conditions~\eqref{dfn:solutions-class-type2} are less strict
than those in~\eqref{dfn:solutions-class}; this is due to the regularity of $\qm$ in $\Gi_{\tzau}$.

As in the previous subsection, we are ready to formulate
\begin{theorem}\label{thm:domain-dependence-type2}
  {\bf The Domain of Dependence at the RCN of Infinity.} Let $\Gi$ satisfy conditions~\ref{cond_c_2} and D, and
  let $u \in \mU$ be the solution of the mixed problem~\eqref{eqn:mixed-problem-type2}.
  Then, for a.e. $\hat{T} \leq T$, the domain of dependence equality
  \begin{equation}\label{eqn:domain-dependence-type2}
    E(\hat{T}) = E(0) + \int_{0}^{\hat{T}}\left(\int_{\Gi_{\tzau}}\rho(t, x)\frac{\partial u}{\partial t}d\mu\right) dt
  \end{equation}
  holds with
  \begin{equation}\label{eqn:energy-integral-type2}
    E(t) \vcentcolon= \frac{1}{2}\left.\int_{\Gi_{\tzau}}\left(\left(\frac{\partial u(t, x)}{\partial t}\right)^{2} + \left|\nabla u(t, x)\right|^{2} +
      V(x)u^{2}(t, x) \right) d\mu\right..
  \end{equation}
\end{theorem}
Note that the energy integral~\eqref{eqn:energy-integral-type2} is nonnegative because of the conditions~D and~\eqref{dfn:solutions-class-type2}, and because of Lemma~\ref{lemma:positivity}.
\begin{proof}
  The proof is a simplified version of the one given in Theorem~\ref{thm:domain-dependence};
  indeed, the Dirichlet condition~\eqref{dfn:solutions-class-type2:a} leads only to integrals at the
  cylinder bases in the expressions for $E(0)$ and $E(T)$, so we omit it here.
\end{proof}

Similar to Corollary~\ref{corol:unique-solutions}, we have this corollary of Theorem~\ref{thm:domain-dependence-type2}
\begin{corollary}
  \textbf{Uniqueness of Solutions in the RCN of Infinity.}
  A solution $u\in\mU$ of the mixed problem~\eqref{eqn:mixed-problem-type2} is unique in the domain of dependence $Q_{T}$.
\end{corollary}

Finally, the corollary below establishes the energy inequality in the domain of dependence $Q_{T}$.
\begin{corollary}
  \textbf{Energy Inequality for Solutions in the RCN of Infinity.}
  For a.e. $\hat{T} < T$ we have 
  \begin{equation*}
    E(\hat{T}) \leq C_1\left(E(0) + \int_{0}^{\hat{T}}\left(\int_{\Gi_{\tzau}}\rho^{2}(t, x) d\mu\right)dt\right)
  \end{equation*}
  with $C_1 = C_1(\Gi_{\tzau})$.
\end{corollary}

\section{Schr{\"o}dinger Wave Equation in the RCN. Existence of Solutions}\label{chap:existence}
In this section, we prove existence of solutions of
the mixed problem~\eqref{eqn:hyperbolic-equation} in neighborhoods of regular or singular points and existence of solutions
of the mixed problem~\eqref{eqn:mixed-problem-type2} in neighborhoods of infinity.
The proofs in this section are pretty standard,
so we provide either their sketch or references to the well-known results for the case of regular points.

Consider a static Dirichlet problem for a fixed $t\in [0, T]$
\begin{equation}\label{eqn:static-dirichlet}
  \begin{aligned}
    & \left.\int_{\bmTa_{p,\tzau}}\left( \langle\nabla u, \nabla\phi\rangle + Vu\phi\right) d\mu\right. = \int_{\bmTa_{p,\tzau}}\rho\phi d\mu\\
    & \text{for all }\phi \in C^\infty_{0}(\bmTa_{p,\tzau}),  u \in W_{0}^{1, 2}(\bmTa_{p,\tzau}) \text{ and } \rho, \qm^{1/2}u, \Vp^{1/2}u \in L^{2}(\bmTa_{p,\tzau}).
  \end{aligned}
\end{equation}
Note that $\Vp\in L^{1}_{\loc}(M)$, and measurable $\Vm$ satisfies inequality~\eqref{eqn:bilinear-form} in Lemma~\ref{lemma:positivity}
with $\qm$ replaced by $\Vm$, so the integrals in~\eqref{eqn:static-dirichlet} are well defined.
For the case of a RCN of singularity $\Gammaj$ test functions $\phi\in C^{\infty}_{0}(\bmTa_{\Gammaj, \tzau} \setminus \Gammaj)$.
  
We have the following:
\begin{theorem}\label{thm:static-dirichlet}
  \textbf{Existence of Static Dirichlet Solutions.}
  A solution $u$ of the problem \eqref{eqn:static-dirichlet} exists and unique in $\bmTa_{p,\tzau}$.
  We can extend it by zero to the entire manifold $M$. 
\end{theorem}
\begin{proof}
  As we have noted before, the domain $\bmTa_{p,\tzau}$ has the Lipschitz boundary, and,
  as it was shown in~\cite{RSMUP_1957__27__284_0}, for instance, the trace operator could be
  defined on $W^{1,2}(\bmTa_{p,\tzau})$, so that the Dirichlet condition
  $\left.u\right|_{\partial \bmTa_{p,\tzau}} = 0$ is well posed.

  Consider in $L^{2}(\bmTa_{p,\tzau})$ a Hilbert space $\mathcal{H}$ with the inner product defined by
  \begin{equation}\label{eqn:hilbert-space-dirichlet}
    \begin{aligned}
      & (u, v)_{\mathcal{H}} = \left.\int_{\bmTa_{p,\tzau}}\left( \langle\nabla u, \nabla v\rangle + \Vp uv -\Vm uv\right)d\mu\right.,\\
      & u, v \in W_{0}^{1, 2}(\bmTa_{p,\tzau}),\text{ with }  \Vp^{1/2}u,  \Vp^{1/2}v, \qm^{1/2}u,  \qm^{1/2}v \in L^{2}(\bmTa_{p,\tzau}), 
    \end{aligned}
  \end{equation}
  and $\mathcal{H}$ is the closure w.r.t.~the norm defined in~\eqref{eqn:hilbert-space-dirichlet}.

  Note that $\mathcal{H}$ is indeed the Hilbert space due to nonnegativity of the norm corresponding
  to~\eqref{eqn:hilbert-space-dirichlet} due to Lemma~\ref{lemma:positivity}; also
  $||u||_{\mathcal{H}} = 0$ iff $u=0$, since
  \begin{equation*}
    ||u||^{2}_{\mathcal{H}} \geq (1-\delta) \left.\int_{\bmTa_{p,\tzau}}|\nabla u|^{2}d\mu\right., 
  \end{equation*}
  so $u=0$ a.e. in $\bmTa_{p,\tzau}$ due to the Dirichlet boundary condition.
  We also claim that the closure w.r.t.~this norm exists - the norm corresponds to the quadratic
  form for the symmetric, densely defined, and non-negative operator~\eqref{eqn:schrodinger} in $L^{2}(\bmTa_{p,\tzau})$. 

  The right-hand side of~\eqref{eqn:static-dirichlet} is the linear bounded functional
  in $\mathcal{H}$; indeed,
  \begin{equation*}
    \begin{aligned}
      & \left.\int_{\bmTa_{p,\tzau}}\rho\phi d\mu\right. = (\rho, \phi)_{L^2(\bmTa_{p,\tzau)}} \leq ||\rho||_{L^2(\bmTa_{p,\tzau})} ||\phi||_{L^2(\bmTa_{p,\tzau})} \\
      & \leq C_1 ||\rho||_{L^2(\bmTa_{p,\tzau})} ||\nabla \phi||_{L^2(\bmTa_{p,\tzau})} \leq  C_1 ||\rho||_{L^2(\bmTa_{p,\tzau})} ||\phi||_{\mathcal{H}}.
    \end{aligned}
  \end{equation*}
  Here the second-to-last inequality is due to the Poincare inequality.
  
  The bilinear form $B$ on the left-hand side of~\eqref{eqn:static-dirichlet} is bounded
  \begin{equation*}
    |B[u,v]| \vcentcolon= |(u, v)_{\mathcal{H}}| \leq ||u||_{\mathcal{H}} ||v||_{\mathcal{H}},
  \end{equation*}
  therefore, by the Lax\text{--}Milgram theorem the solution of~\eqref{eqn:static-dirichlet}
  exists and is unique in $\mathcal{H}$.

  Owing to a well-known result by A. P. Calder\'on~\cite{key0143037m}, that any
  function $u \in W^{1,2}_0(\bmTa_{p,\tzau})$ on the Lipschitz domain could be extended by zero to
  the entire manifold $M$, the proof follows.
\end{proof}
We are ready to formulate the following:
\begin{theorem}\label{thm:existence-solutions}
  \textbf{Existence of Solutions.}
  \begin{itemize}
  \item A solution $u\in\mU$ of the Cauchy\text{--}Dirichlet problem~\eqref{eqn:hyperbolic-equation} defined in~\eqref{dfn:solutions-class}
    with Dirichlet boundary conditions on
    $[0, T] \times \partial \bmTa_{p,\tzau}([0, T] \times \partial \bmTa_{\Gammaj,\tzau})$ exists for any choice of
    the initial conditions~\eqref{eqn:hyperbolic-equation:b},~\eqref{eqn:hyperbolic-equation:c},~\eqref{eqn:hyperbolic-equation:d}, and
    the source function $\rho$.
  \item A solution $u\in\mU$ of the mixed problem~\eqref{eqn:mixed-problem-type2} defined in~\eqref{dfn:solutions-class-type2}
    with Dirichlet conditions on $[0, T] \times \parGi_{\tzau}$ exists for any choice of
    the initial conditions~\eqref{eqn:mixed-problem-type2:b},~\eqref{eqn:mixed-problem-type2:c},~\eqref{eqn:mixed-problem-type2:d}, and
    the source function $\rho$.
  \end{itemize}
\end{theorem}
\begin{proof}
  A complete proof can be found in Theorem 9.2 in~\cite{Wilcox1962}, and it is essentially based on the existence of solutions
  of the static Dirichlet problem established in Theorem~\ref{thm:static-dirichlet}.
\end{proof}

\section{Finite Propagation Speed of Solutions of the Schr{\"o}dindger Wave Equation.
  Essential Self-Adjointness of the Schr{\"o}dinger Operator.}\label{chap:gfps}

We defined classes~\eqref{dfn:solutions-class} and~\eqref{dfn:solutions-class-type2}
of solutions $\mU$ of the mixed problems~\eqref{eqn:hyperbolic-equation} and~\eqref{eqn:mixed-problem-type2},
and we proved their existence and uniqueness in the domains of dependence defined
in Theorems~\ref{thm:domain-dependence} and~\ref{thm:domain-dependence-type2}. 
Note that these domains are defined in possibly small neighborhoods of either regular or singular
points and for the entire neighborhoods of infinity. Our aim is to extend the results of the previous sections to the entire $M$.

Consider the minimal operator $H_{0}$ defined by the expression in~\eqref{eqn:schrodinger} with the domain
$D(H_{0}) = \{u \in L^{2}(M): u \in W_{0}^{1,2}(M), \Vp^{1/2}u, \qm^{1/2}u \in L^{2}(M)\}$.

Now, let us consider the Cauchy problem in $L^{2}(M)$ for its adjoint operator $H_{0}^{*}$
\begin{equation}\label{eqn:hyperbolic-equation-main}
  \begin{aligned}
    & \frac{d^{2} u}{d t^{2}} + H_{0}^{*}u = 0,\\
    & u(t) \in C^{2}([0, T), D(H_{0}^{*})),\\
    & u(0) = f, \frac{du}{dt}(0) = 0, f \in D(H_{0}^{*}),
  \end{aligned}
\end{equation}
and we are going to research the uniqueness of its strong solutions $u$ for some $T > 0$.
A solution here is twice continuously differentiable vector function with values in $D(H_{0}^{*})$
satisfying the equation and initial conditions~\eqref{eqn:hyperbolic-equation-main}.

In the course of investigating the uniqueness of the solutions for problem~\eqref{eqn:hyperbolic-equation-main},
we establish sufficient conditions for the {\it global finite propagation speed (GFPS)} of the solutions of
the Cauchy problem~\eqref{eqn:hyperbolic-equation-main}, i.e., we want to show that the solution value $u(\tnot, \xnot)$
is uniquely defined by the initial conditions on some compact subset $G \Subset M$ depending on both $\tnot$ and $\xnot$.

A solution $u$ in~\eqref{eqn:hyperbolic-equation-main} is equivalent to a solution of the equation
in the weak form
\begin{equation}\label{eqn:hyperbolic-equation-main-weak}
  \left(\frac{d^{2} u}{d t^{2}}, \phi\right) + \left(u, H_{0}\phi\right) = 0\ \text{for any }\phi \in D(H_{0}),
\end{equation}
and we use this equation while establishing the existence and uniqueness of its solutions.

As noted in the Berezansky theorem - see Theorem 6.2 in~\cite{berezansky1978self}-
if a solution is unique for $t \in [0, T)$ for some $T > 0$, and if $H_{0}$ is semibounded from below,
then it is essentially self-adjoint, and we use this method of hyperbolic
equations to prove the essential self-adjointness of $H_{0}$. 

Very notable results related to this method of hyperbolic equations are 
those of P.~R.~Chernoff~\cite{CHERNOFF1973401}
and T.~Kato~\cite{KATO1973415},
and B.~M.~Levitan~\cite{Lev61}.
Chernoff, in particular, considers the Schr{\"o}dinger
operator semibounded from below with smooth potential,
and he proves the finite propagation speed of the solutions
of the wave equation with the initial conditions in $C_{0}^{\infty}(M)$ and the essential self-adjointness
of this operator and its powers. Kato extended the results of Chernoff
to not semibounded from below the Schr{\"o}dinger operators in $\RRn$,
and Levitan provided another proof of the Sears theorem. 

A.~Chumak~\cite{Chumak1973} explicitly constructed the domain of the dependence of solutions
of~\eqref{eqn:hyperbolic-equation-main} for the Beltrami\text{--}Laplace operator, and the GFPS property was
a simple corollary of this construction. The author used the uniqueness of solutions of~\eqref{eqn:hyperbolic-equation-main}
to show the essential self-adjointness of the Beltrami\text{--}Laplace operator.

Yu.~B.~Orochko~\cite{Oro82}, also referred in survey~\cite{Oro88}, studied the GFPS for more
general second-order elliptic operators in $\RRn$ with singular potentials
and not bounded at infinity. Using the method of hyperbolic equations, he obtained sufficient
conditions of the essential self-adjointness of the Schr{\"o}dinger operator comparable with those
defined by P.~Hartman in~\cite{hartman1951}
and by R.~Ismagilov in~\cite{Ism62} for $\mathbb{R}$, and by M.~Gimadislamov~\cite{Gim68}
who extended Ismagilov's results to $\RRn, n \geq 1.$
Note that both Ismagilov and Gimadislamov considerelliptic operators of any even order.

The Ismagilov criterion considers the behavior of a regular potential $\Vm$ in some closed bounded concentric
layers going out to infinity. If the potential does not decrease rapidly in these layers, and if the layers
are sufficiently wide, then the Schr{\"o}dinger operator is essentially self-adjoint. Note that the proofs
in~\cite{Ism62} and ~\cite{Gim68} used quadratic form estimates for the maximal Schr{\"o}dinger operator.

Orochko proved that the conditions of the Ismagilov criterion
guarantee the GFPS for the solutions of the Cauchy problem~\eqref{eqn:hyperbolic-equation-main}; thus, his results
demonstrated the power of the hyperbolic equation method and its physics essence.

Orochko~\cite{Oro88} and Chernoff~\cite{chernoff1977schrodinger} extended their results
to singular potentials having small relative bounds w.r.t.~the Laplacian, which belong to the Kato class, for example.

F.~S.~Rofe-Beketov~\cite{Rof70} further improved the results of Ismagilov in~\cite{Ism62} for regular potentials;
Rofe-Beketov introduced a function whose norm of gradient does not grow too rapidly compared
with $\qm^{-1/2}$ in layer sets of this function. In~\cite{Rof70}, such a function was constructed when the conditions in~\cite{Ism62} were satisfied.

Oliynyk~\cite{Oleinik1994} extended the results of~\cite{Rof70} to Riemannian manifolds without boundary and showed that the conditions on the magnitude of the gradient of the above function in~\cite{Rof70}
imply classical completeness of the potential~\eqref{eqn:classical-completeness}.

In the present work, we want to extend the results in~\cite{Oleinik1994} via the hyperbolic equation method; intuitively,
the classical completeness of the potential, i.e., the impossibility for a classical particle to reach
infinity in finite time, implies a restriction on the propagation speed, and we wanted and hoped to connect
the classical completeness of the potential with the GFPS via the hyperbolic equation method.

Both Orochko~\cite{Oro88} and Chumak~\cite{Chumak1973} explicitly described how solutions
of~\eqref{eqn:hyperbolic-equation-main} with compactly supported $f$ in the initial conditions propagate with time.
Chumak showed that for the case of the Laplace\text{--}Beltrami operator,
the characteristic cone with the vertex at $(\tnot, \xnot)$ is locally spanned by curves
$\{(t, x) \in \mathbb{R}\times M: \tnot - t = s(\xnot, x)\}$, where $s(\xnot, x)$ is the distance along
a geodesic curve connecting $x$ and $\xnot$.

Orochko introduced the notion of \textit{consistent triples} to track the propagation speed,
which largely depends on the growth of the eigenvalues of the main symbol of the corresponding second-order elliptic
operator in the divergent form in $L^{2}(\RRn)$. The potential $\Vm$ is assumed to belong to $K_{n}(\RRn)$,
a uniform Kato class, i.e., it is a small perturbation of the main symbol operator.

Therefore, in the spirit of these works, we also want to estimate a propagation speed on the basis of the results
in Section~\ref{chap:uniqueness}, with a focus on singularities of the potential $\Vm$ and on infinity of $M$. 

In the sequel, we use the IMS Localization Formula stated in~\cite{cycon1987schrodinger}, $\S 3.1$, namely,
suppose that we have a partition of unity of $M$ with functions $J_{\alpha}, \alpha \in \mathcal{A}$
indexed by a set $\mathcal{A}$ satisfying these conditions

\begin{enumerate}[label=(\roman*)]
\item $0 \leq J_{\alpha}(x)\leq 1$ for all $x \in M$;
\item $\sum_{\alpha}J_{\alpha}^{2}(x) = 1$ for all $x\in M$;
\item The family $J_{\alpha}$ is locally finite, i.e. for each compact set $K \subset M$ we have $J_{\alpha} = 0$ for all
  $a \in \mathcal{A}$ with the exception of finitely many indices;
\item $J_{\alpha}\in \Lip^{1,0}_{\loc}(M)$;
\item $\sup_{x\in M}\sum_{\alpha \in \mathcal{A}}|\nabla J_{\alpha}(x)|^{2} < \infty$.
\end{enumerate}
Then the IMS Localization Formula states that for any $\phi \in D(H_{0})$ we have
\begin{equation}\label{eqn:ims-formula}
  (H_{0}\phi, \phi) = \sum_{\alpha\in\mathcal{A}}\left(H_{0}(J_{\alpha}\phi), J_{\alpha}\phi\right)
  - \left(\sum_{\alpha\in\mathcal{A}}|\nabla J_{\alpha}(x)|^{2}\phi, \phi\right).  
\end{equation}
Note that $J_{\alpha}\phi \in D(H_{0})$ and the sums in~\eqref{eqn:ims-formula}
have a finite number of terms due to the local finiteness
of the partition of unity. The last term in~\eqref{eqn:ims-formula}
is called {\it the error of localization,} and it is locally bounded.

The following lemma sets sufficient conditions on the semiboundedness of the Schr{\"o}dinger operator~\eqref{eqn:schrodinger}
in $L^{2}(M)$, and the existence of certain open covers on $M$ will be an important part of these conditions.
\begin{lemma}\label{lemma:semibounded}
  Suppose that $M$ has an admissible open cover satisfying Conditions C and D.
  Then, the minimal operator $H_{0}$ is semibounded from below.
\end{lemma}
\begin{proof}
  Let us consider an element $\mTa_{p_{\alphai}, \tau_{\alphai}}$ of the admissible cover with $\tau_{\alphai}\geq\tzau$,
  and let us define a piecewise differentiable cutoff function $J_{\tau_{\alphai}}:[0, \infty)\to[0, 1]$ by
  
  $\begin{aligned}
    J_{\tau_{\alphai}}(t) & = 1, \text{if }t\leq \tau_{\alphai} - \teps\\
    & = 0, \text{if } t\geq \tau_{\alphai} - \teps/2\\
    & = \text{linear for }t\in [\tau_{\alphai} - \teps, \tau_{\alphai} - \teps/2] 
  \end{aligned}$

  \noindent with $\teps$ defined in condition~\ref{cond_c_4}.
  Then, we define a Lipschitz cutoff function such that
  
  $\begin{aligned}
    & J_{p_{\alphai}, \tau_{\alphai}}(x) = J_{\tau_{\alphai}}(\tau(p_{\alphai}, x)), x\in M,\ \text{so that}\\
    & \text{supp}(J_{p_{\alphai}, \tau_{\alphai}}) = \bmTa_{p_{\alphai}, \tau_{\alphai} - \teps/2}.
  \end{aligned}$

  \noindent Note that, according to~\ref{cond_c_4}, for any $x\in M$, there is an open cover $\mTa_{p_{\alphai}, \tau_{\alphai}}$ with
  $J_{p_{\alphai}, \tau_{\alphai}}(x) = 1$.

  \noindent Similarly, we define Lipschitz cutoff functions
  for singularity neighborhoods $\mTa_{\Gammaj, \tau_{j}}$ and for
  the neighborhoods $\mTa_{k_{\alphai},\alphai}$ defined in condition~\ref{cond_c_3}.

  \noindent For neighborhoods of infinity $\overGi$ we define the cutoff function by
  \begin{equation}\label{eqn:cut-off-type2}
    \begin{aligned}
      J_{\overGi, \teps}(x) & = 1 - J_{\teps}(\tau(\parGi, x)), x\in \overGi\\
      & = 0\ \text{outside of } \overGi.
    \end{aligned}
  \end{equation}
  
  Note that the gradients of all the cutoff functions above have the same upper bound - due to the boundedness of $|\nabla \tau|$ in~\eqref{eqn:limited-grad-tau}
  in neighborhoods of regular and singular points - and since they are locally finite due to conditions~\ref{cond_c_3},
  then, owing to the condition~\ref{cond_c_4},
  we can renormalize these functions to obtain the subordinate partition of unity
  on $M$, and the renormalized partition functions have a uniformly bounded gradient.

  Thus, we can conclude that the error term 
  $\sup_{x\in M}\sum_{\beta}|\nabla J_{\beta}(x)|^{2} < \infty$ for a.e. $x\in M$,
  where $\beta$ are indices for the set of all partitions of the unity functions defined above.

  The semiboundedness of the operator $H_{0}$ follows from its nonnegativity
  for each $J_{\beta}u$ with $u \in D(H_{0})$ due to Lemma~\ref{lemma:positivity} and
  the boundedness of the error term above in the IMS Localization Formula~\eqref{eqn:ims-formula}.
\end{proof}

The next statement proves the existence of solutions~\eqref{eqn:hyperbolic-equation-main} for
any self-adjoint extension of $H_{0}$; see Theorem 6.2 in~\cite{berezansky1978self} for a more extended formulation.
\begin{proposition}\label{prop-existence-solutions}
  For any self-adjoint extension $H_{0}\subset H_{1} = H_{1}^{*}$ of $H_{0}$ and
  any initial condition $f\in D(H_{0})$ there exists a solution $u$
  of~\eqref{eqn:hyperbolic-equation-main}.
\end{proposition}
We study the uniqueness of the solutions of~\eqref{eqn:hyperbolic-equation-main}, i.e. for $f = 0$, so it is sufficient
for us just to require that $f\in D(H_{0})$.
\begin{proof}
  Since the operator $H_{0}$ is densely defined and semibounded from below due to Lemma~\ref{lemma:semibounded}, it has self-adjoint extensions, so the operator $H_{1}$ exists and bounded from below.

  Denote by $E_{1}$ its partition of unity, and note that the function
  \begin{equation*}
    u(t) = \int_{c}^{\infty}\cos(\sqrt{\lambda}t)dE_{1}f,\ c > -\infty,\ t\in[0, \infty)
  \end{equation*}
  is twice continuously differentiable in $t$ because of the integral
  \begin{equation*}
    \int_{c}^{\infty}\lambda^{2}d(E_{1}f, f) < \infty,
  \end{equation*}
  and we can easily verify that it is the solution of~\eqref{eqn:hyperbolic-equation-main}
  by checking the identity in~\eqref{eqn:hyperbolic-equation-main-weak}
  with any $\phi\in D(H_{0})$.
\end{proof}

We search for solutions of the Cauchy problem~\eqref{eqn:hyperbolic-equation-main} in the domain $D(H_{0}^{*})$, and
we investigate this domain more closely.
\begin{lemma}\label{lemma:operator-domain}
  Suppose that $M$ has an admissible cover satisfying conditions C and D.
  Then the domain $D(H_{0}^{*})\subset\{u \in L^{2}(M)|\  u \in W^{1,2}_{\loc}(M), \Vp^{1/2}u, \qm^{1/2}u \in L_{\loc}^{2}(M)\}.$
\end{lemma}
\begin{proof}
  As in Theorem~\ref{thm:static-dirichlet}, let us consider any element $\mathcal{V}_{\beta}$
  of the admissible cover defined above, and for any $f\in L^{2}(M)$ and any $v\in D(H_{0})$
  with $\text{supp}(v) \subset \mathcal{V}_{\beta}$, we consider this Dirichlet problem
  \begin{equation}\label{eq:od-beta}
    (H_{0}v, g_{\beta}) = (v, f_{\beta})
  \end{equation}
  with $f_{\beta} = J_{\beta}f$, and $J_{\beta}$ is the element of the partition of unity
  with $\text{supp}(J_{\beta})\subset \mathcal{V}_{\beta}$.
  In Theorem~\ref{thm:static-dirichlet}, we establish the existence of solutions
  $g_{\beta}\in W_{0}^{1,2}(\mathcal{V}_{\beta}), \Vp^{1/2}g_{\beta}, \qm^{1/2}g_{\beta} \in L^{2}(\mathcal{V}_{\beta})$, and $g_{\beta}$ can
  be extended by zero to the entire $M$.

  Now, it is easy to see that for any fixed $v$ defined above by adding the left and right-hand sides of~\eqref{eq:od-beta}, we have the following equation:
  \begin{equation}\label{eq:od-all-beta}
    (H_{0}v, g) = (v, f)
  \end{equation}
  with $g \vcentcolon= \sum_{\beta}g_{\beta}$.
  We can utilize the partition of the unity property to extend the equality~\eqref{eq:od-all-beta}
  to all $v\in D(H_{0})$.

  It is clear that for any compact $K\subset M$ in the definition of $g$
  we can find a finite cover $\mathcal{V}_{\beta}$ of $K$,
  so that $g$ belongs to the space stated in the lemma formulation.
\end{proof}

Now, we turn to the study of the GFPS, and we investigate how the domain of dependence
for the Cauchy problem~\eqref{eqn:hyperbolic-equation-main} changes with time; the
next three lemmas and a theorem help us better understand
the nature of propagation in the intersection
of neighborhoods of regular, singular, or infinity points or domains.

\begin{lemma}\label{lemma:lemma-domain-dependence-type2}
  \textbf{Domain of Dependence for Points in Neighborhoods of Infinity.}
  Under the conditions of Theorem~\ref{thm:domain-dependence-type2}
  for any $\tnot > 0$ and $\xnot \in \overGi$ its solution $u\in\mU$ depends on the initial conditions
  defined on some compact subdomain in $\Gi$.
\end{lemma}
\begin{proof}
  Define $\tzau$ such that $\tau(\parGi, \xnot) = \tzau$, then select $\toau, \ttwo > 0$
  such that $\toau < \tzau < \ttwo$ and with a small $\delta^{\prime} > 0$ such that
  $\ttwo - \toau< \delta^{\prime}.$
  Then according to the Theorem~\ref{thm:domain-dependence-type2}, $u(\tnot, \xnot)$ depends
  on the initial conditions in $\overGi_{[\toau, \ttwo]}$ and Dirichlet boundary conditions at
  $[0, \tnot] \times \parGi_{[\toau, \ttwo]}$.
\end{proof}

\begin{lemma}\label{lemma:domain-dependence-singularity}
  \textbf{Domain of Dependence for Points in both Singularity and Regular Neighborhoods.}
  With the conditions~\ref{cond_c_1},~\ref{cond_c_3},~\ref{cond_c_4}, D,
  and the conditions of Theorem~\ref{thm:domain-dependence}, there exists $\tilde{T} > 0$ such that
  for any point $(\tnot, \xnot),\ \xnot \in \mTa_{\Gammajk, \taujk}$ for some $k = 1, 2,\ \dots$,
  there exists a finite set of open covers
  $\mTa_{\Gammajk, \taujk}$ and/or $\mTa_{p_{\alphai, \tau_{\alphai}}}, i = 1,\ 2,\ ...$ of $\xnot$ such that
  $u(\tnot, \xnot)$ depends on the initial conditions in
  $\{(t, x): t = \tnot-\tilde{T},
  x\in (\cup_{k}\mTa_{\Gammajk, \taujk})\cup(\cup_{i} \mTa_{p_{\alphai, \tau_{\alphai}}})\}$.
\end{lemma}
\begin{proof}
  Owing to Condition~\ref{cond_c_3}, we may find a finite set of not more than $m$ open covers $\xnot$ such that $\xnot \in (\cap_{k}\mTa_{\Gammajk, \taujk}) \cap \left(\cap_{i}\mTa_{p_{\alphai}, \tau_{\alphai}}\right)$.
  Our lemma states that we can find a subset of this cover depending on $\xnot$, so that the value of $u(\tnot, \xnot)$ depends only on the initial conditions in the above set with a fixed $\tilde{T}$.

  Define $\tilde{T} = (1-\delta)^{1/2}\teps/2$, where $\delta$ is defined in~\ref{lemma:positivity},
  and, according to Condition~D, it is the same for all covers;
  the parameter $\teps$ is defined in Condition~\ref{cond_c_4}.
  
  For singular covers of $\xnot$, let us consider a mixed problem similar
  to~\eqref{eqn:hyperbolic-equation} in the cylinder
  $[\tnot - \tilde{T}, \tnot - \tilde{T} + T_{j_{k}}]\times \mTa_{\Gammajk, \taujk}$
  with $T_{j_{k}} \vcentcolon= (1-\delta)^{1/2}\taujk$
  \begin{equation}\label{eqn:partition-cauchy-singular}
    \begin{aligned}
      & \frac{\partial^{2} u_{\Gammajk, \taujk}}{\partial t^{2}} + Hu_{\Gammajk, \taujk}  = J_{\Gammajk, \taujk}(x)\rho(t, x)\\
      & u_{\Gammajk, \taujk}(\tnot - \tilde{T}, x) = J_{\Gammajk, \taujk}(x)f(x),\\
      & \frac{\partial u_{\Gammajk, \taujk}(\tnot - \tilde{T}, x)}{\partial t} = J_{\Gammajk, \taujk}(x)g(x),\\
      & u_{\Gammajk, \taujk}(t, x), \left.\frac{\partial u_{\Gammajk, \taujk}(t, x)}{\partial t}\right|_{\partial \bmTa_{\Gammajk, \taujk}} = 0
    \end{aligned}
  \end{equation}
  where the partition of the unity function $J_{\Gammajk, \taujk}$ is defined in Lemma~\ref{lemma:semibounded},
  so the solution $u_{\Gammajk, \taujk}\in\mU$ exists and is unique in the defined dependency cone.
  It is clear from the definition that $J_{\Gammajk, \taujk}(x) = 0$ in a small neighborhood of $\xnot$, if
  $\tau\left(\xnot, \partial{\bmTa_{\Gammajk, \taujk}}\right) < \teps/2$; in that case, the solution
  $u_{\Gammajk, \taujk} = 0$ in a neighborhood of $(\tnot, \xnot)$.

  Similarly, for the regular covers of $\xnot$, we consider this mixed problem
  in the cylinders $[\tnot - \tilde{T}, \tnot - \tilde{T} + T_{\alphai}] \times \bmTa_{p_{\alphai, \tau_{\alphai}}}$
  with $T_{\alphai} \vcentcolon= (1-\delta)^{1/2}\tau_{\alphai}$
  \begin{equation}\label{eqn:partition-cauchy}
    \begin{aligned}
      & \frac{\partial^{2} u_{p_{\alphai, \tau_{\alphai}}}}{\partial t^{2}} + Hu_{p_{\alphai, \tau_{\alphai}}}  = J_{p_{\alphai, \tau_{\alphai}}}(x)\rho(t, x)\\
      & u_{p_{\alphai, \tau_{\alphai}}}(\tnot - \tilde{T}, x) = J_{p_{\alphai, \tau_{\alphai}}}(x)f(x),\\
      & \frac{\partial u_{p_{\alphai, \tau_{\alphai}}}(\tnot - \tilde{T}, x)}{\partial t} = J_{p_{\alphai, \tau_{\alphai}}}(x)g(x),\\
      & u_{p_{\alphai, \tau_{\alphai}}}(t, x), \left.\frac{\partial u_{p_{\alphai, \tau_{\alphai}}}(t, x)}{\partial t}\right|_{\partial \bmTa_{p_{\alphai, \tau_{\alphai}}}} = 0
    \end{aligned}
  \end{equation}
  Additionally, $J_{p_{\alphai, \tau_{\alphai}}}(x) = 0$ in a small neighborhood of $\xnot$, if
  $\tau\left(\xnot, \partial{\mTa_{p_{\alphai, \tau_{\alphai}}}}\right) < \teps/2$; in that case, the solution
  $u_{p_{\alphai, \tau_{\alphai}}} = 0$ in a neighborhood of $(\tnot, \xnot)$.

  From problems~\eqref{eqn:partition-cauchy-singular} and~\eqref{eqn:partition-cauchy}, we can conclude that the equality
  \begin{equation*}
    u(\tnot, \xnot) = \sum_{i}u_{p_{\alphai, \tau_{\alphai}}}(\tnot, \xnot) + \sum_{k}u_{\Gammajk, \taujk}(\tnot, \xnot)
  \end{equation*}
  holds, where the sums are taken over only those covers
  of $\xnot$ where $\tau(\xnot, \partial \bmTa_{p_{\alphai, \tau_{\alphai}}}) \geq \teps/2$
  and $\tau(\xnot, \partial\bmTa_{\Gammajk, \taujk}) \geq \teps/2$, so the solution at $(\tnot, \xnot)$
  depends on the values of the solutions of both problems above on a finite set of covers with the same $t = \tnot - \tilde{T} = \tnot - (1-\delta)^{1/2}\teps/2$.
\end{proof}

\begin{lemma}\label{lemma:interaction-with-infinity}
  \textbf{Domain of Dependence for Points in Vicinity of Neighborhood of Infinity.}
  With the conditions~\ref{cond_c_1},~\ref{cond_c_2},~\ref{cond_c_3},~\ref{cond_c_4}, D,
  and the conditions of the Theorem~\ref{thm:domain-dependence},
  for any $(\tnot, \xnot)\in \mathbb{R}\times M$ and open covers $\mTa_{p_{\alphai, \tau_{\alphai}}}, i = 1,\dots$
  and $\mTa_{\Gammajk, \taujk}, k = 1,\dots$ of regular or singular points of the potential $\qm$
  such that $\xnot\in \mTa_{p_{\alphai, \tau_{\alphai}}}, \mTa_{\Gammajk, \taujk}$ with
  $\mTa_{p_{\alphai, \tau_{\alphai}}}, \mTa_{\Gammajk, \taujk}\cap \overGi \ne \emptyset$ for some fixed $\Gi$,
  there exists $\tzau(\xnot) > 0$ such that its domain of dependence
  $K_{(\tnot,\xnot)}\cap ((-\infty, \tnot] \times \overGi) \subset ((-\infty, \tnot] \times \overGi_{\tzau})$
  for any $\tnot > 0$.
\end{lemma}
\begin{proof}
  As in Lemma~\ref{lemma:domain-dependence-singularity}, we consider these mixed problems in the vicinity
  of $\Gi$ for the regular covers of $\xnot$
  in the cylinders $[\tnot - \tilde{T}, \tnot - \tilde{T} + T_{\alphai}] \times \bmTa_{p_{\alphai}, \tau_{\alphai}}$
  with $T_{\alphai} \vcentcolon= (1-\delta)^{1/2}\tau_{\alphai}$
  \begin{equation}\label{eqn:partition-cauchy-infty}
    \begin{aligned}
      & \frac{\partial^{2} u_{p_{\alphai, \tau_{\alphai}}}}{\partial t^{2}} + Hu_{p_{\alphai, \tau_{\alphai}}}  = J_{\overGi, \teps}(x)J_{p_{\alphai, \tau_{\alphai}}}(x)\rho(t, x)\\
      & u_{p_{\alphai, \tau_{\alphai}}}(\tnot - \tilde{T}, x) = J_{\overGi, \teps}(x)J_{p_{\alphai, \tau_{\alphai}}}(x)f(x),\\
      & \frac{\partial u_{p_{\alphai, \tau_{\alphai}}}(\tnot - \tilde{T}, x)}{\partial t} = J_{\overGi, \teps}(x)J_{p_{\alphai, \tau_{\alphai}}}(x)g(x),\\
      & u_{p_{\alphai, \tau_{\alphai}}}(t, x), \left.\frac{\partial u_{p_{\alphai, \tau_{\alphai}}}(t, x)}{\partial t}\right|_{\partial \bmTa_{p_{\alphai, \tau_{\alphai}}}} = 0
    \end{aligned}
  \end{equation}
  and
  \begin{equation}\label{eqn:partition-cauchy-infty-compl}
    \begin{aligned}
      & \frac{\partial^{2} u_{p_{\alphai, \tau_{\alphai}}}}{\partial t^{2}} + Hu_{p_{\alphai, \tau_{\alphai}}}  = (1 - J_{\overGi, \teps}(x))J_{p_{\alphai, \tau_{\alphai}}}(x)\rho(t, x)\\
      & u_{p_{\alphai, \tau_{\alphai}}}(\tnot - \tilde{T}, x) = (1 - J_{\overGi, \teps}(x))J_{p_{\alphai, \tau_{\alphai}}}(x)f(x),\\
      & \frac{\partial u_{p_{\alphai, \tau_{\alphai}}}(\tnot - \tilde{T}, x)}{\partial t} = (1 - J_{\overGi, \teps}(x))J_{p_{\alphai, \tau_{\alphai}}}(x)g(x),\\
      & u_{p_{\alphai, \tau_{\alphai}}}(t, x), \left.\frac{\partial u_{p_{\alphai, \tau_{\alphai}}}(t, x)}{\partial t}\right|_{\partial \bmTa_{p_{\alphai, \tau_{\alphai}}}} = 0.
    \end{aligned}
  \end{equation}
  Mixed problems in the neighborhoods of singular points $\bmTa_{\Gammajk, \taujk}$ are defined similarly
  to~\eqref{eqn:partition-cauchy-infty} and~\eqref{eqn:partition-cauchy-infty-compl}.

  From the definition~\eqref{eqn:cut-off-type2} of $J_{\overGi, \teps}$, it is clear that
  $\text{supp}(u_{p_{\alphai, \tau_{\alphai}}}(\tnot-\tilde{T}, \cdot))\cap\overGi \subset\overGi_{\tau_{\alphai}} $
  of~\eqref{eqn:partition-cauchy-infty}, and, as for the mixed problem in~\eqref{eqn:mixed-problem-type2},
  this solution is uniquely defined in $[T, \tnot-\tilde{T}) \times \overGi_{\tau_{\alphai}}$ for any $T < \tnot-\tilde{T}.$

  For problem~\eqref{eqn:partition-cauchy-infty-compl} and definition~\eqref{eqn:cut-off-type2}, $\text{supp}(u_{p_{\alphai, \tau_{\alphai}}}(\tnot-\tilde{T}, \cdot))\cap\overGi \subset\overGi_{\teps}$,
  and we define
  $$\tzau \vcentcolon= \sup_{x\in\overGi_{\teps}}\left(\{\tau_{\alphai}: x\in \mTa_{p_{\alphai}, \tau_{\alphai}}\}
    \cup\{\taujk: x \in \mTa_{\Gammajk, \taujk}\}\cup\{\teps\}\right),$$
  and it is a finite value due to the finiteness of the cover $\overGi_{\teps}$,
  as stated in condition~\ref{cond_c_3}. 
\end{proof}

We are now ready to combine Lemmas~\ref{lemma:lemma-domain-dependence-type2},
~\ref{lemma:domain-dependence-singularity}, and~\ref{lemma:interaction-with-infinity} into the following
\begin{theorem}\label{thm:gfps}
  \textbf{Global Final Propagation Speed.}
  As stated previously in Lemma~\ref{lemma:lemma-domain-dependence-type2}
  and in Theorem~\ref{thm:domain-dependence-type2},
  assume that RCNs of infinity $\overGi, i = 1,\ ...$ satisfy conditions~\ref{cond_c_2} and D. 

  Similarly, as formulated 
  in Lemmas~\ref{lemma:domain-dependence-singularity} and~\ref{lemma:interaction-with-infinity},
  and in Theorem~\ref{thm:domain-dependence},
  for the regular and singular points, 
  we assume that the conditions~\ref{cond_c_1},~\ref{cond_c_2},~\ref{cond_c_3},~\ref{cond_c_4}, and D are satisfied.
  
  The solution of $u\in\mU$ of the
  Cauchy problem~\eqref{eqn:hyperbolic-equation-main} for any $(\tnot, \xnot)$
  uniquely depends on the initial conditions
  on some compact set $K_{(\tnot, \xnot)} \Subset M$.
\end{theorem}
\begin{proof}
  It is clear from the conclusions of Lemmas~\ref{lemma:lemma-domain-dependence-type2},~\ref{lemma:domain-dependence-singularity},
  and~\ref{lemma:interaction-with-infinity} that we can reach the initial conditions for $t=0$ in $[\tnot/\tilde{T}] + 1$ steps, and
  the resulting domain of dependence is bounded. 
\end{proof}

We are now ready to formulate 
\begin{theorem}\label{thm:essential-selfadjoinness}
  \textbf{Essential Self-Adjointness of $H_{0}$.}
  As in Lemma~\ref{lemma:semibounded}, suppose that $M$ has an admissible open cover satisfying Conditions C and D,
  and, as in Lemma~\ref{lemma:lemma-domain-dependence-type2} and Theorem~\ref{thm:domain-dependence-type2},
  RCNs of infinity $\overGi, i = 1,\ ...$ satisfy conditions~\ref{cond_c_2} and D. 

  As stated in Lemmas~\ref{lemma:domain-dependence-singularity} and~\ref{lemma:interaction-with-infinity},
  and in Theorem~\ref{thm:domain-dependence},
  for the regular and singular points of $M$,
  we assume that the conditions~\ref{cond_c_1},~\ref{cond_c_2},~\ref{cond_c_3},~\ref{cond_c_4}, and D are satisfied.

  The operator $H_{0}$ is essentially self-adjoint.
\end{theorem}
\begin{proof}
  In Lemma~\ref{lemma:semibounded}, we establish the semiboundedness of $H_{0}$, the Proposition~\ref{prop-existence-solutions} proves the existence of solutions of~\eqref{eqn:hyperbolic-equation-main}. In Lemma~\ref{lemma:operator-domain}, we prove an important inclusion for the domain $D(H_{0}^{*})$, and Lemmas~\ref{lemma:lemma-domain-dependence-type2},~\ref{lemma:domain-dependence-singularity}, and~\ref{lemma:interaction-with-infinity}, together
  with the Theorem~\ref{thm:gfps}, establish the uniqueness of solutions from the classes $\mU$ defined in~\eqref{dfn:solutions-class} and~\eqref{dfn:solutions-class-type2}, which, in turn, imply the uniqueness of solutions $u\in C^{2}\left([0, \tilde{T}), D(H_{0}^{*})\right)$ of the Cauchy problem~\eqref{eqn:hyperbolic-equation-main} for $\tilde{T}$ defined in Lemma~\ref{lemma:domain-dependence-singularity}; in fact, the uniqueness is true for any interval $[0, T), T > 0$.

  Thus, according to Theorem~6.2~\cite{berezansky1978self}, the essential self-adjointness of $H_{0}$ follows.
\end{proof}

Notably, we can formulate a very simple
\begin{corollary}\label{cor:operator-minorant}
  Assume that the hypotheses of Theorem~\ref{thm:essential-selfadjoinness}
  and the conditions of Lemmas~\ref{lemma:semibounded} and~\ref{lemma:operator-domain} are satisfied.
  Suppose that we can find a small $0\leq \tilde{\delta} < 1$
  satisfying $\tilde{\delta} + \delta < 1$ with $\delta$
  defined in Lemma~\ref{lemma:positivity} for all points on $M$, such that the operator inequality holds
  \begin{equation}\label{eqn:operator-inequality}
    -\tilde{\delta}\Delta - \Vm \geq -\qm
  \end{equation}
  in the sense of forms on $D(H_{0})$. The Schr{\"o}dinger operator $H_{0}$ is essentially self-adjoint.
\end{corollary}
Note that a condition similar to~\eqref{eqn:operator-inequality} was presented in Theorem 2.7 in~\cite{Braverman_2002}; 
the Laplacian in~\eqref{eqn:operator-inequality} is nonpositive, so this condition
is weaker than the functional condition in~\eqref{eqn:metric}. 
\begin{proof}
  In Lemma~\ref{lemma:positivity}, we find $0\leq \delta < 1$, so that condition~\eqref{eqn:bilinear-form} is satisfied, i.e., we assume that $\delta$ is homogeneous on $M$.

  The existence of an admissible cover on $M$ implies semiboundedness from below the operator $H_{0}$,
  and for $H_{0}$ to be essentially self-adjoint, owing to Theorem~6.2 in~\cite{berezansky1978self},
  it is sufficient to show that the solution
  of the Cauchy problem~\eqref{eqn:hyperbolic-equation-main} is unique for some interval $[0, \tilde{T})$ with $\tilde{T} > 0$.
  
  The condition~\eqref{eqn:operator-inequality} together with Lemma~\ref{lemma:positivity} implies nonnegativity of the energy integral $E(\Tb)$ in Lemma~\ref{lemma:energy-integral} for both the lateral surface and the base of the cone, with the initial conditions vanishing near the boundary of its base.

  Indeed, for instance, in the third integral of the estimate for $E(\Tb)$ in Corollary~\ref{corol:energy-inequality}, we use corresponding
  notation for $\Tb$ and $\bar{\delta}$, for $\delta + \tilde{\delta} < \delta_{1} < \bar{\delta}$ this inequality
  \begin{align*}
    E(\Tb) & \geq \frac{1}{(1-\bar{\delta})^{1/2}}\left.\int_{\mTa_{p,\tzau}}\left(\frac{\bar{\delta}-\delta_{1}}{1-\delta_{1}}\left(\frac{\partial u}{\partial t}\right)^{2} + \delta_{1}\left|\nabla u\right|^{2} + Vu^2\right)d\mu\right.\\
           & \geq \frac{1}{(1-\bar{\delta})^{1/2}}\left.\int_{\mTa_{p,\tzau}} \left((\delta + \tilde{\delta})\left|\nabla u\right|^{2} - \Vm u^2\right)d\mu\right.\\
           & \geq \frac{1}{(1-\bar{\delta)}^{1/2}}\left.\int_{\mTa_{p,\tzau}} \left(\delta\left|\nabla u\right|^{2} - \qm u^2\right)d\mu\right. \geq 0.
  \end{align*}
  
  Under the conditions of our corollary, the second to the last inequality is valid because vanishing solution $u$
  near the boundary $\partial \bmTa_{p,\tzau}$, and the last inequality is due to Lemma~\ref{lemma:positivity}.

  To establish the uniqueness of the solutions of~\eqref{eqn:hyperbolic-equation-main} with the initial conditions on the entire $M$,
  we apply the partition of unity defined in Lemma~\ref{lemma:semibounded} to the initial conditions of the Cauchy problem; it is clear
  that with these initial conditions, all the solutions are unique in their corresponding dependency cones. Moreover, with the minimal
  overlap $\teps < \tzau$ defined in condition~\ref{cond_c_4} of the admissible covers, the unique
  solution of~\eqref{eqn:hyperbolic-equation-main} can be extended to the time interval $[0, (1-\bar{\delta})^{1/2}\teps)$; here, for $\Tb$, we use the corresponding definition of $T$ in~\eqref{dfn:T-domain-dependence} with $\bar{\delta}$ therein instead of $\delta$. 
\end{proof}
As an alternative to Corollary~\ref{cor:operator-minorant}, the authors of~\cite{Braverman_2002} stated in the remark after Theorem~2.7
that the condition for $\tilde{\delta} < 1$
is essential; we add here that, otherwise, the integral $E(\Tb)$ becomes infinite for $\tilde{\delta} = 1$ and thus for $\bar{\delta} = 1$ too,
and its dependency cone definition is no longer valid.

\section{The Schwarzschild, Reissner-Nordstr{\"o}m, and de Sitter metrics}\label{chap:examples}
In the course of studying the domains of dependence for the solutions of
the local mixed problem~\eqref{eqn:hyperbolic-equation},
we use the inner time metric~\eqref{eqn:metric}, namely, for a regular
point $p\in M$ of the potential, we select the range control neighborhood $\bmTa_{p,\tzau}$
satisfying conditions~\eqref{dfn:positive-neighborhood-upd}, and then, we use
the equation~\eqref{eqn:domain-dependence-regular} to define the domain of dependence
with the time cap $T$ defined in~\eqref{dfn:T-domain-dependence}.

Note also that in the domains of dependence definitions~\eqref{eqn:domain-dependence-regular}
for regular and singular points we use minimizing curves w.r.t.~the metric~\eqref{eqn:metric},
and we notice that the domain of the dependence cone is contained inside of the cone
\begin{equation*}
  \{(t, x): t\geq 0, t + \tau(x, p) = \tzau,\ x \in \bmTa_{p,\tzau}\},
\end{equation*}
which is the light cone of this Lorentzian metric
\begin{equation}\label{eqn:minkowski}
  d\ell^{2} \vcentcolon= -\qm dt^{2} + dl^{2} = -\qm(dt^{2}  - d\tau^{2}) + d\omega^{2}, 
\end{equation}
where the formula for $dl^{2}$ in the second equality is taken from~\eqref{eqn:riemann-tau}.
Note that this metric admits an \quotes{unbounded speed of light} in a neighborhood of singular points of $\qm$.

The light cone in metric~\eqref{eqn:minkowski} consists of light rays
$\{(t, x): t = \pm\tau(p, x)\}$
at the origin $(0, p)$.

While investigating the global finite propagation speed property for solutions of the Cauchy
problem~\eqref{eqn:hyperbolic-equation-main}, we have established that the defined conical domains of dependence
are included in the local past light cones
of the Lorentzian metric~\eqref{eqn:minkowski}, and we wanted to research an inverse problem:
given a well-known Lorentzian metric, investigate its corresponding Schr{\"o}dinger operator and
solutions of the Cauchy problem, their global finite propagation speed property, the range control neighborhood
property for its singular points, black hole neighborhoods, etc.

We will conclude this section with examples of metrics corresponding to the hydrogen atom and its Coulomb potential,
strong singularity cases, etc.

\subsection{Schwarzschild Metric}

The Schwarzschild metric, see \S~5.5 in~\cite{hawking1975large}, is given by
\begin{equation*}
  ds^{2} = -c^{2}\left(1 - \frac{2m}{r}\right)dt^{2} + \left(1 - \frac{2m}{r}\right)^{-1}dr^{2} + r^{2}g_{\Omega},
\end{equation*}
where $c$ is the speed of light, $m$ is the mass of the black hole, $M = \mathbb{R}^{3}\setminus \mathcal{D}(0, 2m)$,
a closed disc with a radius of $2m$ in the Euclidean metric, and the spherically symmetric metric on $M$ is defined as
\begin{equation}\label{dfn:schwartz-metric-space}
  dl^{2} = \left(1 - \frac{2m}{r}\right)^{-1}dr^{2} + r^{2}g_{\Omega}
\end{equation}
with $g_{\Omega}$ being the standard Euclidean metric on a 2-dimensional unit sphere.

So
\begin{equation*}
  \qm = c^{2}\left(1 - \frac{2m}{r}\right).
\end{equation*}
For this example let us study the Cauchy problem~\eqref{eqn:hyperbolic-equation-main}
in the neighborhood of "singularity" with $0 < r - 2m \leq 1$. We have used the quote
signs for the term {\it singularity} for the reasons outlined below.

We have
\begin{equation*}
  \begin{aligned}
    & \tau = \int_{2m}^{r}\qm^{-1/2}dl = \int_{2m}^{r}c^{-1}\left(1 - \frac{2m}{r}\right)^{-1/2}\left(1 - \frac{2m}{r}\right)^{-1/2}dr\\
    & = c^{-1}\int_{2m}^{r}\left(1 - \frac{2m}{r}\right)^{-1}dr =  c^{-1}\int_{2m}^{r}\frac{r}{r-2m}dr\\
    & = c^{-1}\int_{2m}^{r}\left(1 + \frac{2m}{r-2m}\right)dr = \frac{r - 2m}{c} - \frac{2m}{c}\log(r - 2m). 
  \end{aligned}
\end{equation*}
In the second equality, we use the expression~\eqref{dfn:schwartz-metric-space} for the metric $dl$, so
$\tau\to\infty$ when $r\to 2m$, and since $\qm\to 0$ when $r\to 2m$, the boundary $r=2m$ is, in fact, infinity,
and we consider the mixed problem~\eqref{eqn:mixed-problem-type2} in its neighborhood.
Note that in~\S~5.5 of~\cite{hawking1975large}, the expression for $\tau$ is denoted by $r^{*}$, {\it the tortoise coordinates} up to
a constant factor, and the Schwarzschild metric is transformed to the Eddington\text{--}Finkelstein form.

Let us verify RCN conditions~\eqref{dfn:positive-neighborhood-upd} of infinity,
and for condition~\eqref{dfn:positive-neighborhood-upd:b}, we have
\begin{equation}\label{eqn:schwarts-pos-neigh-b}
  \begin{aligned}
    & \qm\tau = c^{2}\left(1 - \frac{2m}{r}\right)\left(\frac{r - 2m}{c} - \frac{2m}{c}\log(r - 2m)\right)\\
    & = c/r\left((r - 2m)^{2} - 2m(r-2m)\log(r - 2m)\right)\\
    & = O(-(r-2m)\log(r - 2m)),
  \end{aligned}
\end{equation}
and the left-hand side of~\eqref{dfn:positive-neighborhood-upd:b} tends to zero as $r\to 2m$.

For condition~\eqref{dfn:positive-neighborhood-upd:a}, $\sigma(x) = 2\sqrt{\pi}r$, and, dropping this constant multiplier,
the expression for the logarithm can be estimated by
\begin{equation*}
  \begin{aligned}
    & \log(w) = \log(\qm^{3/4}\sigma(x)\tau) = \log\left(c^{3/2}\left(1 - \frac{2m}{r}\right)^{3/4}r\left[\frac{r - 2m}{c} - \frac{2m}{c}\log(r - 2m)\right]\right)\\
    & = \log\left(c^{1/2}r^{1/4}(r - 2m)^{3/4}\left[(r - 2m) - 2m\log(r - 2m)\right]\right)\\
    & = \log\left(-2mc^{1/2}r^{1/4}(r - 2m)^{3/4}\log(r - 2m)\left[1 - \frac{r - 2m}{2m}\log^{-1}(r - 2m)\right]\right)\\
    & = \log\left(2mc^{1/2}r^{1/4}\right) + \log\left(-(r - 2m)^{3/4}\log(r - 2m)\right) \\
    & + \log\left(1 - \frac{r - 2m}{2m}\log^{-1}(r - 2m)\right).
  \end{aligned}
\end{equation*}
Note that for the first and third terms of the last equality for $r\to +2m$, we estimate
\begin{equation*}
  \begin{aligned}
    & \left|\frac{\partial \log\left(2mc^{1/2}r^{1/4}\right)}{\partial r}\right| < C_1\text{ and}\\
    & \left|\frac{\partial \log\left(1 - \frac{r - 2m}{2m}\log^{-1}(r - 2m)\right)}{\partial r}\right| < C_2,
  \end{aligned}
\end{equation*}
and the second term can be estimated by
\begin{equation*}
  \begin{aligned}
    & \left|\frac{\partial \log\left(-(r - 2m)^{3/4}\log(r - 2m)\right)}{\partial r}\right| =
    \left|\frac{3/4(r-2m)^{-1/4}\log(r - 2m) + (r-2m)^{-1/4}}{(r - 2m)^{3/4}\log(r - 2m)}\right| \\
    & \leq 3/4(r-2m)^{-1} + (r-2m)^{-1}\left|\log^{-1}(r - 2m)\right|.
  \end{aligned}
\end{equation*}

Furthermore, we use unit vector equality because the metric~\eqref{dfn:schwartz-metric-space}
\begin{equation*}
  \begin{aligned}
    \qm^{-1/2}\frac{\partial}{\partial \tau} = \qm^{1/2}\frac{\partial}{\partial r},
  \end{aligned}
\end{equation*}
and, using the estimate~\eqref{eqn:schwarts-pos-neigh-b} and the estimates for the three terms above,
the left-hand side of~\eqref{dfn:positive-neighborhood-upd:a} can be estimated by
\begin{equation}\label{eqn:schwarts-pos-neigh-a}
  \begin{aligned}
    & \tau\left|\frac{\partial \log(w)}{\partial \tau}\right|
    = \qm^{1/2}\tau\left|\qm^{-1/2}\frac{\partial \log(w)}{\partial \tau}\right|
    = \qm^{1/2}\tau\left|\qm^{1/2}\frac{\partial \log(w)}{\partial r}\right|\\
    & = \qm\tau\left|\frac{\partial \log(w)}{\partial r}\right|
    \leq C_{3}(r - 2m)|\log(r-2m)|\ [C_{1}+ C_{2} + 3/4(r-2m)^{-1} \\
    & + (r-2m)^{-1}\left|\log^{-1}(r - 2m)\right|] \leq C_{4}|\log(r-2m)|.
  \end{aligned}
\end{equation}
The right-hand sides of~\eqref{eqn:schwarts-pos-neigh-b} and~\eqref{eqn:schwarts-pos-neigh-a}
monotonically decrease to zero and increase to infinity, respectively, and their product monotonically
decreases to zero when $r\to 2m$, so if we plot the curve
$$\Qcal = \{(x, y) = \left(\tau\left|\frac{\partial \log(\qm^{3/4}\sigma(x)\tau)}{\partial \tau}\right|, \qm\tau\right),
2m < r < r_{0}\}$$ for some $r_{0}$, then it asymptotically approaches $x-$axis as $r\to 2m$, and, moreover, this curve
will be under the hyperbola $xy = 1/16$, so for any point $(\tilde{x}, \tilde{y})$ of this curve,
we can always find a point $(\hat{x}, \hat{y})$ with $\hat{x}\hat{y} = 1/16$ such that
$\tilde{x} < \hat{x}$ and $\tilde{y} < \hat{y}$, so that in~\eqref{dfn:positive-neighborhood-upd}
we can choose $\Czero = \hat{x}, \epszero = 1/\Czero,$ and $A - \del0=1/2$ with very small $\del0 > 0$.
Thus, we prove that for some $r_{0} > 2m$ and any fixed $2m < r \leq r_{0}$
the neighborhood $\{x\in M: r \leq |x| \leq r_{0}\}$ is an RCN.

In Lemma~\ref{lemma:semibounded}, we require that $\delta$ in
the inequality~\eqref{eqn:bilinear-form} of Lemma~\ref{lemma:positivity} is the same for all $r$, so let us
prove that this condition can be satisfied.

Recall that we define $\delta$ in inequalities~\eqref{eqn:greater-A} and~\eqref{eqn:less-A}, so
if we can find $\del0$ such that $A + \del0 < 1$, $\del0 < A$, 
\begin{equation*}
  1 - 4\left(1/\Czerosquare + \epszerosquare\right)\left(\tau\frac{\partial \log(w)}{\partial \tau}\right)^{2}
  > A + \del0
\end{equation*}
and
\begin{equation*}
  4\left(\Czerosquare + 1/\epszerosquare\right)(\qm\tau)^{2} < A - \del0
\end{equation*}
for all fixed $r$, then we can find that $\delta = (A-\del0)/(A + \del0) < 1$.

Conditions~\eqref{eqn:greater-A} and~\eqref{eqn:less-A} similarly correspond to
$\tau \left|\frac{\partial \log(w)}{\partial \tau}\right| <
\frac{\Czero}{2}\sqrt{\frac{1-A-\del0}{1 + \Czerosquare \epszerosquare}}
$
and
$
\qm\tau <  \frac{\epszero}{2}\sqrt{\frac{A-\del0}{1 + \Czerosquare \epszerosquare}},
$
and, as in Corollary~\ref{cor:positive-positive-neighborhood-upd}, the product of the left-hand sides for
fixed $A$ and $\del0$ with $A+\del0 = 1/2$ can be estimated by
\begin{equation*}
  \begin{aligned}
    & \qm\tau^{2}\left|\frac{\partial \log(w)}{\partial \tau}\right| <
    \frac{\Czero \epszero}{4(1 + \Czerosquare \epszerosquare)}\sqrt{(1-A-\del0)(A-\del0)}\\
    & \leq 1/8\sqrt{(1-A-\del0)(A+\del0 - 2\del0)} = 1/8\sqrt{1/2(1/2 - 2\del0)}\\
    & = 1/16\sqrt{1 - 4\del0},
  \end{aligned}
\end{equation*}
and for any $2m < r \leq r_{0}$ we can find $\Czero$ and $\epszero$ such that the curve $\Qcal$ lies
below the hyperbola $xy=1/16\sqrt{1 - 4\del0}$ for all $\{x: r \leq |x| \leq r_{0}\}$,
so that conditions~\eqref{dfn:positive-neighborhood-upd}
and Lemma~\ref{lemma:semibounded} are satisfied. 

The RCN conditions~\eqref{dfn:positive-neighborhood-upd}
and the uniformity of the estimate~\eqref{eqn:bilinear-form}
imply that the mixed  problem~\eqref{eqn:mixed-problem-type2}
has a unique solution $u\in\mU$ in the cylinder
$\{(t, x): [0, \infty) \times \{x\in M: r_{1} \leq |x| \leq r_{2}\}\}$
for any $2m < r_{1} < r_{2} \leq r_{0}$;
its domain of dependence is defined in~\eqref{eqn:domain-dependence-type2}, and,  
if the initial conditions~\eqref{eqn:mixed-problem-type2:b} and~\eqref{eqn:mixed-problem-type2:c}
are zero outside of the spherical layer $[r_{1}, r_{2}]$,
then its solution will remain zero outside of the corresponding cylinder. 

\subsection{Reissner\text{--}Nordstr{\"o}m Metric}

The Reissner\text{--}Nordstr{\"o}m metric, see \S~5.5 in~\cite{hawking1975large}, is given by
\begin{equation*}
  ds^{2} = -\left(1 - \frac{2m}{r} + \frac{e^{2}}{r^2}\right)dt^{2} + \left(1 - \frac{2m}{r} + \frac{e^{2}}{r^2}\right)^{-1}dr^{2} + r^{2}g_{\Omega},
\end{equation*}
where $m$ is the gravitational mass, $e$ is the electric charge,
and the spherically symmetric metric on $M$ is defined by
\begin{equation}\label{dfn:reissner-nord-metric-space}
  dl^{2} = \left(1 - \frac{2m}{r}+ \frac{e^{2}}{r^2}\right)^{-1}dr^{2} + r^{2}g_{\Omega}
\end{equation}
with $g_{\Omega}$ being the standard Euclidean metric on a 2-dimensional unit sphere.

So,
$$
\qm = \left(1 - \frac{2m}{r} + \frac{e^{2}}{r^2}\right)
$$
and for
$$
r_{\pm} = m \pm \sqrt{m^{2} - e^{2}}
$$
we consider these domains where the potential is positive, and the metric~\eqref{dfn:reissner-nord-metric-space}
is Lorentzian for
\begin{equation}\label{eqn:res-nord-cases}
  \qm=\begin{cases}
    \left(1 - \frac{r_{-}}{r}\right)\left(1 - \frac{r_{+}}{r}\right) & \text{for } 0 < r < r_{-} \text{ or } r > r_{+} \text{ when } m^{2} > e^{2}\\
    \left(1 - \frac{m}{r}\right)^{2} & \text{for } 0 < r < m \text{ or } r > m \text { when } m^{2} = e^{2}\\
    \left(1 - \frac{m}{r}\right)^{2} + \frac{e^{2}-m^{2}}{r^{2}} & \text{for } r > 0 \text{ when } m^{2} < e^{2}.\\
  \end{cases}
\end{equation}
In all three cases in~\eqref{eqn:res-nord-cases} there is a singularity at $r=0$, and the first case
contains two factors similar to the Schwarzschild potentials, the second case has a squared Schwarzschild potential,
and the third case is regular everywhere.

Let us calculate $\tau$ for all these cases, and, as in the Schwarzschild case, we evaluate it in either the singularity neighborhood
or in the neighborhood of Schwarzschild horizons. As in the previous example,
\begin{equation}\label{eqn:reis-nord-tau}
  \tau = \int\frac{1}{1 - \frac{2m}{r} + \frac{e^{2}}{r^{2}}}dr, 
\end{equation}
and it is very close to the variable $r^{*}$
defined in~\S~5.5 of~\cite{hawking1975large}, page 157.
We consider these cases
\begin{equation*}
  \tau=\begin{cases}
    r + \frac{r_{+}^{2}}{r_{+}-r_{-}}\log\left(1- \frac{r}{r_{+}}\right)
    -\frac{r_{-}^{2}}{r_{+}-r_{-}}\log\left(1-\frac{r}{r_{-}}\right) & 0 \leq r < r_{-},\ m^{2} > e^{2}\\
    r + \frac{r_{+}^{2}}{r_{+}-r_{-}}\log(r_{+}-r) - \frac{r_{-}^{2}}{r_{+}-r_{-}}\log(r_{-}-r) & 0 < r < r_{-},\ m^{2} > e^{2}\\
    r + \frac{r_{+}^{2}}{r_{+}-r_{-}}\log(r - r_{+})
    - \frac{r_{-}^{2}}{r_{+}-r_{-}}\log(r-r_{-}) & r > r_{+},\ m^{2} > e^{2}\\
    r + m\log\left(\left(1-\frac{r}{m}\right)^{2}\right) + \frac{mr}{m-r} & 0 \leq r < m,\ m^{2} = e^{2}\\
    r + m\log((r-m)^{2}) + \frac{m^2}{m-r} & 0 < r < m,\ m^{2} = e^{2}\\
    r + m\log((m-r)^{2}) +\frac{m^2}{r-m} & r > m,\ m^{2} = e^{2}\\
    r + m\log\left(\frac{r^{2}}{e^{2}}- \frac{2mr}{e^{2}} + 1\right)\\
    + \frac{2m^2-e^2}{e^{2}-m^{2}}\left[\arctan\left(\frac{m}{\sqrt{e^{2}-m^{2}}}\right) + \arctan\left(\frac{r-m}{\sqrt{e^{2}-m^{2}}}\right)\right] &
    r \geq 0,\ m^{2} < e^{2},
  \end{cases}
\end{equation*}
which correspond to a neighborhood of singularity at $r=0$ - see cases 1, 4, and 7;
the rest of the cases correspond to infinities similar to Schwarzschild infinity
as we approach spheres with radii $r_{-}, r_{+},$ and $m$ from inside, outside,
and both directions, respectively.

\subsubsection{Cases 1, 4, and 7}
Let us consider the case of singularity at $r=0$. The first three terms of the Taylor series expansion for the first case, for instance, yield
\begin{equation*}
  \begin{aligned}
    & \tau = r + \frac{r_{+}^{2}}{r_{+} - r_{-}}\log\left(1- \frac{r}{r_{+}}\right) - \frac{r_{-}^{2}}{r_{+}-r_{-}}\log\left(1-\frac{r}{r_{-}}\right)\\
    & = r + \frac{r_{+}^{2}}{r_{+} - r_{-}}\left[- \frac{r}{r_{+}} + \frac{r^{2}}{2r_{+}^{2}} - \frac{r^{3}}{3r_{+}^{3}}\right]
    - \frac{r_{-}^{2}}{r_{+}-r_{-}}\left[-\frac{r}{r_{-}} + \frac{r^{2}}{2r_{-}^{2}} - \frac{r^{3}}{3r_{-}^{3}}\right] + O(r^{4})\\
    & = \left[1 - \frac{r_{+}}{r_{+} - r_{-}} + \frac{r_{-}}{r_{+} - r_{-}}\right]r + \left[\frac{1}{2(r_{+} - r_{-})} - \frac{1}{2(r_{+} - r_{-})}\right]r^{2}\\
    & + \left[-\frac{1}{3r_{+}(r_{+} - r_{-})} + \frac{1}{3r_{-}(r_{+} - r_{-})}\right]r^{3} + O(r^{4})
    = \frac{r^{3}}{3r_{+}r_{-}} + O(r^4)\\
    & = \frac{r^{3}}{3e^{2}} + O(r^4),
  \end{aligned}
\end{equation*}
and we can also see this from the main part $\frac{r^{2}}{e^{2}}$ of the integrand in~\eqref{eqn:reis-nord-tau}.

Let us check the RCN conditions~\eqref{dfn:positive-neighborhood-upd}. For~\eqref{dfn:positive-neighborhood-upd:b}, we have
\begin{equation*}
\qm\tau  = \left(1 - \frac{2m}{r} + \frac{e^{2}}{r^{2}}\right)\left(\frac{r^{3}}{3e^{2}} + O(r^4)\right) = \frac{r}{3} + O(r^{2}),
\end{equation*}
and it can be made arbitrarily small for small $r$.

For condition~\eqref{dfn:positive-neighborhood-upd:a}, as in the previous example, $\sigma=2\sqrt{\pi}r$, and we drop constant factors
in the estimate below:
\begin{equation*}
  \begin{aligned}
    & \log(w) = \log(\qm^{3/4}\sigma\tau) = \log\left[\left(1 - \frac{2m}{r} + \frac{e^{2}}{r^2}\right)^{3/4}r\left(\frac{r^{3}}{3e^{2}} + O(r^4)\right)\right]\\
    & =\log\left(\frac{r^{5/2}}{3e^{1/2}} + O(r^{7/2})\right) = 5/2\log(r) + C + \log(1 + O(r)),
  \end{aligned}
\end{equation*}
so that $\frac{\partial \log(w)}{\partial r} = 5/2r^{-1} + O(1)$, and,
following the same four identities in~\eqref{eqn:schwarts-pos-neigh-a}, we obtain
\begin{equation*}
  \begin{aligned}
    & \tau\left|\frac{\partial \log(w)}{\partial \tau}\right|
    = \qm\tau\left|\frac{\partial \log(w)}{\partial r}\right|\\
    & = \left(\frac{r}{3} + O(r^{2})\right) \left(5/2r^{-1} + O(1)\right) = O(1),\\
  \end{aligned}
\end{equation*}
and for small $r > 0$ the RCN conditions~\eqref{dfn:positive-neighborhood-upd} are
satisfied, and in cases 1, 4, and 7 the solutions of the Cauchy problem~\eqref{eqn:hyperbolic-equation}
exist and are unique for all $t$.

\subsubsection{Cases 2 and 3}
Let us turn to cases 2 and 3; the investigation of the neighborhood of infinities is very similar to that
of Schwarzschild infinity, so we are going to be brief and, at the same time, provide necessary detail.
Both cases are symmetric, so it is sufficient to consider case 2, for instance.

For the RCN condition~\eqref{dfn:positive-neighborhood-upd:b},
we have
\begin{equation*}
  \begin{aligned}
    & \qm\tau = \frac{(r_{-}-r)(r_{+}-r)}{r^{2}} \left\{r + \frac{r_{+}^{2}}{r_{+}-r_{-}}\log(r_{+}-r)-\frac{r_{-}^{2}}{r_{+}-r_{-}}\log(r_{-}-r)\right\} \\
    & = O\left((r_{-}-r)\log(r_{-}-r)\right),
  \end{aligned}
\end{equation*}
and is similar to the estimate~\eqref{eqn:schwarts-pos-neigh-b} for the Schwarzschild metric. 
For condition~\eqref{dfn:positive-neighborhood-upd:a}, using the estimate~\eqref{eqn:schwarts-pos-neigh-b},
we calculate
\begin{equation*}
  \begin{aligned}
    & \tau\left|\frac{\partial \log(w)}{\partial \tau}\right|
    = \qm\tau\left|\frac{\partial \log(w)}{\partial r}\right|\\
    & = O\left((r_{-}-r)|\log(r_{-}-r)|\right) O\left((r_{-}-r)^{-1}\right) = O(|\log(r_{-}-r)|),
  \end{aligned}
\end{equation*}
and, as in the Schwarzschild case, the conditions~\eqref{dfn:positive-neighborhood-upd} are satisfied,
and we can find small values $\epsilon > 0$ so that we can define the same $\delta > 0$ in
the inequality~\eqref{eqn:bilinear-form} of Lemma~\ref{lemma:positivity} for $r \in [r_{-}- \epsilon, r_{-})$.
Therefore, for both cases, the neighborhoods of $r=r_{-}$ and $r=r_{+}$ are RCNs of infinity when we approach them from inside and outside,
respectively. 

\subsubsection{Cases 5 and 6}
Both cases are similar in how we perform estimates, so we concentrate on the case 5. We have
\begin{equation*}
  \qm= \frac{(r-m)^{2}}{r^2},
\end{equation*}
for condition~\eqref{dfn:positive-neighborhood-upd:b},
we estimate
\begin{equation*}
\qm\tau  = \frac{(r-m)^{2}}{r^2} \left(r + m\log((r-m)^{2}) + \frac{m^2}{m-r}\right) = O(m-r),
\end{equation*}
and for condition~\eqref{dfn:positive-neighborhood-upd:a}, we again estimate
\begin{equation*}
  \begin{aligned}
    & \log(w) = \log\left(\qm^{3/4}\sigma\tau\right) = \log\left(\frac{(m-r)^{3/2}}{r^{3/2}} r \left(r  + m\log((m-r)^{2}) + \frac{m^2}{m-r}\right)\right)\\
    & = \log\left(\frac{m^2(m-r)^{1/2}}{r^{1/2}} \left(1 + 2/m(m-r)\log(m-r) - \frac{r(m - r)}{m^2}\right)\right)\\
    & = 1/2 \log(m-r) - 1/2\log(r) + \log\left(1 + 2/m(m-r)\log(m-r) - \frac{r(m - r)}{m^2}\right)+\dots,
  \end{aligned}
\end{equation*}
and its derivative is bounded by 
$
  \left|\frac{\partial \log(w)}{\partial r}\right| = O\left((m-r)^{-1}\right),
$
so that the left-hand side of condition~\eqref{dfn:positive-neighborhood-upd:a} is estimated by
\begin{equation*}
  \tau\left|\frac{\partial \log(w)}{\partial \tau}\right|
  = \qm\tau\left|\frac{\partial \log(w)}{\partial r}\right|
  =  O(m-r) O((m-r)^{-1}) = O(1),
\end{equation*}
and the conditions~\eqref{dfn:positive-neighborhood-upd} are satisfied
in some neighborhoods of infinity at $r=m$.

\subsubsection{Case Summary}
In all the cases we observed that the manifold $M$ is complete w.r.t.~the metric~\eqref{eqn:metric}, and there are RCNs
of both the singularity at $r=0$ and the infinities in other cases, 
such that the Cauchy problem~\eqref{eqn:hyperbolic-equation-main} has a unique solution defined everywhere on $M$ for all $t > 0$.

Note that the potential $\qm\sim \frac{e^{2}}{r^{2}}$ at $r=0$, and, compared with Example~\ref{ex:hydrogen} for the Euclidean space,
this potential admits RCNs for any $e$; this is due to the form of
the metric~\eqref{dfn:reissner-nord-metric-space}, which is very different from the Euclidean metric at the origin.

\subsection{De Sitter Metric}
The De Sitter metric is defined as
\begin{equation*}
  ds^{2} = -\left(1 - \frac{r^{2}}{\ell^2}\right)dt^{2} + \left(1 - \frac{r^{2}}{\ell^2}\right)^{-1}dr^{2} + r^{2}g_{\Omega}, 0 \leq r < \ell
\end{equation*}
where $\ell$ is the cosmological horizon, and the metric on $M = \mathbb{R}^{3}$ is defined by
\begin{equation*}
  dl^{2} = \left(1 - \frac{r^{2}}{\ell^2}\right)^{-1}dr^{2} + r^{2}g_{\Omega}, 0 \leq r < \ell.
\end{equation*}

\noindent Thus, $\qm = 1 - \frac{r^{2}}{\ell^2}$, and we have 
\begin{equation*}
  \tau = \int_{r}^{\ell} \qm^{-1}dr = \frac{\ell}{2}\log\frac{1+\frac{r}{\ell}}{1-\frac{r}{\ell}},
\end{equation*}
and $\tau \to \infty$ when $r \to \ell$, and the horizon $r = \ell$ is, in fact, infinity.

\noindent For condition~\eqref{dfn:positive-neighborhood-upd:b} we calculate
\begin{equation*}
  \qm\tau = \left(1 - \frac{r^{2}}{\ell^2}\right)\frac{\ell}{2}\log\frac{1+\frac{r}{\ell}}{1-\frac{r}{\ell}}
  = O\left(-\left(1 - \frac{r}{\ell}\right)\log\left(1 - \frac{r}{\ell}\right)\right),
\end{equation*}
and $\qm\tau \to 0$ when $r \to \ell$.

\noindent For condition~\eqref{dfn:positive-neighborhood-upd:a}, as in previous chapters, we calculate
\begin{equation*}
  w = \qm^{3/4}\sigma\tau = C_{1}\left(1 - \frac{r^{2}}{\ell^2}\right)^{3/4}\log\frac{1+\frac{r}{\ell}}{1-\frac{r}{\ell}},
\end{equation*}
and the corresponding derivative is estimated by 
\begin{equation*}
  \frac{\partial\log (w)}{\partial r} = O\left(1 - \frac{r}{\ell}\right)^{-1}.
\end{equation*}
Therefore, for condition~\eqref{dfn:positive-neighborhood-upd:a}, we evaluate
\begin{equation*}
  \begin{aligned}
    & \tau \left| \frac{\partial \log(w)}{\partial \tau} \right| = \qm\tau \left| \frac{\partial \log(w)}{\partial r} \right|\\
    & = O\left(-\left(1 - \frac{r}{\ell}\right)\log\left(1 - \frac{r}{\ell}\right)\right) O\left(\left(1 - \frac{r}{\ell}\right)^{-1}\right) \\
    & = O\left(-\log\left(1 - \frac{r}{\ell}\right)\right),
  \end{aligned}
\end{equation*}
and, similar to the case of the Schwarzschild metric, the conditions~\eqref{dfn:positive-neighborhood-upd}
are satisfied in the neighborhood of infinity, so the horizon $\ell$ is an RCN of infinity.

In the next two examples we define Lorentzian metrics from the Cauchy problems
for the wave equations with the Schr{\"o}dinger operators we considered
before. 
\subsection{Minkowski Metric}
The Minkowski metric on $\mathbb{R}\times \RRn, n \geq 3$ is defined by
\begin{equation}\label{dfn:minkowski}
  ds^{2} = -c^{2}dt^{2} + dr^{2} + r^{2}g_{\Omega},
\end{equation}
where $c$ is the speed of light.
Therefore, $\qm=c^{2}$ and $\tau=r/c$; we have already shown in Example~\ref{ex:minkowski} that a small neighborhood of the origin is RCN, and the light cones for~\eqref{dfn:minkowski} are defined by $t=\pm r/c$.

\subsection{Hydrogen Atom}
For the hydrogen atom we define a corresponding Lorentzian metric in $\mathbb{R}\times \mathbb{R}^{3}$
with the Coulomb potential
\begin{equation*}
  ds^{2} = -\frac{dt^{2}}{r} + dr^{2} + r^{2}g_{\Omega}.
\end{equation*}

We estimate $$\tau = \int_{0}^{r}\qm^{-1/2}dr = \int_{0}^{r}r^{1/2}dr = 2/3r^{3/2},$$ and the future and past light cones are defined by $t=\pm 2/3r^{3/2}$.

We have the following: 
\begin{proposition}\label{prop:uncertainty-principle}
  \textbf{Time\text{--}Energy Uncertainty Principle Estimate for a Hydrogen Atom.}
  The following inequality holds:
  \begin{equation*}
    \Delta E \Delta t \geq 7/8,
  \end{equation*}
  where $\Delta t$ is an estimate of the time needed to achieve an energy level transition $\Delta E$.
\end{proposition}
\begin{proof}
  Consider two energy levels $n_{1}$ and $n_{2}$ and suppose that $n_{2} > n_{1}$
  and that the electron transitions from state $n_{1}$ to $n_{2}$.
  Therefore, the energy difference between these states is
  $\Delta E := E_{n_{2}} - E_{n_{1}} = -1/(4n_{2}^{2}) + 1/(4n_{1}^{2}) = 1/4(1/n_{1}^{2} - 1/n_{2}^{2})$. Note that $\Delta E > 0$.

  To estimate $\Delta t$, according to the Bohr\text{--}Rutherford model,
  an electron of the hydrogen atom in state $n$ rotates
  around its atom in a circular motion with an orbit of radius $r_{n} = a_{0}n^{2}$, where
  $a_{0}$ is the constant Bohr radius; we normalize it to 1. A trajectory line of the electron is inside the light cone above, so we can estimate
  $\Delta t \geq 2/3(r_{n_{2}}^{3/2} - r_{n_{1}}^{3/2}) = 2/3\left(n_{2}^{3} - n_{1}^{3}\right)$; hence, we obtain
  \begin{equation*}
    \begin{aligned}
      & \Delta E \Delta t \geq 1/6\left(1/n_{1}^{2} - 1/n_{2}^{2}\right)\left(n_{2}^{3} - n_{1}^{3}\right)\\
      & = 1/6 \left(1/n_{1} - 1/n_{2}\right)\left(1/n_{1} + 1/n_{2}\right)(n_{2} - n_{1})\left(n_{1}^{2} + n_{1}n_{2} + n_{2}^{2}\right) \\
      & = 1/6 \left(n_{2}/n_{1} + n_{1}/n_{2} - 2\right) \left(1/n_{1} + 1/n_{2}\right)\left(n_{1}^{2} + n_{1}n_{2} + n_{2}^{2}\right)\\
      & \geq 1/6*1/2*3/2(1 + 2 + 4) = 7/8.
    \end{aligned}
  \end{equation*}
  Here, the minimum is attained for the ground state $n_{1}= 1$ and for the next excited state $n_{2}= 2$.

  \noindent Note that the expression $\Delta E\Delta t$ will remain positive when the transition
  is from a higher state $n_{1}$ to a lower state $n_{2}$, i.e., when $n_{2} < n_{1}$ above. 
\end{proof}

\subsection{Potential in Example~\ref{ex:hydrogen}}
From Example~\ref{ex:hydrogen} with $\alpha=1$ and $\beta^{2}=1/(4n^{2})$, let us
define corresponding Lorentzian metric as $\mathbb{R}\times \RRn, n=1,\ 2,\ ...$ 
\begin{equation*}
  ds^{2} = -\frac{1}{4n^{2}r^{2}}dt^{2} + dr^{2} + r^{2}g_{\Omega}.
\end{equation*}
We treat here all $\RRn$ as naturally embedded into $\mathbb{R}^{\infty}$ with the map
$i:\RRn \hookrightarrow \mathbb{R}^{\infty}$ by
$i(x_1,\ x_2,\ \dots,\ x_{n}) = (x_1,\ x_2,\ \dots,\ x_{n}, 0,\ \dots)$, so the distance from the origin
$r\vcentcolon=\left(\sum_{i=1}^{\infty}|x_{i}|^{2}\right)^{1/2}$ is well defined.

We choose $\beta$ so that a small neighborhood of origin is the RCN,
and we estimate
$$\tau = \int_{0}^{r}\qm^{-1/2}dr = \int_{0}^{r}2nrdr = nr^{2},$$
thus, the
future and past light cones are paraboloids defined by $t=\pm nr^{2}$.

When we consider conditions on the future light cones
in $\mathbb{R}^{\infty}$
\begin{equation*}
  \sum_{i}^{n}x_{i}^{2} \leq t/n\ \text{for all } n > 0,
\end{equation*}
then, necessarily, we must have $t\to\infty$ when $n\to\infty$; otherwise, when $\overline{\lim}_{n\to\infty}t < \infty$,
then $x_{i}\to 0$ for all $i$, and an orbit of a particle would collapse on the origin.
If we assume that all critical orbits are bounded and nonzero when $n\to\infty$, then all legitimate orbits
belong to the Hilbert space $\ell^{2}\subset \mathbb{R}^{\infty}$.

\section{Addendum. Strongly Singular Potential \\
  $-\beta^{2}/|x|^{2}$ in $\RRn\setminus\{0\}, n\geq 5$.}
As we noted at the end of Example~\ref{ex:hydrogen}, the larger estimate for the parameter $\beta$
given in Example D.1 in~\cite{Braverman_2002} has to do with the fact that $M=\RRn\setminus\{0\}$ has a boundary
at the origin, it is not complete in either Euclidean or $\tau$ metrics,
so even for the Laplacian to be essentially self-adjoint, we must have the condition $n \geq 4$.
In this section we wanted to show that the essential self-adjointness conditions are intimately related
to the inner time metric~\eqref{eqn:metric}.

A special case of Theorem~1 in~\cite{Bruse1998} provides the following estimate:
\begin{equation}\label{eqn:grad-div-estimate}
  \int_{\RRn}|\nabla \phi(x)|^{2}dx \geq \int_{\RRn}(\text{div}\ X - |X|^{2}) \phi^{2}(x)dx
\end{equation}
for each real-valued $\phi \in C^\infty_{0}(\RRn\setminus\{0\})$ and any Lipschitz vector field $X$.
Define metric
\begin{equation}\label{eqn:add-tau-metric}
  \tau(0, x) = \int_{0}^{|x|}q^{-1/2}(x)dr = \frac{r^{2}}{2\beta},
\end{equation}
and for the vector field
\begin{equation}\label{eqn:add-vf}
  X = \frac{\partial}{\partial \tau},
\end{equation}
we estimate
\begin{equation}\label{eqn:square-X}
  |X|^{2} = \qm(x).
\end{equation}
This is due to the definition of the metric~\eqref{eqn:riemann-tau}, and to estimate $\text{div}X$, we first rewrite
\begin{equation*}
  X = \frac{\partial}{\partial \tau} = \frac{\partial r}{\partial \tau} \frac{\partial}{\partial r}  =
  \frac{1}{\frac{\partial \tau}{\partial r}}\frac{\partial}{\partial r} = \frac{\beta}{r} \frac{\partial}{\partial r}, 
\end{equation*}
and, using spherical coordinates in $\RRn$, we estimate
\begin{equation}\label{eqn:divergence-sing}
  \text{div} X = \frac{1}{r^{n-1}}\frac{\partial}{\partial r}(\beta r^{-1}r^{n-1}) = \beta(n-2)r^{-2},
\end{equation}
so that the expression on the right-hand side of~\eqref{eqn:grad-div-estimate} can be estimated by
$\text{div}\ X - |X|^{2} = \beta(n-2)r^{-2} - \qm(x) = (\beta(n-2) - \beta^{2})r^{-2}$, and for this expression to be $\geq \qm(x)$, we must have $\beta^{2} \leq \frac{(n-2)^{2}}{4}$,
and the operator~\eqref{eqn:schrodinger} is nonnegative on $C^\infty_{0}(\RRn\setminus\{0\})$ for such $\beta$.

Now let us turn to the essential self-adjointness conditions for the operator~\eqref{eqn:schrodinger}.
We rely on the following Theorem 3~\cite{Bruse1998}.
\begin{theorem}[Correcting Potentials]\label{thm:correct-potential}
  Suppose that
  \begin{equation}\label{eqn:correct-potential}
    -\qm(x) \geq |\nabla\eta|^{2} + |X|^{2} - \text{div} X - C_{1}
  \end{equation}
  for some $\eta\in C^{2}(M)$ such that $\eta \to \infty $ when $|x| \to 0$,
  a Lipschitz vector field $X$, and some $C_{1}\geq 0$.
  Additionally, assume also that the function $\eta$ satisfies the inequality
  \begin{equation}\label{eqn:eta-inequality}
    |\nabla\eta(x)|^{2} \leq C_{2}e^{2\eta}, \text{for a.e. } x\in M.
  \end{equation}
  The operator~\eqref{eqn:schrodinger} with $D(H)=C^{\infty}_{0}(M)$ is essentially self-adjoint.
\end{theorem}
Note that Theorem~\ref{thm:correct-potential} is a much simpler version of Theorem~3 in~\cite{Bruse1998}, where the author considers more general elliptic operators, and the domain $M$ may have multiple disjoint regular boundaries of any dimension less than $n$, etc.

We are ready to formulate the following:
\begin{corollary}
  The Laplace operator is essentially self-adjoint for $n\geq 4$.
  The Schr{\"o}dinger operator~\eqref{eqn:schrodinger}
  with $\widetilde{\Vm} =\widetilde{\qm}\vcentcolon= \alpha |x|^{-2}, \alpha \geq 0$ is essentially
  self-adjoint when $\alpha \leq \left(\frac{n-2}{2}\right)^{2} - 1$.
\end{corollary}
We prove here the sufficiency of these conditions, but, as noted in Corollary~3 and Remark~2 of~\cite{Bruse1998},
these conditions are also necessary.
\begin{proof}
  Our proof is somewhat different from the one given in~\cite{Bruse1998}, as we are going to utilize a correcting potential
  $\Vm=\qm=\beta^{2}|x|^{-2}$, a multiple of the original one, 
  then define for it the metric $\tau$ in~\eqref{eqn:add-tau-metric}, then vector field $X$ in~\eqref{eqn:add-vf},
  and make use of the expressions~\eqref{eqn:square-X} and~\eqref{eqn:divergence-sing}.

  For inequality~\eqref{eqn:eta-inequality}, we define $\eta = -1/2\log(\tau)$, and
  $e^{-2\eta}|\nabla\eta(x)|^{2} = |\nabla e^{-\eta}|^{2} = |\nabla \tau^{1/2}|^{2} = |\frac{\nabla r}{\sqrt{2\beta}}|^{2} = \frac{1}{2\beta}$,
  so in~\eqref{eqn:eta-inequality}, the constant $C_{2} = \frac{1}{2\beta}$.

  By multiplying each term of~\eqref{eqn:correct-potential} by $e^{-2\eta}$, we obtain
  $e^{-2\eta}\widetilde{\qm} = \alpha \tau r^{-2} = \frac{\alpha}{2\beta}$,
  $e^{-2\eta}|X|^{2} = e^{-2\eta}\qm = \beta^{2} \tau r^{-2} = \beta/2$, and
  $e^{-2\eta}\text{div}X = \beta(n-2) r^{-2}\tau = (n-2)/2$;
  thus, condition~\eqref{eqn:correct-potential} can be rewritten in this form
  \begin{equation*}
    \frac{1}{2\beta}\left[\beta^{2} - (n-2)\beta + 1 + \alpha\right] - C_{1}e^{-2\eta} \leq 0, x\in M,
  \end{equation*}
  and, since $C_{1}e^{-2\eta}\to 0$ when $|x|\to 0$, this condition is equivalent to
  \begin{equation}\label{eqn:correct-potential-final}
    \beta^{2} - (n-2)\beta + 1 + \alpha \leq 0.
  \end{equation}
  For the Laplacian operator with $\alpha = 0$, this condition can be satisfied only when $n \geq 4$.
  For any other value $\alpha > 0$, the minimal value
  of the left-hand side of~\eqref{eqn:correct-potential-final} can be set to
  $1 + \alpha - \left(\frac{n-2}{2}\right)^{2}$, so to satisfy~\eqref{eqn:correct-potential-final}, we must have $\alpha \leq \left(\frac{n-2}{2}\right)^{2} - 1$.
\end{proof}

\section{Appendix. Kato and Stummel Classes.}
In this Appendix we provide a brief survey of the Kato and Stummel classes of the potentials,
their definitions and most important properties,
and it is essentially a shorter version of Appendix~C in~\cite{Braverman_2002}.
\subsection{Stummel Classes}
These classes of potentials were first introduced in~\cite{Stummel1956/57}, and
more details about them can be found in~\cite{cycon1987schrodinger}, \S~1.2 and
in~\cite{schechter1986spectra}, Chapts.~5 and~9.

A real-valued measured potential $f$ belongs to the uniform Stummel class $S_{n}(\RRn)$ iff the following conditions hold:
\begin{equation}\label{eq:stummel-class}
  \begin{aligned}
    & \lim_{r\downarrow 0}\left[\sup_{x}\int_{|x-y|\leq r}|x-y|^{4-n}|f(y)|^{2}dy\right] = 0 && \text{for } n\geq 5\\
    & \lim_{r\downarrow 0}\left[\sup_{x}\int_{|x-y|\leq r}\log(|x-y|^{-1})|f(y)|^{2}dy\right] = 0 && \text{for } n = 4\\
    & \sup_{x}\int_{|x-y|\leq r_{0}}|f(y)|^{2}dy < \infty && \text{for } n \leq 3,
  \end{aligned}
\end{equation}
where $r_{0}$ is an arbitrary fixed value.
If $\phi: \RRn \to \RRn$ is a diffeomorphism that is linear at infinity, then it is clear
that $\phi^{*}(f) = f\circ \phi$ also belongs to $S_{n}(\RRn)$.
In this way, we can always define a local Stummel class on $S_{n,\loc}(M)$ as a set of functions
satisfying~\eqref{eq:stummel-class} in local coordinates defined by a diffeomorphism $\phi$. $S_{n,\loc}(M)$ is invariant under diffeomorphisms in $M$ and under multiplication by a real valued function from $L^{\infty}_{loc}(M)$.

\noindent We have these relationships between Stummel and classes $L^{p}$:

\begin{equation}\label{eq:stummel-incl}
  L^{p}_{\loc}(M) \subset S_{n, \loc}(M)\ \text{if}\ p > n/2\ \text{for}\ n\geq 4\ \text{and if}\ p=2\ \text{for}\ n\leq 3.
\end{equation}

\noindent The same inclusion is valid for the uniform class $S_{n}(\RRn)$ if $L^{p}_{\loc}(\RRn)$, which is replaced
by the uniform $L^{p}_{\text{unif}}(\RRn)$, i.e., functions from $L^{p}_{\loc}(\RRn)$,
whose $L^{p}$ norms over unit balls on $\RRn$ are uniformly bounded.

The following important majorization property is satisfied for the Stummel classes of potentials, i.e.
if $V\in S_{n}(\RRn)$, then for any $\varepsilon > 0$ there exists $C > 0$ such that

$$||Vu||_{2} \leq \varepsilon ||\Delta u||_{2} + C ||u||_{2}$$ for any $u\in C_{0}^{\infty}(\RRn)$.

\subsection{Kato Classes}
The Kato class was introduced in~\cite{Kato1972}, and one could
find a very good reference in~\cite{cycon1987schrodinger}, \S~1.2.

A real-valued measurable function $f\in K_{n}(\RRn)$, the uniform Kato class,
if it satisfies these conditions
\begin{equation}\label{dfn:kato-class}
  \begin{aligned}
    & \lim_{r\downarrow 0}\left[\sup_{x}\int_{|x-y|\leq r}|x-y|^{2-n}|f(y)|dy\right] = 0 && \text{for } n\geq 3\\
    & \lim_{r\downarrow 0}\left[\sup_{x}\int_{|x-y|\leq r}\log(|x-y|^{-1})|f(y)|dy\right] = 0 && \text{for } n = 2\\
    & \sup_{x}\int_{|x-y|\leq r_{0}}|f(y)|dy < \infty && \text{for } n = 1
  \end{aligned}
\end{equation}
where $r_{0} > 0$ is an arbitrary fixed number.

The class $K_{n}(\RRn)$ is invariant
by multiplication of the bounded real-valued measurable function and under
diffeomorphisms of $\RRn$ which are linear at infinity. Thus we can define
$K_{n,\loc}(M)$ for any smooth manifold $M$ with $\dim(M) = n$.

Similar to the inclusions~\eqref{eq:stummel-incl}, these inclusions for the Kato classes
\begin{equation*}
  L^{p}_{\loc}(M) \subset K_{n, \loc}(M)\ \text{if}\ p > n/2\ \text{for}\ n\geq 2\ \text{and if}\ p=2\ \text{for}\ n = 1.
\end{equation*}
We can replace these inclusions for the uniform $K_{n}(\RRn)$
if we use the class $L^{p}_{\text{unif}}(\RRn)$ instead of $L^{p}(\RRn)$.

The following majorization property for the forms makes Kato classes very important:
if $V\in K_{n}(\RRn)$, then for any $\varepsilon > 0$, there exists $C > 0$ such that

$$|(Vu, u)| \leq \varepsilon (-\Delta u, u) + C||u||_{2}^{2}$$

\noindent for any $u \in C^{\infty}_{0}(\RRn)$.

We want to mention that the definition~\eqref{dfn:kato-class}
of the uniform Kato class is not applicable to the general
Riemannian manifolds, and a more general definition of $K_{n}(M)$ is given
in~B.~G{\"u}neysu paper~\cite{guneysu2014katoclass} - see Definition~2.6
and Theorem~2.13 for the relative bound estimate similar to~\eqref{eq:relative-bound};
very briefly, Definition~2.6, similar to~\eqref{dfn:kato-class}, uses a smooth integral kernel
for the operator $\text{e}^{\frac{t}{2}\Delta}$.
For example, Corollary~2.11  gives an analytical definition of potentials in $K_{n}(M)$
if
\begin{enumerate}
\item The manifold $M$ is geodesically complete with Ricci curvature $\text{Ric}(M) > -C$ for some $C > 0$;
\item The volume $\vol(B_g(x, r)) \geq K r^{n}$ for any $x\in M, r < R,$ and for some $K, R > 0$.
\end{enumerate}
For $p \geq 1$ if $n =1$ and $p > n/2$ if $n \geq 2$ we have $L^{p}(M) + L^{\infty}(M) \subset K_{n}(M)$.   

\section*{Acknowledgments}
I would like to thank Professor Victor Bangert for useful discussions about sets of positive reach, and I am personally
indebted to Professor Ognjen Milatovic for support in writing this paper and helpful remarks.
I would like to express my gratitude to the members of the Analysis and Geometry Seminar at Northeastern University
Professors Robert McOwen and Maxim Braverman for their valuable feedback on my presentation.
I thank Professor David Massey for reviewing this manuscript and making helpful suggestions.

\section*{Conflict of Interest}
I declare that I have no conflicts of interest related to this paper.

\section*{Data Availability Statement}
I certify that no new data were created or analyzed in this study. Data sharing is not applicable to this paper.

\bibliographystyle{acm}
\bibliography{on_fin_prop_sp_comappl_math} 

@book{berezansky1978self,
author = { Berezanskii, IU. M. },
title = { Selfadjoint operators in spaces of functions of infinitely many variables / Yu. M. Berezanskii },
isbn = { 0821845152 },
publisher = { American Mathematical Society Providence, R.I },
pages = { xv, 383 p. ; },
url = {https://bookstore.ams.org/mmono-63},
year = { 1986 },
type = { Book },
language = { English },
subjects = { Theory of distributions (Functional analysis); Spectral theory (Mathematics) },
life-dates = { 1986 -  },
}

@article{Braverman_2002,
   title={Essential self-adjointness of Schr{\"o}dinger-type operators on manifolds},
   volume={57},
   ISSN={1468-4829},
   url={http://dx.doi.org/10.1070/RM2002v057n04ABEH000532},
   DOI={10.1070/rm2002v057n04abeh000532},
   number={4},
   journal={Russian Mathematical Surveys},
   publisher={IOP Publishing},
   author={Braverman, M and Milatovic, O and Shubin, M},
   year={2002},
   month={Aug},
   pages={641–692}
}

@incollection {key0143037m,
	AUTHOR = {Calder\'on, A. P.},
	TITLE = {Lebesgue spaces of differentiable functions
		 and distributions},
	BOOKTITLE = {Partial differential equations},
	EDITOR = {Morrey, Jr., Charles B.},
	SERIES = {Proceedings of Symposia in Pure Mathematics},
	NUMBER = {4},
	PUBLISHER = {American Mathematical Society},
	ADDRESS = {Providence, RI},
	YEAR = {1961},
	PAGES = {33--49},
	NOTE = {(Berkeley, CA, 21--22 April 1960). MR:0143037.
		 Zbl:0195.41103.},
	ISSN = {0082-0717},
}

@article{Chumak1973,
   author = {Chumak, A. A.},
   year = {1973},
   month = {11},
   URL = {https://doi.org/10.1007/BF01090797},
   pages = {649--655},
   title = {{Self-adjointness} of the {Beltrami-Laplace} operator on a complete paracompact {Riemannian} manifold without boundary},
   volume = {25},
   number = {6},
   journal = {Ukrainian Mathematical Journal},
   doi = {10.1007/BF01090797},
   ISSN={1573-9376},
}

@book{cycon1987schrodinger,
  title={Schr{\"o}dinger Operators: With Applications to Quantum Mechanics and Global Geometry},
  author={Cycon, H.L. and Simon, B. and Froese, R.G. and Kirsch, W. and Beiglb{\"o}ck, E.},
  isbn={9783540167587},
  lccn={86013953},
  series={Springer study edition},
  url={https://books.google.com/books?id=HR\_-P2mxkSkC},
  year={1987},
  publisher={Springer}
}

@incollection{Davies1999,
author = {Davies, E. B.},
year = {1999},
pages = {55--67},
title = {A {Review} of {Hardy} {Inequalities}.},
booktitle = {Rossmann J., Takáč P., Wildenhain G. (eds) The Maz’ya Anniversary Collection. Operator Theory: Advances and Applications},
volume = {110},
publisher = {Birkhäuser, Basel},
doi = {10.1007/978-3-0348-8672-7_5},
URL ={https://doi.org/10.1007/978-3-0348-8672-7_5}
}

@article{RSMUP_1957__27__284_0,
     author = {Gagliardo, Emilio},
     title = {Caratterizzazioni delle tracce sulla frontiera relative ad alcune classi di funzioni in $n$ variabili},
     journal = {Rendiconti del Seminario Matematico della Universit\`a di Padova},
     pages = {284--305},
     publisher = {Seminario Matematico of the University of Padua},
     volume = {27},
     year = {1957},
     zbl = {0087.10902},
     mrnumber = {102739},
     language = {it},
     url = {http://www.numdam.org/item/RSMUP_1957__27__284_0/}
}

@article{Gim68,
 URL = {http://mi.mathnet.ru/mz9451},
 DOI={https://doi.org/10.1007/BF01116446},
 author = {Gimadislamov, Mazgar},
 journal = {Math. Notes},
 number = {3},
 pages = { 674--679},
 title = {Sufficient conditions that the minimum and maximum of partial differential operators should coincide and that their spectra should be discrete},
 volume = {4},
 year = {1968}
}

@article{hartman1951,
 ISSN = {00029327, 10806377},
 URL = {http://www.jstor.org/stable/2372315},
 author = {Philip Hartman},
 journal = {American Journal of Mathematics},
 number = {3},
 pages = {635--645},
 publisher = {Johns Hopkins University Press},
 title = {The Number of {$L^{2}$}-Solutions of {$x'' + q(t)x = 0$}},
 urldate = {2022-04-18},
 volume = {73},
 year = {1951}
}

@article{Ism62,
 URL = {http://mi.mathnet.ru/dan26157},
 author = {Ismagilov, Rais},
 journal = {Dokl. Akad. Nauk SSSR},
 number = {6},
 pages = {1239--1242},
 title = {Self-adjointness conditions of higher-order differential operators},
 volume = {142},
 year = {1962}
}

@article{Kalf_Walter_1972,
   title={Strongly singular potentials and essential self-adjointness of singular elliptic operators in {$C_{0}^{\infty}({\mathbb{R}}^n\setminus\{0\})$}},
   volume={10},
   DOI={https://doi.org/10.1016/0022-1236(72)90059-6},
   number={1},
   journal={Journal of Functional Analysis},
   publisher={Elsevier},
   author={Kalf, Hubert and Walter, Johann},
   year={1972},
   month={May},
   pages={114-130}
}

@article{lytchak_yaman,
 ISSN = {00029947},
 URL = {http://www.jstor.org/stable/3845559},
 author = {Alexander Lytchak and Asli Yaman},
 journal = {Transactions of the American Mathematical Society},
 number = {7},
 pages = {2917--2926},
 publisher = {American Mathematical Society},
 title = {On {Hölder} Continuous {Riemannian} and {Finsler} Metrics},
 volume = {358},
 year = {2006}
}

@article{Oleinik1993,
author = {Oleinik, I. M.},
ISSN = {1573-8876},
URL = {https://doi.org/10.1007/BF01209558},
year = {1993},
month = {09},
pages = {934--939},
title = {{On} the {Essential} {Self-Adjointness} of the {Schroedinger} {Operator} on {Complete Riemannian Manifolds}},
volume = {54},
journal = {Mathematical Notes},
doi = {10.1007/BF01209558}
}

@article{Oleinik1994,
author = {Oleinik, I. M.},
ISSN = {1573-8876},
URL = {https://doi.org/10.1007/BF02112477},
year = {1994},
month = {04},
pages = {380--386},
title = {{On} the {Connection} of the {Classical} and {Quantum Mechanical Completeness} of a {Potential} at {Infinity} on {Complete Riemannian Manifolds}},
volume = {55},
journal = {Mathematical Notes},
doi = {10.1007/BF02112477}
}

@article{OleinikProceedings,
 ISSN = {00029939, 10886826},
 URL = {http://www.jstor.org/stable/119024},
 author = {I. M. Oleinik},
 journal = {Proceedings of the American Mathematical Society},
 number = {3},
 pages = {889--900},
 publisher = {American Mathematical Society},
 title = {{On} the {Essential Self-Adjointness} of the {General Second Order Elliptic Operators}},
 volume = {127},
 year = {1999}
}

@article{Oro82,
   author = {Orochko, Yu. B.},
   URL = {http://www.ams.org/mathscinet-getitem?mr=679034},
   year = {1982},
   pages = {1764--1772},
   title = {The property of global finite rate of propagation of a second-order elliptic differential expression},
   volume = {18},
   number = {10},
   journal = {Differ. Uravn.}
}

@article{Oro88,
   author = {Orochko, Yu. B.},
   URL = {https://doi.org/10.1070/RM1988v043n02ABEH001728},
   year = {1988},
   month = {09},
   pages = {51--102},
   title = {The hyperbolic equation method in the theory of operators of {Schr\"odinger} type with a locally integrable potential},
   volume = {43},
   number = {2},
   journal = {Russian Math. Surveys},
   doi = {10.1007/BF01209558}
}

@article{Rof70,
   title={Conditions for the self-adjointness of the {Schrodinger} operator},
   volume={8},
   url={https://doi.org/10.1007/BF01673689},
   DOI={10.1007/BF01673689},
   journal={Math. Notes},
   author={Rofe-Beketov, F. S.},
   year={1970},
   pages={888--894},
   number={6}
}

@article{S_mann_2018,
   title={On geodesics in low regularity},
   volume={968},
   ISSN={1742-6596},
   url={http://dx.doi.org/10.1088/1742-6596/968/1/012010},
   DOI={10.1088/1742-6596/968/1/012010},
   journal={Journal of Physics: Conference Series},
   publisher={IOP Publishing},
   author={Sämann, Clemens and Steinbauer, Roland},
   year={2018},
   month={Feb},
   pages={012010}
}

@article{VisLad56,
 URL = {http://www.ams.org/mathscinet-getitem?mr=94577},
 author = {Vishik, M. I. and Ladyzhenskaya, O. A.},
 journal = {Uspekhi Mat. Nauk},
 number = {6(72)},
 pages = {41--97},
 title = {Boundary value problems for partial differential equations and certain classes of operator equations},
 volume = {11},
 year = {1956}
}

@article{Wilcox1962,
 URL = {https://doi.org/10.1007/BF00281202},
 author = {Wilcox, Calvin H.},
 journal = {Archive for Rational Mechanics and Analysis},
 number = {1},
 pages = {361--400},
 publisher = {Johns Hopkins University Press},
 title = {Initial--boundary value problems for linear hyperbolic partial differential equations of the second order},
 volume = {10},
 year = {1962},
 DOI={10.1007/BF00281202}
}

@book{hawking1975large,
  asin = {0521099064},
  author = {Hawking, Stephen W. and Ellis, G. F. R.},
  description = {The Large Scale Structure of Space-Time (Cambridge Monographs on Mathematical Physics) (9780521099066): Stephen W. Hawking, G. F. R. Ellis},
  dewey = {530.11},
  ean = {9780521099066},
  isbn = {0521099064},
  publisher = {Cambridge University Press},
  title = {The Large Scale Structure of Space-Time (Cambridge Monographs on Mathematical Physics)},
  url = {http://www.amazon.com/Structure-Space-Time-Cambridge-Monographs-Mathematical/dp/0521099064},
  year = 1975
}

@book{evansgariepy2015,
  title={Measure Theory and Fine Properties of Functions, Revised Edition (1st ed.)},
  author={Evans, L.C. and Gariepy, R.F.},
  isbn={978-1-4822-4238-6},
  url={https://doi.org/10.1201/b18333},
  year={2015},
  publisher={Chapman and Hall/CRC Press}
}

@article{Federer_1959,
   title={Curvature Measures},
   volume={93},
   pages={418-491},
   url={https://doi.org/10.2307/1993504},
   number={3},
   journal={Transactions of the American Mathematical Society},
   author={Federer, Herbert},
   year={1959}
}

@article{Bangert1982,
   title={Sets with positive reach},
   ISSN={1420-8938},
   volume={38},
   pages={54-57},
   url={https://doi.org/10.1007/BF01304757},
   number={1},
   journal={Archiv der Mathematik},
   author={Bangert, Victor},
   year={1982},
   DOI={10.1007/BF01304757}
}

@misc{lytchak2023note,
      title={A note on subsets of positive reach}, 
      author={Alexander Lytchak},
      year={2023},
      eprint={2302.05157},
      archivePrefix={arXiv},
      primaryClass={math.MG},
      url={https://arxiv.org/abs/2302.05157},
      howpublished = {\url{https://arxiv.org/abs/2302.05157}}
}

@article{sears_1950,
      title={Note on the {Uniqueness} of the {Green's} {Functions} {Associated} {With Certain Differential Equations}},
      volume={2},
      DOI={10.4153/CJM-1950-029-9},
      journal={Canadian Journal of Mathematics},
      publisher={Cambridge University Press},
      author={Sears, D. B.},
      year={1950},
      pages={314–325}
}

@article{titchmarsh_1949,
     title={On the {Uniqueness} of the {Green's} {Function} {Associated} with a {Second-order} {Differential} {Equation}},
     volume={1},
     DOI={10.4153/CJM-1949-018-x},
     number={2},
     journal={Canadian Journal of Mathematics},
     publisher={Cambridge University Press},
     author={Titchmarsh, E. C.},
     year={1949},
     pages={191–198}
}

@article{CHERNOFF1973401,
title = {Essential self-adjointness of powers of generators of hyperbolic equations},
journal = {Journal of Functional Analysis},
volume = {12},
number = {4},
pages = {401-414},
year = {1973},
issn = {0022-1236},
doi = {https://doi.org/10.1016/0022-1236(73)90003-7},
url = {https://www.sciencedirect.com/science/article/pii/0022123673900037},
author = {Paul R Chernoff}
}

@article{KATO1973415,
title = {A remark to the preceding paper by {Chernoff}},
journal = {Journal of Functional Analysis},
volume = {12},
number = {4},
pages = {415-417},
year = {1973},
issn = {0022-1236},
doi = {https://doi.org/10.1016/0022-1236(73)90004-9},
url = {https://www.sciencedirect.com/science/article/pii/0022123673900049},
author = {Tosio Kato}
}

@article{Lev61,
title = {On a~theorem of {Titchmarsh} and {Sears}},
journal = {Uspekhi Mat. Nauk},
volume = {16},
number = {4(100)},
pages = {175--178},
year = {1961},
url = {http://mathscinet.ams.org/mathscinet-getitem?mr=132288},
author = {B.~M.~Levitan}
}

@article{chernoff1977schrodinger,
  title={Schr{\"o}dinger and {Dirac} operators with singular potentials and hyperbolic equations},
  author={Paul R Chernoff},
  journal={Pacific Journal of Mathematics},
  volume={72},
  number={2},
  pages={361--382},
  year={1977},
  publisher={Mathematical Sciences Publishers}
}

@article{Kato1972,
title={Schr{\"o}dinger operators with singular potentials},
author = {Tosio Kato},
journal = {Israel Journal of Mathematics},
volume = {13},
number = {1},
pages = {135--148},
url = {https://doi.org/10.1007/BF02760233},
doi = {10.1007/BF02760233},
year = {1972},
issn  = {1565-8511}
}

@article{guneysu2014katoclass,
 ISSN = {00029939, 10886826},
 URL = {http://www.jstor.org/stable/23808318},
 author = {Batu G{\"u}neysu},
 journal = {Proceedings of the American Mathematical Society},
 number = {4},
 pages = {1289--1300},
 publisher = {American Mathematical Society},
 title = {Kato's Inequality and Form Boundedness of {Kato} Potentials on Arbitrary {Riemannian} Manifolds},
 urldate = {2024-10-27},
 volume = {142},
 year = {2014}
}

@article{Mila2023,
author = {Milatovic, Ognjen},
title = {Self-adjointness of non-semibounded covariant {Schrödinger} operators on {Riemannian} manifolds},
journal = {Mathematische Nachrichten},
volume = {296},
number = {9},
pages = {3967-3985},
doi = {https://doi.org/10.1002/mana.202100252},
url = {https://onlinelibrary.wiley.com/doi/abs/10.1002/mana.202100252},
eprint = {https://onlinelibrary.wiley.com/doi/pdf/10.1002/mana.202100252},
year = {2023}
}

@article{Gneysu2015HeatKI,
  title={Heat Kernels in the Context of {Kato} Potentials on Arbitrary Manifolds},
  author={Batu G{\"u}neysu},
  journal={Potential Analysis},
  year={2015},
  volume={46},
  pages={119-134},
  url={https://api.semanticscholar.org/CorpusID:119690583}
}

@article{Bruse2004,
author = {Brusentsev, A. G.},
year = {2004},
month = {10},
pages = {31-61},
title = {Selfadjointness of elliptic differential operators in $L_2(G)$, and correction potentials},
volume = {65},
journal = {Transactions of the Moscow Mathematical Society},
doi = {10.1090/S0077-1554-04-00144-X}
}

@article{Bruse1998,
 author = {Brusentsev, A. G.},
 title = {Near-boundary behavior of the potential of an elliptic operator that ensures its essential selfadjointness},
 fjournal = {Matematicheskaya Fizika, Analiz, Geometriya},
 journal = {Mat. Fiz. Anal. Geom.},
 issn = {1027-1767},
 volume = {5},
 number = {3-4},
 pages = {149--165},
 year = {1998},
 language = {Russian},
 zbMATH = {1525591},
 Zbl = {0951.35034}
}

@article{NENCIU20172619,
title = {Drift-diffusion equations on domains in {${\mathbb{R}}^{d}$}: Essential self-adjointness and stochastic completeness},
journal = {Journal of Functional Analysis},
volume = {273},
number = {8},
pages = {2619-2654},
year = {2017},
issn = {0022-1236},
doi = {https://doi.org/10.1016/j.jfa.2017.06.022},
url = {https://www.sciencedirect.com/science/article/pii/S0022123617302513},
author = {Gheorghe Nenciu and Irina Nenciu}
}

@article{GRUMMT2012480,
title = {Essential selfadjointness of singular magnetic {Schrödinger} operators on {Riemannian} manifolds},
journal = {Journal of Mathematical Analysis and Applications},
volume = {388},
number = {1},
pages = {480-489},
year = {2012},
issn = {0022-247X},
doi = {https://doi.org/10.1016/j.jmaa.2011.09.060},
url = {https://www.sciencedirect.com/science/article/pii/S0022247X11009243},
author = {Robert Grummt and Martin Kolb}
}

@article{Prandi2016QuantumCO,
  title={Quantum confinement on non-complete Riemannian manifolds},
  author={Dario Prandi and Luca Rizzi and Marcello Seri},
  journal={Journal of Spectral Theory},
  year={2016},
  url={https://api.semanticscholar.org/CorpusID:119158133}
}

@article{DEVINATZ197758,
title = {Essential self-adjointness of {Schrödinger-type} operators},
journal = {Journal of Functional Analysis},
volume = {25},
number = {1},
pages = {58-69},
year = {1977},
issn = {0022-1236},
doi = {https://doi.org/10.1016/0022-1236(77)90032-5},
url = {https://www.sciencedirect.com/science/article/pii/0022123677900325},
author = {A Devinatz}
}

@article{HELLWIG1969279,
title = {A criterion for self-adjointness of singular elliptic differential operators},
journal = {Journal of Mathematical Analysis and Applications},
volume = {26},
number = {2},
pages = {279-291},
year = {1969},
issn = {0022-247X},
doi = {https://doi.org/10.1016/0022-247X(69)90151-6},
url = {https://www.sciencedirect.com/science/article/pii/0022247X69901516},
author = {Birgitta Hellwig}
}

@article{Rofe-Beketov1985,
title = {Necessary and sufficient conditions for a finite propagation rate for elliptic operators},
journal = {Ukrainian Mathematical Journal},
volume = {37},
number = {5},
pages = {547-549},
year = {1985},
issn = {1573-9376},
doi = {10.1007/BF01061187},
url = {https://doi.org/10.1007/BF01061187},
author = {Rofe-Beketov, F. S.}
}

@article{NENCIU2008,
author = {Nenciu, Gheorghe and Nenciu, Irina},
year = {2008},
month = {11},
pages = {377-394},
title = {On Confining Potentials and Essential Self-Adjointness for {Schrödinger} Operators on Bounded Domains in {${\mathbb{R}}^n$}},
volume = {10},
journal = {Annales Henri Poincare},
doi = {10.1007/s00023-009-0412-1}
}

@article{Stummel1956/57,
author = {Stummel, F.},
journal = {Mathematische Annalen},
pages = {150-176},
title = {Singuläre elliptische Differentialoperatoren in Hilbertschen Räumen.},
url = {http://eudml.org/doc/160513},
volume = {132},
year = {1956/57},
}

@book{schechter1986spectra,
  title={Spectra of Partial Differential Operators},
  author={Schechter, M.},
  isbn={9780444878229},
  lccn={85020730},
  series={Annals of Discrete Mathematics},
  url={https://books.google.com/books?id=il5pyRt8odMC},
  year={1986},
  publisher={North-Holland}
}

\end{document}